\documentclass[10pt]{amsart}

\usepackage{mathtools}
\usepackage{hyperref}
\usepackage{xcolor}
\usepackage[normalem]{ulem}

\usepackage{accents}

\newcommand\redsout{\bgroup\markoverwith{\textcolor{red}{\rule[0.5ex]{2pt}{0.4pt}}}\ULon}

\newcommand{\E}{\mathbb{E}}
\newcommand{\N}{\mathbb{N}}

\newcommand{\Q}{\mathbb{Q}}
\newcommand{\R}{\mathbb{R}}

\newcommand{\Pb}{\mathbb{P}}
\newcommand{\tp}{t^{\prime}}

\newcommand{\ps}{s^{\prime}}

\newcommand{\ve}{\varepsilon}

\newcommand{\esp}{\epsilon}

\def\={{\;\mathop{=}\limits^{\text{(law)}}\;}}

\newtheorem{theorem}{Theorem}[section]
\newtheorem{prop}[theorem]{Proposition}

\numberwithin{equation}{section}

\usepackage[english]{babel}

\title[\texttt{Existence of strong solutions of fractional Brownian sheet driven SDEs with integrable drift.}]
{Existence of strong solutions of fractional Brownian sheet driven SDEs with integrable drift.
}

\author[A.-M. Bogso]{Antoine-Marie Bogso}
\address{University of Yaounde I,\\
	Faculty of Sciences, Department of Mathematics,\\
	P.O. Box 812, Yaounde, Cameroon\\
	AIMS Ghana, P.O. Box LGDTD 20046, Summerhill Estates, Eat Legon Hills, Santoe, Acrra}
\email{antoine.bogso@facsciences-uy1.cm}

\author[O. Menoukeu-Pamen]{Olivier Menoukeu-Pamen}
\address{Institute for Financial and Actuarial Mathematics (IFAM), \\
	Department of Mathematical Sciences, University of Liverpool, \\
	Liverpool L69 7ZL, UK\\
	AIMS Ghana, P.O. Box LGDTD 20046, Summerhill Estates, East Legon Hills, Santoe, Acrra}
\email{menoukeu@liverpool.ac.uk}

\author[F. Proske]{Frank Proske}
\address{University of Oslo \\
	CMA, Department of Mathematical Sciences, University of Oslo \\
	Moltke Moes Vei 35, Po Box 1053, Blindern, 0316 Oslo, Norway}
\email{proske@math.uio.no}

\newtheorem{thm}{Theorem}
\newtheorem{lem}[thm]{Lemma}%
\newtheorem{cor}[thm]{Corollary}%

\newtheorem{rem}{Remark}%

\begin{document}

\title[Fractional Brownian sheet driven SDEs with integrable drift.]{Strong solutions of fractional Brownian sheet driven SDEs with integrable drift.}


\begin{abstract}
	We prove the existence of a unique  Malliavin  differentiable strong solution to a stochastic differential equation on the plane with merely integrable coefficients driven by the fractional Brownian sheet with Hurst parameters less than 1/2. The proof of this result relies on a compactness criterion for square integrable Wiener functionals from Malliavin calculus ([Da Prato, Malliavin and Nualart, 1992]), variational techniques developed in the case of fractional Brownian motion ([Ba\~nos, Nielssen, and Proske, 2020]) and the concept of sectorial local nondeterminism (introduced in [Khoshnevisan and Xiao, 2007]). The latter concept enable us to improve the bound of the Hurst parameter (compare with [Ba\~nos, Nielssen, and Proske, 2020]).	 
\end{abstract}

\maketitle

\section{Introduction}
\label{intro}
In this paper we aim at studying solutions $X^{x}=(X_{s,t}^{x},(s,t)\in
\lbrack 0,T]^{2})$ of the stochastic differential equation (SDE) on the
plane 
\begin{equation} 
	X_{s,t}^{}=x+\int_{0}^{s}\int_{0}^{t}b(s_{1},t_{1},X_{s_{1},t_{1}}^{})\,%
	\mathrm{d}t_{1}\mathrm{d}s_{1}+W_{s,t}^{H},\quad 0\leq s,t\leq T,\,x\in 
	\mathbb{R}^{d},  \label{MainPb}
\end{equation}%
where $W^{H}=(W_{s,t}^{H},(s,t)\in \lbrack 0,T]^{2})$ is a $d$-dimensional
fractional Brownian sheet with Hurst index $H=(H_{1},H_{2})\in (0,\frac{1}{2}%
)^{2}$ and where $b:\,[0,T]^{2}\times \mathbb{R}^{d}\rightarrow \mathbb{R}%
^{d}$ is a Borel-measurable function satisfying some conditions that we will
specified later. We recall that $W^{H}$ coincides with the Wiener sheet,
when $H=(\frac{1}{2},\frac{1}{2})$.

The SDE \ref{MainPb} is a hyperbolic stochastic partial differential
equation (HPSDE), whose corresponding differential version takes the
following form for $X(s,t)=X_{s,t}^{x}$ and $W^{H}(s,t)=W_{s,t}^{H}$
\begin{align*}
	\left\{
	\begin{array}{ll}
		\frac{\partial ^{2}}{\partial s\partial t}X(s,t)=b(s,t,X(s,t))+\frac{
			\partial^{2}}{\partial s\partial t}W^{H}(s,t), & (s,t)\in (0,T]^{2},\\& \\X(s,0)=X(0,t)=x, & s,t\in [0,T],
	\end{array}
	\right.
\end{align*}
where $\frac{\partial ^{2}}{\partial s\partial t}W^{H}(s,t)$ is the white
noise of the fractional Brownian sheet. Interestingly, if $d=1$, the latter
SPDE can be transformed into the following stochastic wave equation with a
non-linear forcing term by using a formal $\frac{\pi }{2}$ rotation:%
\begin{equation}
	\frac{\partial ^{2}}{\partial t^{2}}X(t,x)-\frac{\partial ^{2}}{\partial
		x^{2}}X(t,x)=\frac{\partial ^{2}}{\partial t\partial x}%
	B^{H}(t,x)+b(t,x,X(t,x)).  \label{WaveEq}
\end{equation}%
for another fractional Brownian sheet $B^{H}(t,x)$. See e.g. \cite{Wa84} and \cite{FaNu93} for
further details in the case of a Brownian sheet.

The objective of this paper is to construct a unique global strong
solution to the HSPDE (\ref{MainPb}) in the case of a merely {
	integrable} vector field%
{
	\begin{equation}
		b\in L^1_{\infty}:=L^{\infty}([0,T]^2;L^{1 }(\R^d;\mathbb{R}^{d})),
		\label{CoeffClass}
	\end{equation}%
	when $H_1+H_2<\frac{1}{2(d+1)}$.}  {A strong solution here refers to a solution that is a progressively measurable functional of the driving noise $W^{H}$. We point out that this result extends the recent findings of  \cite{BDMP22} from the case of a
	Wiener sheet to the setting of non-Markov fields, when $W^{H}$ is the
	initial noise. Furthermore, it generalises the work in \cite{BNP18}, which dealt with additive noise driven by fractional Brownian motion with Hurst parameter $H<\frac{1}{2}$ to the case of a fractional Brownian sheet.}

{Our motivation for analysing HSPDEs of the type \eqref{MainPb}, inspired by the relatively new field of stochastic regularization theory for (a priori ill-posed) ordinary and partial differential equations (see, e.g., \cite{Gess} and \cite{Fla10} for an overview), is twofold:}

{	1. In \cite{MMNPZ13,MoNP15} it was for the first time observed that
	strong solutions to SDEs driven by a $L^{\infty }-$drift vector field and an
	additive Wiener noise are (locally) Sobolev differentiable with respect to
	the initial value and Malliavin differentiable (see also \cite{BFGM14}). Later, the authors in \cite{BNP18}
	noted that additive noise given by a fractional Brownian motion with
	a sufficiently small Hurst parameter yields even higher regularity in the
	associated stochastic flow of the corresponding singular equation compare to 
	\cite{BaNi16}, in terms of higher order (local) Sobolev
	differentiability. Further improvements in this direction was
	achieved in \cite{ABP17a}, where the authors obtained $C^{\infty }-$%
	flows in the case of certain Gaussian driving noise with non-H\"{o}lder
	continuous paths. As in \cite{ABP17a} we may (more generally)
	consider a two-parameter noise process $\mathbb{W%
	}$  in the HSPDE \eqref{MainPb}, defined as}%
\begin{equation}
	\mathbb{W}_{s,t}=\sum_{i\geq 1}\lambda _{i}W_{s,t}^{H_{i}},\leq s,t\leq T%
	\text{,}  \label{W}
\end{equation}%
where $W^{H_{i}},i\geq 1$ is a sequence of a independent fractional Brownian
sheets with Hurst indices $H_{i}=(H_{1,i},H_{2,i})\in (0,\frac{1}{2})^{2}$ \
and a sequence $(\lambda _{i})_{i\geq 1}\subset \mathbb{R}$, which e.g.
insures $L^{2}$-convergence of the sum in \eqref{W} and satisfies certain
conditions (compare \cite{ABP17a}).  {Similarly to \cite%
	{ABP17a}, we may also expect in this case a regularization by
	noise effect of singular hyperbolic PDEs, which gives rise to solutions that are smooth with respect to the initial condition. Hence, in
	particular, if $d=1$, such a result would imply the existence (and
	uniqueness) of a smooth solution to the stochastic wave equation \eqref%
	{WaveEq}, where $W^{H}$ is replaced by $\mathbb{W}$.}

2. Another motivation for this paper comes from the restoration of
well-posedness of (singular) hyperbolic PDEs of the type%
\begin{align*}
	\frac{\partial ^{2}}{\partial s\partial t}u(s,t)=b(s,t,u(s,t))
\end{align*}%
in the context of \emph{path-by-path} concept of solutions, originally introduced in Davie's groundbreaking work \cite{Da07} for Wiener noise-driven SDEs. Applied to HSPDE of the form \eqref{MainPb}, this concept means that there exists a measurable set $\Omega
^{\ast }$ (depending on the initial value $x$) of probability mass $1$ such
that for all $\omega \in \Omega ^{\ast }$, there exists a unique \emph{%
	deterministic} function $Y=Y(\omega )\in C([0,T]^{2};\mathbb{R}^{d})$ that
satisfies \eqref{MainPb}. Based on a local-time-space decomposition and
techniques in \cite{Da07}, this type of solutions was first studied in \cite%
{BDM21a}  and \cite{BMP22} in the case of
Brownian sheet and bounded variation drift vector fields. See also
Bechthold,\ Harang, Rana \cite{BHR23} for the case of fractional Brownian sheet and distributional vector fields in the framework
of Besov spaces $B_{p,q}^{\alpha }(\mathbb{R}^{d})$, where solutions are described using equations in the nonlinear Young sense. In the this case, it is unclear whether solutions to \eqref{MainPb} with respect to a classical Lebesgue integral and a discontinuous vector field can be interpreted as solutions of equations of the non-linear Young type (see e.g. \cite[Remark 4.5.(i)]{ALL23}). Regarding the $1-$%
parameter version of \eqref{MainPb}, we also mention the work \cite{AMP23}, where the authors establish path-by-path uniqueness
of solutions in Davie's sense for SDEs with fractional Brownian motion
additive noise, when $b$ is essentially bounded and integrable and $H<\frac{1%
}{2}$. See also \cite{CG16} for the case of continuous vector
fields, an approach later generalized by \cite{BHR23} to the $%
2-$parameter setting by applying techniques based on a multiparameter sewing
Lemma, the concept of local time and properties of an averaging operator. Compare also with \cite{Khoa}.
The result in \cite{AMP23} relies on techniques developed in %
\cite{Sh16}, combined with the aforementioned regularizing properties
of fractional Brownian motion (with a sufficiently small Hurst
parameter). The latter in connection with 1., suggests that such a
result should be also possible for the HSPDE (\ref{MainPb}).
Therefore, the results obtained in this paper can be considered a first step towards
the goals formulated in 1. and 2., which we persue in forthcoming articles.

{Finally, we would like to mention some related results in the literature concerning the analysis of solutions to \eqref{MainPb}. The
	first results on strong existence and pathwise uniqueness of solutions to
	SDEs on the plane with multiplicative Wiener noise date back to \cite%
	{Ca72} and \cite{Ye81}, where the authors assume that the
	driving vector fields are Lipschitz continuous and of linear growth. We also
	refer to \cite{Ye87} for strong solutions in the case a given
	deterministic boundary process, and \cite{Ye85} regarding the construction
	of weak solutions under the conditions of a continuous drift vector field with
	linear growth. Furthermore,  \cite{NuaSan85} investigates the
	regularity of solutions of Wiener sheet-driven SDEs in the sense of
	Malliavin differentiability, requiring sufficient smoothness of the
	coefficients. 
	See also \cite{QuTi07} for the case of
	a stochastic wave equation driven by fractional Brownian sheet.}

{As for SDEs on the plane with non-Lipschitz continuous vector fields,
	results on existence and uniqueness in the Wiener sheet case can be found in Nualart, Tindel \cite{NuTi97}. In that work, the authors use
	a comparison theorem to prove these results, under growth and
	monotonicity conditions with respect the (one-dimensional) drift coefficient.
	We also highlight the work \cite%
	{ENO03}, which extends the results of \cite%
	{NuTi97} to the case of a fractional Brownian motion sheet with Hurst
	indices $H_{1},H_{2}\leq \frac{1}{2}$. In this context, for results on the path-by-path uniqueness in the sense of Davie with
	respect to solutions to singular SDEs on the plane, we also recall the
	previously mentioned work  \cite%
	{BDM21a} in the Wiener sheet case (see also   \cite{BMP22}).} 


{In this paper, as mentioned earlier, we aim to prove the existence and uniqueness of Malliavin differentiable strong solutions to the HSPDE \eqref%
	{MainPb} for bounded and integrable vector fields $%
	b$ as defined in \ref{CoeffClass}. This result generalises the corresponding result in \cite{BDMP22} for the
	Wiener sheet to the case of $W^{H}$, where $H\in \left( 0,\frac{1}{2}\right) ^{2}$. The proof relies on a compactness criterion for square-integrable functionals of a
	Wiener sheet from Malliavin calculus (see e.g. \cite{BDMP22} or \cite%
	{BNP18}, \cite{MMNPZ13} in the one-parameter case, or \cite{BBMBP23} in
	the Hilbert space setting),  variational
	techniques developed in \cite{BNP18} in the case of a
	fractional Brownian motion and  {the concept of sectorial local
		nondeterminism.}}

Let us stress that the sectorial local
nondeterminism was first introduced by \cite{KhXi07} for the Brownian sheet and extended to fractional Brownian sheets in \cite{WX07}. In \cite{KhXi07}, the authors applied sectorial local nondeterminism to study distributional properties of the level sets and the continuity of the local times of the Brownian sheet. A continuation of this work by \cite{KXW06} focuses on the geometry of the Brownian sheet's surface. In \cite{WX07}, the sectorial local nondeterminism is used to describe the geometric and Fourier analytic properties of fractional Brownian sheets. It is well known that  Brownian sheets are not locally nondeterministic, unlike fractional Brownian motions. As a result, we cannot obtain similar estimates as in \cite{BNP18}. However, sectorial local nondeterminism provides alternative estimates that also allow to prove the existence of a unique strong solution to \eqref{MainPb} when  {$H_1+H_2<\frac{1}{2(d+1)}$. If, in addition, $b$ is bounded, then \eqref{MainPb} has a unique strong solution provided $\max\{H_1,H_2\}<\frac{1}{2(d+1)}$.}
We also note that the methods from related results mentioned above are not applicable to vector fields $b$ in \eqref%
{CoeffClass} for fractional Brownian sheet.  {Our strategy consists of four steps. First, we establish the weak existence  and the joint uniqueness in law for \eqref{MainPb}. The next two steps demonstrate that every weak solution is a strong solution. Precisely, we show that for every $(s,t)\in[0,T]^2$, $X^x_{s,t}$ is measurable with respect to the $\sigma$-algebra $\mathcal{F}^H_{s,t}$ generated by the random variables $\{W^H_{s_1,t_1};s_1\leq s,t_1\leq t\}$. Since $X^x_{s,0}=x=X^x_{0,t}$ is $\mathcal{F}_{0,0}$-measurable for every $s,\,t\in[0,T]$, it suffices to prove that $X^x_{s,t}$ is $\mathcal{F}^H_{s,t}$-measurable for all $(s,t)\in(0,T]^2$. In the second step, we approximate the drift $b$ by a sequence of smooth functions with compact support and we show that the sequence of the corresponding solutions converges weakly to a limit, which is $(\mathcal{F}^H_{s,t})$-adapted. In the third step, we apply a compactness criterion from Malliavin calculus to obtain a relative compactness property on the  sequence constructed in the second step.
	In the final step, we	
	show that $(X^x_{s,t})$ is a strong solution and apply a dual Yamada-Watanabe argument to deduce pathwise uniqueness.  }

Our paper is structured as follows: In Section \ref{defcon0}, we discuss the basic mathematical framework for this article. In Section \ref{secbasest}, we establish some key estimates based on the property of sectorial (strong) local nondeterminism of random fields (see, for example, Theorem \ref{RegulariseTheo}). Finally, in Section 4,
we apply these estimates to prove our main result (Theorem \ref{TheoEqmainPbS4}) on the strong uniqueness of solutions for \eqref{MainPb}.

\section{Preliminaries}\label{defcon0}

\subsection{Fractional calculus} 
Here are some basic definitions and results on two-parameter calculus that can be found in \cite{ENO03} and \cite{TuTu03}. We refer to \cite{SKM93} and \cite{Li01} for an exhaustive survey on classical one-parameter fractional calculus. 

Let $f\in L^p([0,T]^2)$ with $T>0$ and $p\geq1$. For $\alpha>0$, $\beta>0$, the $(\alpha,\beta)$-order left-sided Riemann-Liouville fractional integral of $f$ on $(0,T)^2$ is defined by
\begin{align}\label{EqFracInt}
	I^{\alpha,\beta}f(x,y)=\frac{1}{\Gamma(\alpha)\Gamma(\beta)}\int_0^x\int_0^y(x-u)^{\alpha-1}(y-v)^{\beta-1}f(u,v)\,\mathrm{d}u\mathrm{d}v,
\end{align}
for almost all $(x,y)\in[0,T]^2$, where $\Gamma$ is the Gamma function. Observe that \eqref{EqFracInt} rewrites
\begin{align*}
	I^{\alpha,\beta}f(x,y)=I^{\alpha}(I^{\beta}f(\cdot,y))(x),
\end{align*}
where $I^{\alpha}$ (respectively $I^{\beta}$) is the $\alpha$-order (respectively $\beta$-order) left-sided Riemann-Liouville fractional integral on $(0,T)$. The integral $I^{\alpha,\beta}$ can be iterated as it is shown by the following first composition formula.
\begin{align*}
	I^{\alpha^{\prime},\beta^{\prime}}(I^{\alpha,\beta}f)=I^{\alpha^{\prime}+\alpha,\beta^{\prime}+\beta}f.
\end{align*}
{The fractional derivative can be defined} as the reverse operation of the fractional integral. Assume in addition $\alpha<1$, $\beta<1$ and denote by $I^{\alpha,\beta}(L^p)$ the image of $L^p([0,T]^2)$ by the operator $I^{\alpha,\beta}$. If $f\in I^{\alpha,\beta}(L^p)$, then the function $\phi$ such that $f=I^{\alpha,\beta}\phi$ is unique in $L^p([0,T]^2)$ and is equal to the $(\alpha,\beta)$-order left-sided Riemann-Liouville derivative of $f$ given by
\begin{align*}
	D^{\alpha,\beta}f(x,y)=&\frac{\partial^2I^{1-\alpha,1-\beta}f}{\partial x\partial y}(x,y)\\
	=&\frac{1}{\Gamma(1-\alpha)\Gamma(1-\beta)}\frac{\partial^2}{\partial x\partial y}\int_0^x\int_0^y\frac{f(u,v)}{(x-u)^{\alpha}(y-v)^{\beta}}\,\mathrm{d}u\mathrm{d}v.
\end{align*}
If $\min\{\alpha p,\beta p\}>1$, then any function in $I^{\alpha,\beta}(L^p)$ is $(\alpha-1/p,\beta-1/p)$-H\"older continuous. Moreover, any H\"older continuous function of order $(\gamma,\delta)$
with $\gamma>\alpha$ and $\delta>\beta$ has a fractional derivative of order $(\alpha,\beta)$. The left-sided derivative of $f$ defined above has the following Weil representation
\begin{align*}
	D^{\alpha,\beta}f(x,y)=&\frac{\mathbf{1}_{(0,T)^2}(x,y)}{\Gamma(1-\alpha)\Gamma(1-\beta)}\Big\{\frac{f(x,y)}{x^{\alpha}y^{\beta}}+\frac{\alpha}{y^{\beta}}\int_0^x\frac{f(x,y)-f(u,y)}{(x-u)^{\alpha+1}}\mathrm{d}u\\& \qquad+\frac{\beta}{x^{\alpha}}\int_0^y\frac{f(x,y)-f(x,v)}{(y-v)^{\beta+1}}\mathrm{d}v\\& 
	\qquad+\alpha\beta\int_0^x\int_0^y\frac{f(x,y)-f(u,y)-f(x,v)+f(u,v)}{(x-u)^{\alpha+1}(y-v)^{\beta+1}}\mathrm{d}u\mathrm{d}v
	\Big\},
\end{align*}
where the convergence of the integrals at singularity $x=u$ or $y=v$ holds in $L^p$ sense.

By construction, we have
{
	\begin{align*}
		I^{\alpha,\beta}(D^{\alpha,\beta}f)=f,\,\forall\,f\in I^{\alpha,\beta}(L^p)\,\text{ and }\,
		D^{\alpha,\beta}(I^{\alpha,\beta}f)=f,\,\forall\,f\in L^1([0,T]^2).
\end{align*}}
Finally, if $f\in I^{\alpha+\alpha^{\prime},\beta+\beta^{\prime}}(L^1)$, with $\max\{\alpha+\beta,\alpha^{\prime}+\beta^{\prime}\}\leq1$, we also have the composition formula 
\begin{align*}
	D^{\alpha,\beta}(D^{\alpha^{\prime},\beta^{\prime}}f)=D^{\alpha+\alpha^{\prime},\beta+\beta^{\prime}}f.
\end{align*}

\subsection{Fractional Brownian sheet} Let $W^H=(W^H_{s,t},(s,t)\in[0,T]^2)$ be a $d$-dimensional fractional Brownian sheet with Hurst index $H=(H_1,H_2)\in(0,1/2)^2$ defined on a probability space $(\Omega,\mathcal{F},\Pb)$.  {More} precisely, $W^H$ is a centered Gaussian process with covariance matrix
\begin{align*}
	&R_H(s_1,t_1,s_2,t_2)_{i,j}:=\E\left[W^{H,(i)}_{s_1,t_1}W^{H,(j)}_{s_2,t_2}\right]\\=&\frac{1}{4}\delta_{ij}(s_2^{2H_1}+s_1^{2H_1}-\vert s_2-s_1\vert^{2H_1})(t_2^{2H_2}+t_1^{2H_2}-\vert t_2-t_1\vert^{2H_2}),\, i,j=1,\ldots,d,
\end{align*}
where $\delta_{ij}=1$ if $i=j$ and $0$ if $i\neq j$. 

Let $\mathcal{E}$ be the set of step functions on $[0,T]^2$ and denote by $\mathcal{H}$ the Hilbert space defined as the closure of $\mathcal{E}$ with respect to the inner product
{
	\begin{align}\label{FracInnerProd}
		\langle\mathbf{1}_{[0,s_1]\times[0,t_1]},\mathbf{1}_{[0,s_2]\times[0,t_2]}\rangle_{\mathcal{H}}=R_H(s_1,t_1,s_2,t_2).
\end{align} }
The mapping $\mathbf{1}_{[0,s]\times[0,t]}\mapsto W^H_{s,t}$ can be extended to an isometry between $\mathcal{H}$ and the Gaussian subspace of $L^2(\Omega)$ associated with $W^H$. Denote by $\varphi\mapsto W^H(\varphi)$ such isometry. The following representation of $R_H$ holds  {for} $H\in(0,1/2)^2$ (see e.g. \cite{ENO03} and \cite[Proposition 5.1.3]{Nua06}):
\begin{prop}\label{prop:FracKernel}
	Let $H=(H_1,H_2)\in(0,\frac{1}{2})^2$, $\alpha\in(0,\frac{1}{2})$ and $F_{\alpha}:\,(0,\infty)^2\to\R$ be the function defined by
	\begin{align*}
		F_{\alpha}(r,s):=\Big(\frac{s}{r}\Big)^{\alpha-1/2}(s-r)^{\alpha-1/2}\mathbf{1}_{[0,s]}(r)\quad\text{for all }r,\,s.
	\end{align*}
	{Let $K_H$ be the kernel}
	\begin{align*}
		K_H(s_1,t_1,s_2,t_2)= K_{H_1}(s_1,s_2)K_{H_2}(t_1,t_2),
	\end{align*}
	where $0<s_1<s_2$, $0<t_1<t_2$,
	\begin{align*}
		K_{H_1}(s_1,s_2)=c_{H_1}\Big[F_{H_1}(s_1,s_2)+\Big(\frac{1}{2}-H_1\Big)\int_{s_1}^{s_2}\frac{F_{H_1}(s_1,u)}{u}\,\mathrm{d}u\Big],
	\end{align*}
	\begin{align*}
		K_{H_2}(t_1,t_2)=c_{H_2}\Big[F_{H_2}(t_1,t_2)+\Big(\frac{1}{2}-H_2\Big)\int_{t_1}^{t_2}\frac{F_{H_2}(t_1,u)}{u}\,\mathrm{d}u\Big]
	\end{align*}
	and
	\begin{align*}
		c_{H_i}=\sqrt{\frac{2H_i}{(1-2H_i)\beta(1-2H_i,H_i+1/2)}},\quad i=1,2
	\end{align*}
	{with $\beta$, the Beta function. Then}
	\begin{align*}
		R_H(s_1,t_1,s_2,t_2)=\int_0^{s_1\wedge s_2}\int_0^{t_1\wedge t_2}K_H(u,v,s_1,t_1)K_H(u,v,s_2,t_2)\,\mathrm{d}u\mathrm{d}v.
	\end{align*}
\end{prop}
Consider the linear operator $K^{\ast}_H:\,\mathcal{E}\to L^2([0,T]^2)$ given by
\begin{align*}
	(K^{\ast}_H\varphi)(s,t)&=K_H(T,T,s,t)\varphi(s,t)\\&\quad+\int_s^T\int_t^T(\varphi(u,v)-\varphi(s,t))\frac{\partial^2K_H}{\partial u\partial v}(u,v,s,t)\,\mathrm{d}u\mathrm{d}v
\end{align*}
for every $\varphi\in\mathcal{E}$. One can easily see that $$(K^{\ast}_H\mathbf{1}_{[0,s]\times[0,t]})(s_1,t_1)=K_H(s,s_1,t,t_1)\mathbf{1}_{[0,s]\times[0,t]}(s_1,t_1).$$
Hence, for any step functions $\varphi$, $\psi$ in $\mathcal{E}$, we have  {from \eqref{FracInnerProd} and Proposition \ref{prop:FracKernel}}
\begin{align*}
	\langle K^{\ast}_H\varphi,K^{\ast}_H\psi\rangle_{L^2([0,T]^2)}=\langle\varphi,\psi\rangle_{\mathcal{H}}.
\end{align*}
As a consequence, the operator $K^{\ast}_H$ is an isometry between $\mathcal{E}$ and $L^2([0,T]^2)$ that can be extended to the Hilbert space $\mathcal{H}$.	Then the $d$-dimensional process $W=(W_{s,t},(s,t)\in[0,T]^2)$ defined by 
\begin{align}\label{Bsheet}
	W_{s,t}=W^H(({K_H^{\ast}})^{-1}(\mathbf{1}_{[0,s]\times[0,t]}))
\end{align}	
is a $d$-dimensional Brownian sheet, and the process $W^H$ has the following representation
\begin{align}\label{FBsheet}
	W^H_{s,t}=\int_0^s\int_0^tK_H(s,t,s_1,t_1)\,\mathrm{d}W_{s_1,t_1}.
\end{align}	
The proof of similar results in the one parameter case can be found in \cite{AMN01}.

{For every $(s,t)\in[0,T]^2$, we denote by $\mathcal{F}_{s,t}$ the $\sigma$-algebra generated by the random variables $\{W_{u,v};\,u\leq s,v\leq t\}$, where $W$ is defined by \eqref{Bsheet}. It follows from \eqref{Bsheet} and \eqref{FBsheet} that $\mathcal{F}_{s,t}$ coincide with the $\sigma$-algebra $\mathcal{F}^H_{s,t}$ generated by the random variables $\{W^H_{u,v};\,u\leq s,v\leq t\}$.
	We equip the probability space $(\Omega,\mathcal{F},\Pb)$ with the filtration $(\mathcal{F}^H_{s,t},(s,t)\in[0,T]^2)$  augmented by all $\Pb$-null sets.} Next, we give a multidimensional version of Girsanov's theorem for fractional Brownian sheet provided by \cite[Theorem 3]{ENO03}. First we need to introduce the isomorphism $K_H$ from $L^2([0,T]^2)$ onto $I^{H_1+\frac{1}{2},H_2+\frac{1}{2}}(L^2)$ associated with the kernel $K_H(s,t,s_1,t_1)$ and defined as
\begin{align}\label{GirsanovTheo1}
	(K_H\varphi)(s,t)=I^{2H_1,2H_2}s^{\frac{1}{2}-H_1}t^{\frac{1}{2}-H_2}I^{\frac{1}{2}-H_1,\frac{1}{2}-H_2}s^{H_1-\frac{1}{2}}t^{H_2-\frac{1}{2}}\varphi
\end{align}
for any $\varphi\in L^2([0,T]^2)$.

Now, given a $d$-dimensional process $(\theta_{s,t},(s,t)\in[0,T]^2)$ with integrable trajectories, we consider the transformation
\begin{align*}
	\widetilde{W}^H_{s,t}=W^H_{s,t}+\int_{0}^s\int_0^t\theta_{u,v}\,\mathrm{d}u\mathrm{d}v.
\end{align*} 
Observe that
\begin{align*}
	\widetilde{W}^H_{s,t}=\int_0^s\int_0^tK_H(u,v)\,\mathrm{d}W_{u,v}+\int_{0}^s\int_0^t\theta_{u,v}\,\mathrm{d}u\mathrm{d}v=\int_0^s\int_0^tK_H(u,v)\,\mathrm{d}\widetilde{W}_{u,v},
\end{align*}
where
\begin{align*}
	\widetilde{W}_{s,t}=W_{s,t}+\int_0^s\int_0^tK_H^{-1}\Big(\int_0^{\cdot}\int_0^{\cdot}\theta_{\xi,\eta}\,\mathrm{d}\xi\mathrm{d}\eta\Big)(u,v)\,\mathrm{d}u\mathrm{d}v.
\end{align*}
Moreover, $K_H^{-1}\Big(\int_0^{\cdot}\int_0^{\cdot}\theta^{(i)}_{\xi,\eta}\,\mathrm{d}\xi\mathrm{d}\eta\Big)\in L^2([0,T]^2)$  almost surely if, and only if $\int_0^{\cdot}\int_0^{\cdot}\theta^{(i)}_{\xi,\eta}\,\mathrm{d}\xi\mathrm{d}\eta\in I^{H_1+1/2,H_2+1/2}(L^2)$ for all $i=1,\ldots,d$.\\
It follows from  \eqref{GirsanovTheo1} that the inverse operator is given by
\begin{align}\label{GirsanovTheo2}
	(K^{-1}_H\varphi)(s,t)=s^{\frac{1}{2}-H_1}t^{\frac{1}{2}-H_2}D^{1/2-H_1,1/2-H_2}s^{H_1-\frac{1}{2}}t^{H_2-\frac{1}{2}}D^{2H_1,2H_2}\varphi
\end{align}
for any $\varphi\in I^{H_1+\frac{1}{2},H_2+\frac{1}{2}}(L^2([0,T]^2))$. In particular, if $\varphi$ vanishes on the axes and is absolutely continuous, it can be proved that
\begin{align}\label{GirsanovTheo3}
	(K^{-1}_H\varphi)(s,t)=s^{H_1-\frac{1}{2}}t^{H_2-\frac{1}{2}}I^{\frac{1}{2}-H_1,\frac{1}{2}-H_2}s^{\frac{1}{2}-H_1}t^{\frac{1}{2}-H_2} \frac{\partial^2\varphi}{\partial s\partial t}.
\end{align}
\\
Here is a version of Girsanov's theorem for $d$-dimensional fractional Brownian sheet.  {Its proof is a consequence of Girsanov's theorem for $d$-dimensional Brownian sheet (see e.g. \cite[Proof of Theorem 3]{ENO03} for the one-dimensional case).}
\begin{thm}\label{GirsanovTheo}
	Let $\theta=(\theta_{s,t},(s,t)\in[0,T]^2)$ be a $(\mathcal{F}_{s,t})$-adapted $d$-dimensional process with integrable trajectories and define $\widetilde{W}^H_{s,t}=W^H_{s,t}+\int_0^s\int_0^t\theta_{u,v}\,\mathrm{d}u\mathrm{d}v$ for $(s,t)\in[0,T]^2$. Suppose
	\begin{enumerate}
		\item[(i)]$\int_0^{\cdot}\int_0^{\cdot}\theta^{(i)}_{u,v}\,\mathrm{d}u\mathrm{d}v\in I^{H+1/2,H+1/2}(L^2)$, $\Pb$-a.s., for all $i=1,\ldots,d$.
		\item [(ii)]$\E[\mathcal{Z}_{T,T}]=1$, where
		\begin{align*}
			\mathcal{Z}_{T,T}:=&\exp\Big\{\int_0^T\int_0^TK_H^{-1}\Big(\int_0^{\cdot}\int_0^{\cdot}\theta^{}_{\xi,\eta}\,\mathrm{d}\xi\mathrm{d}\eta\Big)(u,v)\cdot\mathrm{d}W_{u,v}\\&\qquad\qquad-\frac{1}{2}\int_0^T\int_0^T\Big\vert K_H^{-1}\Big(\int_0^{\cdot}\int_0^{\cdot}\theta^{}_{\xi,\eta}\,\mathrm{d}\xi\mathrm{d}\eta\Big)(u,v)\Big\vert^2\,\mathrm{d}u\mathrm{d}v\Big\}.
		\end{align*}
	\end{enumerate}
	Then the shifted process $\widetilde{W}^H$ is a $d$-dimensional  $(\mathcal{F}_{s,t})$-adapted fractional Brownian sheet with Hurst index $H$ under the new probability $\widetilde{\Pb}$ defined by $\frac{\mathrm{d}\widetilde{\Pb}}{\mathrm{d}\Pb}=\mathcal{Z}_{T,T}$. 
\end{thm}

We end this section by recalling the property of sectorial local nondeterminism of the real-valued fractional Brownian sheet due to Wu and Xiao \cite{WX07} which plays a crucial role in this paper. Here is a special case of Theorem 1 in \cite{WX07}.
\begin{thm}\label{SLNDPropTheo}
	Let $W^{0,H}=(W^{0,H}_{s,t},(s,t)\in[0,T]^2)$ be a real-valued fractional Brownian sheet with Hurst parameter $H=(H_1,H_2)$.
	For any fixed number $\esp\in(0,1)$, there exists a positive constant $C_{\esp,H}$ depending on $\ve$ and $H$ only such that for all positive integer $n\geq1$, all $(s,t)$, $(s_1,t_1),\,\ldots,\,(s_n,t_n)\in[\esp,T]^2$, we have
	\begin{align*}
			\text{Var}\left[W^{0,H}_{s,t}\vert W^{0,H}_{s_1,t_1},\,\ldots,\,W^{0,H}_{s_n,t_n}\right] \geq C_{\esp,H}\left(\min\limits_{0\leq k\leq n}\vert s-s_k\vert^{2H_1}+\min\limits_{0\leq k\leq n}\vert t-t_k\vert^{2H_2}\right),
	\end{align*}
	where $s_0=0=t_0$.
\end{thm}

\section{Regularising properties of the fractional Brownian sheet}\label{secbasest}
   
\subsection{An integration by parts formula } 
Let $(m,p)\in\N^2$. Let $g:\,[0,T]^{2mp}\to\R$ be  a positive and bounded Borel measurable function  and let $f:\,[0,T]^{2mp}\times(\R^d)^{mp}\to\R$ be a function  of the form
\begin{align}\label{RegulariseDef1}
	f(s,t,z)=\prod\limits_{j=1}^{mp}f_j(s_j,t_j,z_j),
\end{align}
$s=(s_1,\ldots,s_{mp})\in[0,T]^{mp}$, $t=(t_1,\ldots,t_{mp})\in[0,T]^{mp}$, $z=(z_1,\ldots,z_{mp})\in(\R^d)^{mp}$, where $f_j:\,[0,T]^2\times\R^d\to\R$, $j=1,\ldots,mp$ are smooth functions with compact support.\\ 
Consider an integrable function $\kappa:\,[0,T]^{2mp}\to\R$ of the form
\begin{align}\label{RegulariseDef2}
	\kappa(s,t)=\prod\limits_{j=1}^{mp}\kappa_j(s_j,t_j),\quad(s,t)\in[0,T]^{2mp},
\end{align}
where $\kappa_j:\,[0,T]^2\to\R$, $j=1,\ldots,m$ are integrable functions.

Denote by $\alpha_j$ a multi-index and $D^{\alpha_j}$ its corresponding differential operator. For $\alpha=(\alpha_1,\ldots,\alpha_{mp})\in\N_0^{d\times {mp}}$, we denote by $D^{\alpha}f$ the function given by
\begin{align*}
	D^{\alpha}f(s,t,z)=\prod\limits_{j=1}^{mp}D^{\alpha_j}f_j(s_j,t_j,z_j).
\end{align*}

We endow $\R_+^2$ with the partial  {order} ``$\preceq$" (respectively ``$\prec$") defined by
$$
(u,v)\preceq(u^{\prime},v^{\prime})\text{ when }u\leq u^{\prime}\text{ and }v\leq v^{\prime},
$$
respectively
$$
(u,v)\prec(u^{\prime},v^{\prime})\text{ when }u< u^{\prime}\text{ and }v< v^{\prime},
$$
Let $\N_{2mp}:=\{1,2,\ldots,2mp\}$,  {$\mathcal{P}_{2mp}$} be the set of permutations on $\N_{2mp}$ and $\widehat{\mathcal{P}}_{m\vert 2p}$ be the subset of $\mathcal{P}_{2mp}$ defined as
\begin{align*}
	\widehat{\mathcal{P}}_{m\vert 2p} =\Big\{\sigma\in\mathcal{P}^{}_{2mp}:\,\sigma(1+im)<\ldots<\sigma((1+i)m),\text{ for all }i\in\{0,1,\ldots,2p-1\}\Big\}.
\end{align*}
For any $0\leq \bar r<\bar s\leq T$ and $\sigma\in\mathcal{P}_{2mp}$,   we	define  
\begin{align*}
	\nabla^{m}_{\bar r, \bar s					}=\{(s_1,\ldots,s_m)\in[0,T]^{m}:\,\bar r<s_{m}<\ldots<s_{1}<\bar s\}.
\end{align*}
and
\begin{align*}
\nabla^{2mp,\sigma}_{\bar r,\bar s}=\{(s_1,\ldots,s_{2mp})\in[0,T]^{2mp}:\,(s_{\sigma^{-1}(1)},\ldots,s_{\sigma^{-1}(2mp)})\in\nabla^{2mp}_{\bar r,\bar s}\}.
\end{align*}
 $(\esp,\esp)\prec(\bar r,\bar u)\prec(\bar s,\bar t)$, $z\in (\R^d)^{mp}$, $\alpha\in\N_0^{d\times mp}$,     {define}
\begin{align}\label{DefIntPart1}
	\notag\Lambda_{\alpha,p,m}^{f,g}(\bar r,\bar u,\bar s,\bar t,z):=&(2\pi)^{-dm}\int_{(\R^d)^m}\int_{(\nabla^{m}_{\bar r,\bar s})^p}\int_{(\nabla^{m}_{\bar u,\bar t})^p}g(s,t)\prod\limits_{j=1}^{mp}f_j(s_j,t_j,z_j)\\&\qquad\times(-iv_j)^{\alpha_j}\exp\{-i\langle v_j,W^H_{s_j,t_j}-z_j\rangle\}\,\mathrm{d}t\mathrm{d}s\mathrm{d}v, 
\end{align}
where $i$ denotes the imaginary unit. Let $\widetilde g:\,[0,T]^{4mp}\to\R$ be the function given by 
\begin{align*}
\widetilde g(s,t)=g(s_1,\ldots,s_{mp},t_1,\ldots,t_{mp})\times g(s_{mp+1},\ldots,s_{2mp},t_{mp+1},\ldots,t_{2mp})
\end{align*}
for every $s=(s_1,\ldots,s_{2mp})$ and $t=(t_1,\ldots,t_{2mp})$.\\ We define
\begin{align}\label{DefIntpartf2}
	&\Psi_{\alpha,p,m}^{f,g}(\bar r,\bar u,\bar s,\bar t,z)\notag\\:=&\Big(\prod\limits_{\ell=1}^d\sqrt{(2\vert\alpha^{(\ell)}\vert)!}\Big)\sum\limits_{\sigma,\pi\in \widehat{\mathcal{P}}_{2m\vert p}} \int_{\nabla^{2mp}_{\bar r,\bar s}}\int_{\nabla^{2mp}_{\bar u,\bar t}}\widetilde g_{\sigma,\pi}(s,t)\prod\limits_{j=1}^{2mp}\vert f_{[j]} (s_{\sigma(j)},t_{\pi(j)},z)\vert\\&\qquad\times\frac{\mathrm{d}t_1\ldots\mathrm{d}t_{2mp}\mathrm{d}s_{1}\ldots\mathrm{d}s_{2mp}}{\prod\limits_{j=1}^{2mp}\vert s_{j}-s_{j+1}\vert^{H_1(\vert\alpha^{}_{[\sigma^{-1}(j)]}\vert+d)}\vert t_{j}-t_{j+1}\vert^{ H_2(\vert\alpha^{}_{[\pi^{-1}(j)]}\vert+d)}}\notag
\end{align}
and
\begin{align}\label{DefIntPartKapa}
	&\Psi_{\alpha,p,m}^{\kappa,g}(\bar r,\bar u,\bar s,\bar t)\notag\\:=&\Big(\prod\limits_{\ell=1}^d\sqrt{(2\vert\alpha^{(\ell)}\vert)!}\Big)\sum\limits_{\sigma,\pi\in \widehat{\mathcal{P}}_{2m\vert p}} \int_{\nabla^{2mp}_{\bar r,\bar s}}\int_{\nabla^{2mp}_{\bar u,\bar t}}\widetilde g_{\sigma,\pi}(s,t_{})\prod\limits_{j=1}^{2mp}\vert \kappa_{[j]} (s_{\sigma(j)},t_{\pi(j)})\vert\\&\qquad\times\frac{\mathrm{d}t_1\ldots\mathrm{d}t_{2mp}\mathrm{d}s_{1}\ldots\mathrm{d}s_{2mp}}{\prod\limits_{j=1}^{2mp}\vert s_{j}-s_{j+1}\vert^{H_1(\vert\alpha^{}_{[\sigma^{-1}(j)]}\vert+d)}\vert t_{j}-t_{j+1}\vert^{ H_2(\vert\alpha^{}_{[\pi^{-1}(j)]}\vert+d)}}\notag
\end{align}
where $s_{2mp+1}=\bar r$, $t_{2mp+1}=\bar u$, $[j]=j$ for $j\in\{1,\ldots,mp\}$, $[j]=j-mp$ for $j\in\{mp+1,\ldots,2mp\}$,    $\vert\alpha^{(\ell)}\vert=\sum_{j=1}^m\alpha_j^{(\ell)}$, $\vert\alpha_{[j]}\vert=\sum_{\ell=1}^d\alpha_j^{(\ell)}$ and
\begin{align*}
\widetilde g_{\sigma,\pi}(s,t)=\widetilde{g}(s_{\sigma(1)},\ldots,s_{\sigma(2mp)},t_{\pi(1)},\ldots,t_{\pi(2mp)}).
\end{align*}

\begin{thm}\label{RegulariseTheo}
	Suppose $\Psi^{f,g}_{\alpha,p,m}(\bar r,\bar u,\bar s,\bar t,z)$ and $\Psi_{\alpha,p,m}^{\kappa,g}(\bar r,\bar u,\bar s,\bar t)$ given by \eqref{DefIntpartf2} and \eqref{DefIntPartKapa}, respectively, are finite. Then the random variable $\Lambda_{\alpha,p,m}^{f,g}(\bar r,\bar u,\bar s,\bar t,z)$ defined in \eqref{DefIntPart1} belongs to $L^2(\Omega)$ and there exists a constant $C=C(\esp,H,T,d)>0$ such that
	\begin{align}\label{IntPartFormEst1}
		\E\Big[\Big\vert\Lambda_{\alpha,p,m}^{f,g}(\bar r,\bar u,\bar s,\bar t,z)\vert^2\Big]\leq C^{mpd+\vert\alpha\vert}\Psi_{\alpha,p,m}^{f,g}(\bar r,\bar u,\bar s,\bar t,z).
	\end{align}
	Moreover,
	\begin{align}\label{IntPartFormEst2}
		\Big\vert\E\Big[ \int_{(\R^d)^m}\Lambda_{\alpha,p,m}^{\kappa f,g}(\bar r,\bar u,\bar s,\bar t,z)\,\mathrm{d}z\Big]\Big\vert\leq C^{(md+\vert\alpha\vert)/2}(\Psi_{\alpha,p,m}^{\kappa,g}(\bar r,\bar u,\bar s,\bar t))^{1/2}\prod\limits_{j=1}^m\Vert f_j\Vert_{L^1_{\infty}}
	\end{align}
	and the integration by parts formula
	\begin{align}\label{IntPartFormEst3}
		\int_{(\nabla^{m}_{\bar r,\bar s})^p}\int_{(\nabla^{m}_{\bar u,\bar t})^p}g(s,t)D^{\alpha}f(s,t,W^H_{s,t})\mathrm{d}s\mathrm{d}t=\int_{(\R^d)^m}\Lambda_{\alpha,p,m}^{f,g}(\bar r,\bar u,\bar s,\bar t,z)\,\mathrm{d}z,
	\end{align}
	holds.
\end{thm}
\begin{proof}  {We start by proving \eqref{IntPartFormEst1}. The proof follows as in the proof of \cite[Theorem 3.1]{BNP18} with some modification.} 	We have
	\begin{align*}
		&\vert \Lambda_{\alpha,p,m}^{f,g}(\bar r,\bar u,\bar s,\bar t,z)\vert^2\\
		=&(2\pi)^{-2dmp}\int_{(\R^d)^{mp}}\int_{(\nabla^{m}_{\bar r,\bar s})^p}\int_{(\nabla^{m}_{\bar u,\bar t})^p}\prod\limits_{j=1}^{mp}f_j(s_j,t_j,z_j)(-iu_j)^{\alpha_j}e^{-i\langle u_j,W^H_{s_j,t_j}-z_j\rangle} \\&\times\int_{(\R^d)^{mp}}\int_{(\nabla^{m}_{\bar r,\bar s})^p}\int_{(\nabla^{m}_{\bar u,\bar t})^p}  \prod\limits_{j=mp+1}^{2mp}f_{[j]}(s_j,t_j,z_{[j]})(iu_j)^{\alpha_{[j]}}e^{i\langle u_j,W^H_{s_j,t_j}-z_{[j]}\rangle} \\&\times \widetilde g(s,t)\mathrm{d}t_1\ldots\mathrm{d}t_{2mp}\mathrm{d}s_1\ldots\mathrm{d}s_{2mp}\mathrm{d}u_1\ldots\mathrm{d}u_{2mp}
		\\=&\Big(\frac{1}{4\pi^2}\Big)^{dmp}\int_{(\R^d)^{mp}}\int_{(\nabla^{m}_{\bar r,\bar s})^p}\int_{(\nabla^{m}_{\bar u,\bar t})^p}\prod\limits_{j=1}^{mp}f_j(s_j,t_j,z_j)(-iu_j)^{\alpha_j}e^{-i\langle u_j,W^H_{s_j,t_j}-z_j\rangle} \\&\times\int_{(\R^d)^{mp}}\int_{(\nabla^{m}_{\bar r,\bar s})^p}\int_{(\nabla^{m}_{\bar u,\bar t})^p}\prod\limits_{j=mp+1}^{2mp}f_{[j]}(s_j,t_j,z_{[j]})(-iu_j)^{\alpha_{[j]}}e^{-i\langle u_j,W^H_{s_j,t_j}-z_{[j]}\rangle} \\&\times\widetilde g(s,t) \mathrm{d}t_1\ldots\mathrm{d}t_{2mp}\mathrm{d}s_1\ldots\mathrm{d}s_{2mp}\mathrm{d}u_1\ldots\mathrm{d}u_{2mp},
	\end{align*}
	where 
	 the change of variables $(u_{mp+1},\ldots,u_{2mp})\longmapsto(-u_{mp+1},\ldots,-u_{2mp})$ is applied in the last equality.
	Since $\{\nabla^{2mp,\sigma}_{\bar r, \bar s}\}_{\sigma\in\widehat{\mathcal{P}}_{m\vert 2p}}$  is a partition of $(\nabla^{m}_{\bar r,\bar s})^{2p}$  (see Lemma \ref{lemma:Shuffle}), we obtain
	\begin{align*}
		&\vert \Lambda_{\alpha,p,m}^{f,g}(\bar r,\bar s,\bar u,\bar t,z)\vert^2\\=&\frac{(-1)^{\vert\alpha\vert}}{(4\pi^2)^{dm}}\sum\limits_{\sigma,\pi\in \widehat{\mathcal{P}}_{m\vert 2p}}\int_{(\R^d)^{2mp}}\prod\limits_{j=1}^{mp}e^{-i\langle z_j,u_j+u_{j+mp}\rangle}\int_{\nabla^{2mp,\sigma}_{\bar r,\bar s}}\int_{\nabla^{2mp,\pi}_{\bar u,\bar t}}\widetilde g_{}(s,t)\widetilde{f}_{}(s,t,z)\\&\times\prod\limits_{j=1}^{2mp}u_j^{\alpha_{[j]}}\exp\Big\{-i\sum\limits_{j=1}^{2mp}\langle u_j,W^H_{s_j,t_{\sigma(j)}}\rangle\Big\}\mathrm{d}t_j\mathrm{d}s_j\mathrm{d}u_1\ldots\mathrm{d}u_{2mp},
	\end{align*}
	where $\vert\alpha\vert=\sum_{l=1}^d\sum_{j=1}^{mp}\alpha^{(l)}_{j}$, and 
	$
		\widetilde{f}_{}(s,t,z):=\prod\limits_{j=1}^{2mp}f_{[j]}(s_{j},t_{j},z_j).
	$
	 
	By taking the expectation on both sides we obtain
	\begin{align}
		&\notag\E\left[\vert \Lambda_{\alpha,p,m}^{f,g}(\bar r,\bar s,\bar u,\bar t,z)\vert^2\right]\\=&\notag(4\pi^2)^{-dmp}(-1)^{\vert\alpha\vert}\sum\limits_{\sigma,\pi\in\widehat{\mathcal{P}}_{m\vert2p}}\int_{(\R^d)^{2mp}}\prod\limits_{j=1}^{mp}e^{-i\langle z_j,u_j+u_{j+mp}\rangle}\int_{\nabla^{2mp,\sigma}_{\bar r,\bar s}}\int_{\nabla^{2mp,\pi}_{\bar u,\bar t}}\widetilde g_{}(s,t)\\&\notag\times\widetilde{f}_{}(s,t,z)\prod\limits_{j=1}^{2mp}u_j^{\alpha_{[j]}}\exp\Big\{-\frac{1}{2}\text{Var}\Big[\sum\limits_{j=1}^{2mp}\langle u_j,W^H_{s_j,t_{j}}\rangle\Big]\Big\}\mathrm{d}u_1\ldots\mathrm{d}u_{2mp}\,\mathrm{d}t_j\mathrm{d}s_j\\=&\notag(4\pi^2)^{-dmp}(-1)^{\vert\alpha\vert}\sum\limits_{\sigma,\pi\in\widehat{\mathcal{P}}_{m\vert2p}}\int_{(\R^d)^{2mp}}\prod\limits_{j=1}^{mp}e^{-i\langle z_j,u_j+u_{j+mp}\rangle}\int_{\nabla^{2mp,\sigma}_{\bar r,\bar s}}\int_{\nabla^{2mp,\pi}_{\bar u,\bar t}}\widetilde g_{}(s,t)\\&\notag\times\widetilde{f}_{}(s,t,z)\prod\limits_{j=1}^{2mp}u_j^{\alpha_{[j]}}\exp\Big\{-\frac{1}{2}\sum\limits_{\ell=1}^d\text{Var}\Big[\sum\limits_{j=1}^{2mp} u^{(\ell)}_jW^{H,(\ell)}_{s_j,t_{j}}\Big]\Big\}\mathrm{d}u_1\ldots\mathrm{d}u_{2mp}
		\mathrm{d}t_j\mathrm{d}s_j\\=&\notag\frac{(-1)^{\vert\alpha\vert}}{(4\pi^2)^{dmp}}\sum\limits_{\sigma,\pi\in\widehat{\mathcal{P}}_{m\vert2p}}\int_{(\R^d)^{2mp}}\prod\limits_{j=1}^{mp}e^{-i\langle z_j,u_j+u_{j+mp}\rangle}\int_{\nabla^{2mp,\sigma}_{\bar r,\bar s}}\int_{\nabla^{2mp,\pi}_{\bar u,\bar t}}\widetilde g_{}(s,t)\widetilde{f}_{}(s,t,z)\\&\notag\times\prod\limits_{j=1}^{2mp}u_j^{\alpha_{[j]}}\prod\limits_{\ell=1}^{d}\exp\Big\{-\frac{1}{2} (u^{(\ell)})^{\ast}Qu^{(\ell)}\Big\}\mathrm{d}u^{(\ell)}_1\ldots\mathrm{d}u^{(\ell)}_{2mp}\mathrm{d}t_1\ldots\mathrm{d}t_{2mp}\mathrm{d}s_{1}\ldots\mathrm{d}s_{2mp} \\
		\leq&\frac{1}{(4\pi^2)^{dmp}}\sum\limits_{\sigma,\pi\in\widehat{\mathcal{P}}_{m\vert2p}} \int_{\nabla^{2mp,\sigma}_{\bar r,\bar s}}\int_{\nabla^{2mp,\pi}_{\bar u,\bar t}}\widetilde g_{}(s,t)\vert \widetilde{f}_{}(s,t,z)\vert\prod\limits_{\ell=1}^{d}\int_{\R^{2mp}}\Big(\prod\limits_{j=1}^{2mp}\vert u_j^{(\ell)}\vert^{\alpha^{(\ell)}_{[j]}}\Big)\label{IntPartEstimat1}\\& \notag\times\exp\Big\{-\frac{1}{2}\langle Qu^{(\ell)},u^{(\ell)}\rangle \Big\}\mathrm{d}u^{(\ell)}_1\ldots\mathrm{d}u^{(\ell)}_{2mp} \mathrm{d}t_1\ldots\mathrm{d}t_{2mp}\mathrm{d}s_{1}\ldots\mathrm{d}s_{2mp},\notag
	\end{align}
	where $\ast$ denotes the transposition, $u^{(\ell)}=(u^{(\ell)}_1,\ldots,u^{(\ell)}_{2mp})$ and $$Q=Q^{(\ell)}(s_1,t_{1},\ldots,s_{2mp},t_{2mp}):=\left(\E\left[W^{H,(\ell)}_{s_i,t_{i}}W^{H,(\ell)}_{s_j,t_{j}}\right]\right)_{1\leq i,j\leq 2mp}.$$
	Let $(s_1,\ldots,s_{2mp})\in\nabla^{2mp,\sigma}_{\bar r,\bar s}$ and $(t_1,\ldots,t_{2mp})\in\nabla^{2mp,\pi}_{\bar u,\bar t}$. By the change of variable $v^{(\ell)}=Q^{1/2}u^{(\ell)}$, we have
	\begin{align}\notag
		&\int_{\R^{2mp}}\Big(\prod\limits_{j=1}^{2mp}\vert u_j^{(\ell)}\vert^{\alpha^{(\ell)}_{[j]}}\Big)\exp\Big\{-\frac{1}{2}\langle Qu^{(\ell)},u^{(\ell)}\rangle \Big\}\mathrm{d}u_1^{(\ell)}\ldots\mathrm{d}u_{2mp}^{(\ell)}\\=&\notag\int_{\R^{2mp}}\Big(\prod\limits_{j=1}^{2mp}\Big\vert \langle u^{(\ell)}, e_j\rangle\Big\vert^{\alpha^{(\ell)}_{[j]}}\Big)\exp\Big\{-\frac{1}{2}\langle Q^{1/2}u^{(\ell)},Q^{1/2}u^{(\ell)}\rangle \Big\}\mathrm{d}u_1^{(\ell)}\ldots\mathrm{d}u_{2mp}^{(\ell)}\\=&\notag\frac{1}{(\text{det }Q)^{1/2}}\int_{\R^{2mp}}\Big(\prod\limits_{j=1}^{2mp}\Big\vert \langle Q^{-1/2}v^{(\ell)}, e_j\rangle\Big\vert^{\alpha^{(\ell)}_{[j]}}\Big)\exp\Big\{-\frac{1}{2}\langle v^{(\ell)},v^{(\ell)}\rangle \Big\}\mathrm{d}v_1^{(\ell)}\ldots\mathrm{d}v_{2mp}^{(\ell)}\\=&\frac{(2\pi)^{mp}}{(\text{det }Q)^{1/2}}\E\Big[\prod\limits_{j=1}^{2mp}\Big\vert \langle Q^{-1/2}G, e_j\rangle\Big\vert^{\alpha^{(\ell)}_{[j]}}\Big]\label{IntPartChangeVar},
	\end{align}
	where $e_i,\,i=1,\ldots,2mp$ is the standard orthonormal basis of $\R^{2mp}$ and $G=(G_1,\ldots,G_{2mp})$ a centered $\R^{2m}$-valued Gaussian random vector with covariance matrix $I_{2mp\times2mp}$. 
	
	We deduce from Lemma \ref{LemLW12} and \cite[Theorem]{Ma63} (see also \cite[Sect. 1]{AGGS17}) that
	\begin{align*}
		\E\Big[\prod\limits_{j=1}^{2mp}\Big\vert \langle Q^{-1/2}G, e_j\rangle\Big\vert^{\alpha^{(\ell)}_{[j]}}\Big]\leq&\sqrt{\text{perm}(\Sigma)}\leq\sqrt{(2\vert\alpha^{(\ell)}\vert)! \prod\limits_{i=1}^{2\vert\alpha^{(\ell)}\vert}\sigma_{i,i}}\\=&\sqrt{(2\vert\alpha^{(\ell)}\vert)!\prod\limits_{j=1}^{2mp}\Big(\E\Big[\vert\langle Q^{-1/2}G,e_j\rangle\vert^{2}\Big]\Big)^{\alpha^{(\ell)}_{[j]}}},
	\end{align*}
	where perm$(\Sigma)$ is the permanent of the covariance matrix $\Sigma=(\Sigma_{i,j})$ of the Gaussian random vector 
	\begin{align*}
		\Big(\underbrace{\langle Q^{-1/2}G,e_1\rangle,\ldots,\langle Q^{-1/2}G,e_1\rangle}_{\alpha^{(\ell)}_{[1]}\text{ times}},\ldots,\underbrace{\langle Q^{-1/2}G,e_{2mp}\rangle,\ldots,\langle Q^{-1/2}G,e_{2mp}\rangle}_{\alpha^{(\ell)}_{[2mp]}\text{ times}}\Big)
	\end{align*}
	and $\vert\alpha^{(\ell)}\vert:=\sum_{j=1}^{mp}\alpha_j^{(\ell)}$.
	
	Applying the change of variable $u=Q^{-1/2}v$ and Lemma \ref{LemCD82}, we obtain
	\begin{align*}
		\E\Big[\vert\langle Q^{-1/2}G,e_j\rangle\vert^{2}\Big]=&\frac{1}{(2\pi)^{mp}}\int_{\R^{2mp}}\langle Q^{-1/2}v,e_j\rangle^2\exp\Big(-\frac{1}{2}\vert v\vert^2\Big)\,\mathrm{d}v\\=&\frac{(\text{det }Q)^{1/2}}{(2\pi)^{m}}\int_{\R^{2mp}}\langle u,e_j\rangle^2\exp\Big(-\frac{1}{2}\langle Qu,u\rangle\Big)\mathrm{d}u_1\ldots\mathrm{d}u_{2mp}\\=&\frac{(\text{det }Q)^{1/2}}{(2\pi)^{mp}}\int_{\R^{2mp}}u_j^2\exp\Big(-\frac{1}{2}\langle Qu,u\rangle\Big)\mathrm{d}u_1\ldots\mathrm{d}u_{2mp}\\=&\frac{(\text{det }Q)^{1/2}}{(2\pi)^{mp}}\frac{(2\pi)^{(2mp-1)/2}}{(\text{det }Q)^{1/2}}\int_{\R}\frac{v^2}{\widehat\sigma_j^2}\exp\Big(-\frac{1}{2}v^2\Big)\,\mathrm{d}v=\frac{1}{\widehat{\sigma}^2_j},
	\end{align*}
	where $\widehat{\sigma}_j^2:=\text{Var}[W^{(\ell),H}_{s_j,t_{j}}\vert W^{(\ell),H}_{s_i,t_{i}},\,i=1,\ldots,2mp,\,i\neq j]$.\\
	Applying Theorem \ref{SLNDPropTheo}  and Young inequality, we have
	\begin{align*}
		\widehat{\sigma}_j^2\geq
		&C_{\esp,H}\left(\min\limits_{1\leq i\leq 2mp,\,i\neq j}\vert s_j-s_i\vert^{2H_1}+\min\limits_{1\leq i\leq 2mp,\,i\neq j}\vert t_{j}-t_{i}\vert^{2H_2}\right) \\\geq&2C_{\esp,H}\min\limits_{1\leq i\leq 2mp,\,i\neq j}\vert s_j-s_i\vert^{H_1}\min\limits_{1\leq i\leq 2mp,\,i\neq j}\vert t_{\sigma(j)}-t_i\vert^{H_2}.
	\end{align*}
	Hence,
	\begin{align*}
		&\prod\limits_{j=1}^{2mp}(\widehat{\sigma}^2_j)^{\alpha^{(\ell)}_{[j]}}\\\geq&(4C^2_{\esp,H})^{\vert\alpha^{(\ell)}\vert}\prod\limits_{j=1}^{2mp}\min\limits_{1\leq i\leq 2mp,\,i\neq j}\vert s_j-s_i\vert^{H_1\alpha^{(\ell)}_{[j]}}\prod\limits_{j=1}^{2mp}\min\limits_{1\leq i\leq 2mp,\,i\neq j}\vert t_j-t_i\vert^{H_2\alpha^{(\ell)}_{[j]}}\\=&(4C^2_{\esp,H})^{\vert\alpha^{(\ell)}\vert}\prod\limits_{j=1}^{2mp}\min\limits_{1\leq i\leq 2mp,\,i\neq \sigma^{-1}(j)}\vert s_{\sigma^{-1}(j)}-s_i\vert^{H_1\alpha^{(\ell)}_{[\sigma^{-1}(j)]}}\\&\qquad\times\prod\limits_{j=1}^{2mp}\min\limits_{1\leq i\leq 2mp,\,i\neq \pi^{-1}(j)}\vert t_{\pi^{-1}(j)}-t_i\vert^{H_2\alpha^{(\ell)}_{[\pi^{-1}(j)]}}\\\geq&(4C^2_{\esp,H})^{\vert\alpha^{(\ell)}\vert} \prod\limits_{j=1}^{2mp}\min\{\vert s_{\sigma^{-1}(j)}-s_{\sigma^{-1}(j+1)}\vert,\vert s_{\sigma^{-1}(j)}-s_{\sigma^{-1}(j-1)}\vert\}^{H_1\alpha^{\sigma,(\ell)}_{[j]}}\\&\qquad\times\prod\limits_{j=1}^{2mp}\min\{\vert t_{\pi^{-1}(j)}-t_{\pi^{-1}(j+1)}\vert,\vert t_{\pi^{-1}(j)}-t_{\pi^{-1}(j-1)}\vert\}^{H_2\alpha^{\pi,(\ell)}_{[j]}}\\\geq&
		(4C^2_{\esp,H})^{\vert\alpha^{(\ell)}\vert}\prod\limits_{j=1}^{2mp}\vert s_{\sigma^{-1}(j)}-s_{\sigma^{-1}(j+1)}\vert^{2H_1\alpha^{\sigma,(\ell)}_{[j]}}\vert t_{\pi^{-1}(j)}-t_{\pi^{-1}(j+1)}\vert^{2H_2\alpha^{\pi,(\ell)}_{[j]}},
	\end{align*}
	where $s_{\sigma^{-1}(2m+1)}=\bar r$, $ t_{\pi^{-1}(2m+1)}=\bar u$, $\alpha^{\sigma,(\ell)}_{[j]}=\alpha^{(\ell)}_{[\sigma^{-1}(j)]}$ and $\alpha^{\pi,(\ell)}_{[j]}=\alpha^{(\ell)}_{[\pi^{-1}(j)]}$.\\
	As a consequence,
	\begin{align}
		&\notag\prod\limits_{j=1}^{2mp}\Big(\E\Big[\vert\langle Q^{-1/2}G,e_j\rangle\vert^{2}\Big]\Big)^{\alpha^{(\ell)}_{[j]}}=\prod\limits_{j=1}^{2mp}\Big(\frac{1}{\widehat{\sigma}_j^2}\Big)^{\alpha^{(\ell)}_{[j]}}\\\leq& \label{SectLND}\frac{1}{(4C^2_{\esp,H})^{\vert\alpha^{(\ell)}\vert}}\prod\limits_{j=1}^{2mp}\frac{1}{\vert s_{\sigma^{-1}(j)}-s_{\sigma^{-1}(j+1)}\vert^{2H_1\alpha^{\sigma,(\ell)}_{[j]}}\vert t_{\pi^{-1}(j)}-t_{\pi^{-1}(j+1)}\vert^{2H_2\alpha^{\pi,(\ell)}_{[j]}}}.
	\end{align}
	We deduce from \eqref{IntPartChangeVar} and \eqref{SectLND}  {that}
	\begin{align}
		\notag&\int_{\R^{2mp}}\Big(\prod\limits_{j=1}^{2mp}\vert u_j^{(\ell)}\vert^{\alpha^{(\ell)}_{[j]}}\Big)\exp\Big\{-\frac{1}{2}\langle Qu^{(\ell)},u^{(\ell)}\rangle \Big\}\mathrm{d}u_1^{(\ell)}\ldots\mathrm{d}u_{2mp}^{(\ell)}\\ \leq&\frac{(2\pi)^{mp} (C^{(1)}_{\esp,H})^{\vert\alpha^{(\ell)}\vert}\sqrt{(2\vert\alpha^{(\ell)}\vert)!}}{ (\text{det }Q)^{1/2}\prod\limits_{j=1}^{2mp}\vert s_{\sigma^{-1}(j)}-s_{\sigma^{-1}(j+1)}\vert^{H_1\alpha^{\sigma,(\ell)}_{[j]}}\vert t_{\pi^{-1}(j)}-t_{\pi^{-1}(j+1)}\vert^{H_2\alpha^{\pi,(\ell)}_{[j]}}}, \label{SectLND2}
	\end{align}
	where $C^{(1)}_{\esp,H}=1/2C_{\esp,H}$. Using once again Theorem \eqref{SLNDPropTheo},  {we obtain}
	\begin{align}
		\notag\text{det }Q=&\text{det Cov}\left(W^{(\ell),H}_{s_i,t_i},\,i=1,\ldots,2mp\right)\\=\notag&\text{Var}[W^{(\ell),H}_{s_1,t_1}]\prod\limits_{j=2}^{2mp}\text{Var}[W^{(\ell),H}_{s_j,t_j}\vert W^{(\ell),H}_{s_i,t_i},\,i=1,\,\ldots,\,j-1]\\\geq\notag&\frac{\esp^{2(H_1+H_2)}}{4}C_{\esp,H}\prod\limits_{j=2}^{2mp}\left(\min\limits_{1\leq i\leq j-1}\vert s_j-s_i\vert^{2H_1}+\min\limits_{1\leq i\leq j-1}\vert t_j-t_i\vert^{2H_2}\right)\\\geq&\frac{1}{4}\Big(\frac{\esp}{T}\Big)^{2(H_1+H_2)}C_{\esp,H}\prod\limits_{j=1}^{2mp}\vert s_{\sigma^{-1}(j)}-s_{\sigma^{-1}(j+1)}\vert^{2H_1}\vert t_{\pi^{-1}(j)}-t_{\pi^{-1}(j+1)}\vert^{2H_2}.\label{SectLND3}
	\end{align}
	Plugging \eqref{SectLND3} into \eqref{SectLND2}, we obtain
	\begin{align}
		\notag&\int_{\R^{2mp}}\Big(\prod\limits_{j=1}^{2mp}\vert u_j^{(\ell)}\vert^{\alpha^{(\ell)}_{[j]}}\Big)\exp\Big\{-\frac{1}{2}\langle Qu^{(\ell)},u^{(\ell)}\rangle \Big\}\mathrm{d}u_1^{(\ell)}\ldots\mathrm{d}u_{2mp}^{(\ell)}\\ \leq&\frac{  (C^{(2)}_{\esp,H,T})^{mp+\vert\alpha^{(\ell)}\vert}\sqrt{(2\vert\alpha^{(\ell)}\vert)!}}{ \prod\limits_{j=1}^{2mp}\vert s_{\sigma^{-1}(j)}-s_{\sigma^{-1}(j+1)}\vert^{H_1(\alpha^{\sigma,(\ell)}_{[j]}+1)}\vert t_{\pi^{-1}(j)}-t_{\pi^{-1}(j+1)}\vert^{H_2(\alpha^{\pi,(\ell)}_{[j]}+1)}}.\label{SectLND4}
	\end{align}
	where
	\begin{align*}
		C^{(2)}_{\esp,H,T}=2\pi \Big(\frac{1}{\min\{1,2C^2_{\esp,H}\}}\Big)^{1/2}   \Big(\frac{T}{\min\{1,\esp\}}\Big)^{H_1+H_2}.
	\end{align*}
	Using \eqref{IntPartEstimat1} and the change of variable $\ps_k=s_{\sigma^{-1}(k)},\,\tp_k=t_{\pi^{-1}(k)},\,k=1,\ldots,2mp$, we obtain
	\begin{align*}
		&\notag\E\left[\vert \Lambda_{\alpha,p,m}^{f,g}(\bar r,\bar s,\bar u,\bar t,z)\vert^2\right]\\ \leq& C_{\esp,H,T}^{ mpd+\vert\alpha\vert}\Big(\prod\limits_{\ell=1}^d\sqrt{(2\vert\alpha^{(\ell)}\vert)!}\Big)\sum\limits_{\sigma,\pi\in\widehat{\mathcal{P}}_{m\vert2p}} \int_{\nabla^{2mp,\sigma}_{\bar r,\bar s}}\int_{\nabla^{2mp,\pi}_{\bar u,\bar t}}\widetilde g(s,t)\vert \widetilde{f}(s,t,z)\vert\\&\qquad\times\frac{\mathrm{d}t_1\ldots\mathrm{d}t_{2mp}\mathrm{d}s_{1}\ldots\mathrm{d}s_{2mp}}{\prod\limits_{j=1}^{2mp}\vert s_{\sigma^{-1}(j)}-s_{\sigma^{-1}(j+1)}\vert^{H_1(\alpha^{\sigma}_{[j]}\vert+d)}\vert t_{\pi^{-1}(j)}-t_{\pi^{-1}(j+1)}\vert^{H_2(\vert\alpha^{\pi}_{[j]}\vert+d)}}\\=&C_{\esp,H,T}^{ mpd+\vert\alpha\vert}\Big(\prod\limits_{\ell=1}^d\sqrt{(2\vert\alpha^{(\ell)}\vert)!}\Big)\sum\limits_{\sigma,\pi\in\mathcal{P}_{2mp}} \int_{\nabla^{2mp}_{\bar r,\bar s}}\int_{\nabla^{2mp}_{\bar u,\bar t}}\widetilde g_{\sigma,\pi}(\ps,\tp)\vert \widetilde{f}_{\sigma,\pi}(\ps,\tp,z)\vert\\&\qquad\times\frac{\mathrm{d}\tp_1\ldots\mathrm{d}\tp_{2mp}\mathrm{d}\ps_{1}\ldots\mathrm{d}\ps_{2mp}}{\prod\limits_{j=1}^{2mp}\vert \ps_{j}-\ps_{j+1}\vert^{H_1(\vert\alpha^{\sigma}_{[j]}\vert+d)}\vert \tp_{j}-\tp_{j+1}\vert^{ H_2(\vert\alpha^{\pi}_{[j]}\vert+d)}},	 
	\end{align*}
	where $\vert\alpha^{}_{}\vert=\sum_{\ell=1}^d\alpha^{(\ell)}_{}$, $\vert\alpha^{\sigma}_{[j]}\vert=\sum_{\ell=1}^d\alpha^{\sigma,(\ell)}_{[j]}$, $\vert\alpha^{\pi}_{[j]}\vert=\sum_{\ell=1}^d\alpha^{\pi,(\ell)}_{[j]}$, $C_{\esp,H,T}=4\pi C^{(2)}_{\esp,H,T}$, $\widetilde g_{\sigma,\pi}(\ps,\tp)=g(\ps_{\sigma(1)},\ldots,\ps_{\sigma(2mp)},\tp_{\pi(1)},\ldots,\tp_{\pi(2mp)})$ and
	$
		\widetilde{f}_{\sigma,\pi}(\ps,\tp,z):=\widetilde{f}(\ps_{\sigma(1)},\ldots,\ps_{\sigma(2m)},\tp_{\pi(1)},\ldots,\tp_{\pi(2m)},z).
	$
	This completes the proof of \eqref{IntPartFormEst1}. 
	
	We now prove inequality \eqref{IntPartFormEst2}. Using \eqref{IntPartFormEst1}, we have
	\begin{align*}
		&\Big\vert\E\Big[ \int_{(\R^d)^{mp}}\Lambda_{\alpha}^{\kappa f}(\bar r,\bar u,\bar s,\bar t,z)\,\mathrm{d}z\Big]\Big\vert\leq\int_{(\R^d)^{mp}}\Big(\E\Big[\vert\Lambda_{\alpha}^{\kappa f}(\bar r,\bar u,\bar s,\bar t,z)\vert^2\Big]\Big)^{1/2}\,\mathrm{d}z\\ &\leq C^{ (mpd+\vert\alpha\vert)/2}\int_{(\R^d)^{mp}}\left(\Psi^{\kappa f}_{\alpha}(\bar r,\bar u,\bar s,\bar t,z)\right)^{1/2}\,\mathrm{d}z \\&\leq C^{ (mpd+\vert\alpha^{}\vert)/2}\Big(\prod\limits_{\ell=1}^d\sqrt{(2\vert\alpha^{(\ell)}\vert)!}\Big)^{1/2}\Big(\sup\limits_{0\leq r,u\leq T}\prod\limits_{j=1}^{2mp}\int_{\R^d}\vert f_{[j]}(r,u,z_j)\vert_{}\mathrm{d}z_j\Big)^{1/2} \\&\qquad\times\Big(\sum\limits_{\sigma,\pi\in\widehat{\mathcal{P}}_{m\vert2p}} \int_{\nabla^{2mp}_{\bar r,\bar s}}\int_{\nabla^{2mp}_{\bar u,\bar t}}\widetilde g_{\sigma,\pi}(s,t)\prod\limits_{j=1}^{2mp}\frac{\kappa_{[j]}(s_{\sigma(j)},t_{\pi(j)})}{\vert s_{j}-s_{j+1}\vert^{H_1(\vert\alpha^{\sigma}_{[j]}\vert+d)}\vert t_{j}-t_{j+1}\vert^{ H_2(\vert\alpha^{\pi}_{[j]}\vert+d)}}\\&\qquad\qquad\qquad\qquad\times\mathrm{d}t_1\ldots\mathrm{d}t_{2mp}\mathrm{d}s_{1}\ldots\mathrm{d}s_{2mp}
		\Big)^{1/2}\mathrm{d}z\\&=C^{ (mpd+\vert\alpha\vert)/2}\left(\Psi_{\alpha,p,m}^{\kappa,g}(\bar r,\bar u,\bar s,\bar t)\right)^{1/2}\prod\limits_{j=1}^{mp}\Vert f_j\Vert_{L^1_{\infty}}.
	\end{align*}
	Finally, we show the integration by parts formula \eqref{IntPartFormEst3}.			 
	For any $R>0$, define
	\begin{align*}
		\Lambda_{\alpha,p,m,R}^{f,g}(\bar r,\bar u,\bar s,\bar t,z):=&(2\pi)^{-dmp}\int_{B(R)}\int_{(\nabla^{m}_{\bar r,\bar s})^p}\int_{(\nabla^{m}_{\bar u,\bar t})^p}g(s,t)\prod\limits_{j=1}^{mp}f_j(s_j,t_j,z_j)(-iv_j)^{\alpha_j}\\&\qquad\qquad\times\exp\{-i\langle v_j,W^H_{s_j,t_j}-z_j\rangle\}\,\mathrm{d}t\mathrm{d}s\mathrm{d}v,
	\end{align*}
	where $B(R):=\{z\in(\R^d)^m:\,\vert z\vert<R\}$. This implies
	\begin{align}\label{IntPartRHS}
		\vert\Lambda_{\alpha,p,m,R}^{f,g}(\bar r,\bar u,\bar s,\bar t,z)\vert\leq C_R\int_{(\nabla^{m}_{\bar r,\bar s})^p}\int_{(\nabla^{m}_{\bar u,\bar t})^p}g(s,t)\prod\limits_{j=1}^{mp}\vert f_j(s_j,t_j,z_j)\vert\,\mathrm{d}s\mathrm{d}t
	\end{align}
	for a sufficiently large constant $C_R>0$. Since the right-hand side in \eqref{IntPartRHS} is integrable over $(\R^d)^m$,
	similar computations as above show that $\Lambda_{\alpha,p,m,R}^{f,g}(\bar r,\bar u,\bar s,\bar t,z)$ converges to $\Lambda_{\alpha,p,m}^{f,g}(\bar r,\bar u,\bar s,\bar t,z)$ in $L^2(\Omega)$ as $R$ goes to $\infty$ for all $\bar r$, $\bar u$, $\bar s$, $\bar t$ and $z$.
	
	Using the Lebesgue's dominated convergence theorem and the fact that the Fourier transform is an automorphism on the Schwarz space, one has
	\begin{align*}
		&\int_{(\R^d)^{mp}}\Lambda_{\alpha,p,m}^{f,g}(\bar r,\bar u,\bar s,\bar t,z)\,\mathrm{d}z=\lim\limits_{R\to\infty}\int_{(\R^d)^{mp}}\Lambda_{\alpha,p,m,R}^{f,g}(\bar r,\bar u,\bar s,\bar t,z)\,\mathrm{d}z\\
		&=\lim\limits_{R\to\infty}(2\pi)^{-dmp}\int_{(\R^d)^{mp}}\int_{B(R)}\int_{(\nabla^{m}_{\bar r,\bar s})^p}\int_{(\nabla^{m}_{\bar u,\bar t})^p}g(s,t)\prod\limits_{j=1}^{mp}f_j(s_j,t_j,z_j)(-iv_j)^{\alpha_j}\\&\qquad\qquad\qquad\times e^{-i\langle v_j,W^H_{s_j,t_j}-z_j\rangle}\,\mathrm{d}s\mathrm{d}t\mathrm{d}v\mathrm{d}z\\
		&=\lim\limits_{R\to\infty} \int_{(\nabla^{m}_{\bar r,\bar s})^p}\int_{(\nabla^{m}_{\bar u,\bar t})^p}g(s,t)\int_{B(R)}(2\pi)^{-dmp}\int_{(\R^d)^{mp}}\prod\limits_{j=1}^{mp}f_j(s_j,t_j,z_j)e^{i\langle v_j,z_j\rangle}(-iv_j)^{\alpha_j}\\&\qquad\qquad\qquad\times e^{-i\langle v_j,W^H_{s_j,t_j}\rangle}\,\mathrm{d}z\mathrm{d}v\mathrm{d}s\mathrm{d}t\\
		&=\lim\limits_{R\to\infty}\int_{(\nabla^{m}_{\bar r,\bar s})^p}\int_{(\nabla^{m}_{\bar u,\bar t})^p}g(s,t)\int_{B(R)}\prod\limits_{j=1}^{mp}\widehat{f}_j(s_j,t_j,-v_j)(-iv_j)^{\alpha_j}e^{-i\langle v_j,W^H_{s_j,t_j}\rangle}\,\mathrm{d}v\mathrm{d}s\mathrm{d}t\\
		&=\int_{(\nabla^{m}_{\bar r,\bar s})^p}\int_{(\nabla^{m}_{\bar u,\bar t})^p}g(s,t)D^{\alpha}f(s,t,W^H_{s,t})\,\mathrm{d}s\mathrm{d}t,
	\end{align*}	
	where $\widehat{f}_j$ denotes the Fourier	transform of $f_j$ and $\langle\cdot,\cdot\rangle$ the classical inner product in $\R^d$. This ends the proof.
\end{proof}

\subsection{Regularising estimates of the fractional Brownian sheet}		
The next estimate shows that the fractional Brownian sheet regularises \eqref{MainPb}. 

\begin{prop}\label{RegulariseProp1}
	Let $(W_{r,u}^H,(r,u)\in[0,T]^2)$ be a $d$-dimensional fractional Brownian sheet with Hurst index $H=(H_1,H_2)\in(0,\frac{1}{2})^2$ under $(\Omega,\mathcal{F},\Pb)$. Let $f$ and $\kappa$ be given as in \eqref{RegulariseDef1} and \eqref{RegulariseDef2} respectively. Let $p,\,m\in\N$, $\alpha\in(\N_0^m)^d$ be a multi-index, $0<\esp\leq \bar r<r<\bar s\leq T$, $0<\esp\leq \bar u<u<\bar t\leq T$, 
	\begin{align*}
		\kappa_j(s,t)=(K_H(s_j,t_j,r,u)-K_H(s_j,t_j,\bar r,\bar u))^{\ve_j}
	\end{align*} 
	for every $s=(s_1,\ldots,s_m)$, $t=(t_1,\ldots,t_m)$,  $j\in\{1,\dots,m\}$ and $(\ve_1,\ldots,\ve_m)\in\{0,1\}^m$. Suppose
	\begin{align}\label{RegulariseCond}
		\max\{H_1,H_2\}<\frac{\frac{1}{2}-\gamma}{d-1+\vert\alpha_{j}\vert}
	\end{align}
	for all $j=1,\ldots,m$, where $\vert\alpha_j\vert=\sum_{\ell=1}^d\alpha_{j}^{(\ell)}$ and $\gamma\in(0,H)$ is sufficiently small. Then there exists a universal constant $C=C_{\ve,H,T,d}$ (depending on $H$, $T$ and $d$, but independent of $m$, $p$, $\{f_i\}_{i=1,\ldots,m}$ and $\alpha$) such that  
	\begin{align*}
		&\Big\vert\E\Big[\Big(\int_{\nabla^{m}_{ r, \bar s}}\int_{\nabla^{m}_{ u, \bar t}}\prod\limits_{j=1}^mD^{\alpha_j}f_j(s_j,t_j,W^H_{s_j,t_j})\kappa_j(s_j,t_j)\,\mathrm{d}t_j\mathrm{d}s_j\Big)^p\Big]\Big\vert\\\leq& C^{p(md+\vert\alpha\vert)}\left(\prod\limits_{i=1}^d(2p\vert{\alpha}^{(i)}\vert)!\right)^{\frac{1}{4}}\prod\limits_{j=1}^m\Vert f_j\Vert^p_{L^1_{\infty}}\\&\quad\times\Big[r^{H_1-\frac{1}{2}-\gamma}\bar u^{H_2-\frac{1}{2}-\gamma}\Big(\frac{\vert u-\bar u\vert^{\gamma}}{(u\bar u)^{\gamma}} +\frac{\vert r-\bar r\vert^{\gamma}}{(r\bar r)^{\gamma}} \Big)\Big]^{p\sum_{j=1}^{m}\ve_{j}} \\&\quad\times\frac{(\bar s-r)^{p(H_1-\frac{1}{2}-\gamma)\sum_{j=1}^{m} \ve_{j}-p(\vert{\alpha}\vert+md)H_1+pm}}{\Gamma\left(2p(H_1-\frac{1}{2}-\gamma)\sum\limits_{j=1}^{m} \ve_{j}- 2p(\vert{\alpha}\vert+md)H_1+2pm+1\right)^{1/2}}
		\\&\quad\times\frac{(\bar t-u)^{p(H_2-\frac{1}{2}-\gamma)\sum_{j=1}^{m} \ve_{j}-p(\vert{\alpha}\vert+md)H_2+pm}}{\Gamma\left(2p(H_2-\frac{1}{2}-\gamma)\sum\limits_{j=1}^{m} \ve_{j}-2p(\vert{\alpha}\vert+md)H_2+2pm+1\right)^{1/2}}.
	\end{align*}
\end{prop} 
\begin{proof}
	We start by noticing that
	\begin{align*}
		&\Big(\int_{\nabla^{m}_{r,\bar s}}\int_{\nabla^{m}_{u, \bar t}}\prod\limits_{j=1}^{m}D^{\alpha_j}f_j(s_j,t_j,W^H_{s_j,t_j})\kappa_j(s_j,t_j) \mathrm{d}t_{m} \ldots\mathrm{d}t_1\mathrm{d}s_{m}\ldots\mathrm{d}s_1\Big)^p\\=&  \int_{(\nabla_{r,\bar s}^{m})^p} \int_{(\nabla_{u,\bar t}^{m})^p}\prod\limits_{j=1}^{pm}D^{\bar\alpha_{j}}\bar f_{ j}(s_j,t_j,W^H_{s_j,t_j})\bar\kappa_{ j}(s_j,t_j)\mathrm{d}t_1\ldots\mathrm{d}t_{pm}\mathrm{d}s_1\ldots\mathrm{d}s_{pm}\\=&\Lambda^{\bar\kappa\bar f,g}_{\bar\alpha,p,m}(r,\bar s,u,\bar t,z),
	\end{align*}
	where $g=1$, $\bar f_j=f_{j-im}$, $\bar\alpha_j=\alpha_{j-im}$ and
	\begin{align*}
	\bar\kappa_j(s,t)=(K_H(s_j,t_j,r,u)-K_H(s_j,t_j,\bar r,\bar u))^{\bar\ve_j}
	\end{align*}
	with $\bar\ve_j=\ve_{j-im}$ for every $i\in\{0,1,\ldots,p-1\}$,  $j\in\{1+im,\ldots,(1+i)m\}$.
	 
	We deduce from Theorem \ref{RegulariseTheo} that
	\begin{align}
		\notag\Big\vert\E\left[\Lambda^{\bar\kappa\bar f,g}_{\bar\alpha,p,m}(r,\bar s,u,\bar t,z)\right]\vert\notag\leq&\E\left[\Lambda^{\bar\kappa\bar f,g}_{\bar\alpha,p,m}(r,\bar s,u,\bar t,z)^2\right]\\\leq&C_1^{ (pmd+\vert\widetilde\alpha\vert)/2}(\Psi_{\bar\alpha,p,m}^{\bar\kappa,g}(r,\bar s,u,\bar t))^{1/2}\prod\limits_{j=1}^{m}\Vert f_j\Vert^p_{L^1_{\infty}},\label{RegulariseProof1}
	\end{align}	
	where
	\begin{align*}
	&\Psi_{\bar\alpha,p,m}^{\bar\kappa,g}(r,\bar s,u,\bar t)=\prod\limits_{\ell=1}^d\sqrt{(2\vert\bar\alpha^{(\ell)}\vert)!}\sum\limits_{\sigma,\pi\in\widehat{\mathcal{P}}_{m\vert 2p}}\\& \int_{ \nabla^{2mp}_{ r,\bar s}}\int_{ \nabla^{2mp}_{ u,\bar t}}  \prod\limits_{j=1}^{2pm}\frac{\vert (K_H(s_{\sigma(j)},t_{\pi(j)},r,u)-K_H(s_{\sigma(j)},t_{\pi(j)},\bar r,\bar u))^{\bar\ve_{[j]}}\vert\mathrm{d}t_j\mathrm{d}s_{j}}{\vert s_{j}-s_{j+1}\vert^{H_1(\vert\alpha^{}_{[\sigma^{-1}(j)]}\vert+d)}\vert t_{j}-t_{j+1}\vert^{ H_2(\vert\alpha^{}_{[\pi^{-1}(j)]}\vert+d)}}.
	\end{align*}	 
	We deduce from \eqref{RegulariseCond} that
	\begin{align*}
		(H_1-\frac{1}{2}-\gamma)\bar\ve_{[j]}-(\vert\bar\alpha_{[j]}\vert+d)H_1>-1
	\end{align*}
	and
	\begin{align*}
		(H_2-\frac{1}{2}-\gamma)\bar\ve_{[j]}-(\vert\bar\alpha_{[j]}\vert+d)H_2>-1
	\end{align*}
	for all $j=1,\ldots,2pm$.  {Then Lemma \ref{ExistEstLem2} yields}
	\begin{align*}
		&\Psi_{\bar\alpha,p,m}^{\bar\kappa,g}( r, \bar s,u,\bar t)\leq C_1^{pm}\prod\limits_{\ell=1}^d\sqrt{(2\vert\bar\alpha^{(\ell)}\vert)!}\,\Pi_{\gamma,pm}(H_1)\Pi_{\gamma,pm}(H_2)\\&\quad\times\Big[r^{H_1-\frac{1}{2}-\gamma}\bar u^{H_2-\frac{1}{2}-\gamma}\Big(\frac{\vert u-\bar u\vert^{\gamma}}{(u\bar u)^{\gamma}} +\frac{\vert r-\bar r\vert^{\gamma}}{(r\bar r)^{\gamma}} \Big)\Big]^{\sum_{j=1}^{2pm}\bar\ve_{[j]}}\\&\quad\times(\bar s-r)^{(H_1-\frac{1}{2}-\gamma)\sum_{j=1}^{2pm} \bar\ve_{[j]}-2(\vert\bar\alpha\vert+pmd)H_1+2pm}\\&\quad\times(\bar t-u)^{(H_2-\frac{1}{2}-\gamma)\sum_{j=1}^{2pm} \bar\ve_{[j]}-2(\vert\bar\alpha\vert+pmd)H_2+2pm},
	\end{align*}
	where
	\begin{align*}
		&\Pi_{\gamma,pm}(H_1)\\=&\frac{\prod\limits_{j=2}^{2pm}\Gamma\left(1-(\vert\bar\alpha_{[j]}\vert+d)H_1\right)  \prod\limits_{j=1}^{2pm}\Gamma\left((H_1-\frac{1}{2}-\gamma) \bar\ve_{[j]}-(\vert\bar\alpha_{[j]}\vert+d)H_1+1\right)}{\Gamma\left((H_1-\frac{1}{2}-\gamma)\sum\limits_{j=1}^{2pm} \bar\ve_{[j]}- 2(\vert\bar\alpha\vert+pmd)H_1+2pm+1\right)}
	\end{align*}
	and
	\begin{align*}
		&\Pi_{\gamma,pm}(H_2)\\=&\frac{\prod\limits_{j=2}^{2pm}\Gamma\left(1-(\vert\bar\alpha_{[j]}\vert+d)H_2\right)  \prod\limits_{j=1}^{2pm}\Gamma\left((H_2-\frac{1}{2}-\gamma) \bar\ve_{[j]}-(\vert\bar\alpha_{[j]}\vert+d)H_2+1\right)}{\Gamma\left((H_2-\frac{1}{2}-\gamma)\sum\limits_{j=1}^{2pm} \bar\ve_{[j]}- 2(\vert\bar\alpha\vert+pmd)H_2+2pm+1\right)}.
	\end{align*}
	Now observe that $\sum_{j=1}^{2pm}\bar\ve_{[j]}=2p\sum_{j=1}^m\ve_j$, $\vert\bar\alpha\vert=p\vert\alpha\vert$,  $\vert\bar\alpha^{(\ell)}\vert=p\vert\alpha^{(\ell)}\vert$ for every $\ell\in\{1,\ldots,d\}$ and for a large enough positive constant $C_2$,
	\begin{align*}
		&C^{pm}_2\\\geq&\prod\limits_{j=2}^{2pm}\Gamma\left(1-(\vert\bar\alpha_{[j]}\vert+d)H_1\right)  \prod\limits_{j=1}^{2pm}\Gamma\left((H_1-\frac{1}{2}-\gamma) \bar\ve_{[j]}-(\vert\bar\alpha_{[j]}\vert+d)H_1+1\right)\\&\times\prod\limits_{j=2}^{2pm}\Gamma\left(1-(\vert\bar\alpha_{[j]}\vert+d)H_2\right)  \prod\limits_{j=1}^{2pm}\Gamma\left((H_2-\frac{1}{2}-\gamma) \bar\ve_{[j]}-(\vert\bar\alpha_{[j]}\vert+d)H_2+1\right).
	\end{align*}
	As a consequence, we have
	\begin{align} \label{RegulariseProof2}
		&(\Psi_{\bar\alpha,p,m}^{\bar\kappa,g}( r, \bar s,u,\bar t))^{\frac{1}{2}}\notag\\\leq& (C_1C_2)^{(pmd+p\vert\alpha\vert)}\left(\prod\limits_{\ell=1}^d(2p\vert\alpha^{(\ell)}\vert)!\right)^{\frac{1}{4}}\\&\notag\times\Big[r^{H_1-\frac{1}{2}-\gamma}\bar u^{H_2-\frac{1}{2}-\gamma}\Big(\frac{\vert u-\bar u\vert^{\gamma}}{(u\bar u)^{\gamma}} +\frac{\vert r-\bar r\vert^{\gamma}}{(r\bar r)^{\gamma}} \Big)\Big]^{p\sum_{j=1}^{m}\ve_{j}} \\&\notag\times\frac{(\bar s-r)^{p(H_1-\frac{1}{2}-\gamma)\sum_{j=1}^{m} \ve_{j}-p(\vert\alpha\vert+md)H_1+pm}}{\Gamma\left(2p(H_1-\frac{1}{2}-\gamma)\sum\limits_{j=1}^{m} \ve_{j}- 2p(\vert\alpha\vert+md)H_1+2pm+1\right)^{1/2}}
		\\&\notag\times\frac{(\bar t-u)^{p(H_2-\frac{1}{2}-\gamma)\sum_{j=1}^{m} \ve_{j}-p(\vert\alpha\vert+md)H_2+pm}}{\Gamma\left(2p(H_2-\frac{1}{2}-\gamma)\sum\limits_{j=1}^{m} \ve_{j}- 2p(\vert\alpha\vert+md)H_2+2pm+1\right)^{1/2}},
	\end{align}
	which gives the desired result by combining \eqref{RegulariseProof1} and \eqref{RegulariseProof2}.
\end{proof}
\begin{prop}\label{RegulariseProp2}
	Let $(W_{r,u}^H,(r,u)\in[0,T]^2)$ be a $d$-dimensional fractional Brownian sheet with Hurst index $H=(H_1,H_2)\in(0,\frac{1}{2})^2$ under $(\Omega,\mathcal{F},\Pb)$. Let $f$ and $\kappa$ be given as in \eqref{RegulariseDef1} and \eqref{RegulariseDef2} respectively. Let $m,\,p\in\N$, $\alpha\in(\N_0^m)^d$ be a multi-index, $\esp>0$, $\esp\leq r<\bar s\leq T$, $\esp\leq u<\bar t\leq T$, 
	\begin{align*}
		\kappa_j(s,t)= K_H(s_j,t_j, r, u)^{\ve_j}
	\end{align*} 
	for every $s=(s_1,\ldots,s_m)$, $t=(t_1,\ldots,t_m)$,  $j\in\{1,\dots,m\}$ and $(\ve_1,\ldots,\ve_m)\in\{0,1\}^m$. Suppose
	\begin{align}\label{RegulariseCond2}
		\max\{H_1,H_2\}<\frac{1}{2(d-1+\vert\alpha_{j}\vert)}
	\end{align}
	for all $j=1,\ldots,m$. Then there exists a universal constant $C=C_{\esp,H,T,d}$ (depending on $H$, $T$ and $d$, but independent of $m$, $p$, $\{f_i\}_{i=1,\ldots,m}$ and $\alpha$) such that  
	\begin{align*}
		&\Big\vert\E\Big[\Big(\int_{\nabla^{m}_{ r, \bar s}}\int_{\nabla^{m}_{ u, \bar t}}\prod\limits_{j=1}^mD^{\alpha_j}f_j(s_j,t_j,W^H_{s_j,t_j})\kappa_j(s_j,t_j)\,\mathrm{d}t_j\mathrm{d}s_j\Big)^p\Big]\Big\vert\\\leq& C^{p(md+\vert\alpha\vert)}\left(\prod\limits_{i=1}^d(2p\vert\alpha^{(i)}\vert)!\right)^{\frac{1}{4}}\prod\limits_{j=1}^m\Vert f_j\Vert^p_{L^1_{\infty}}\Big( r^{H_1-\frac{1}{2}} u^{H_2-\frac{1}{2}} \Big)^{p\sum_{j=1}^{m}\ve_{j}}  \\&\quad\times\frac{(\bar s- r)^{p(H_1-\frac{1}{2})\sum_{j=1}^{m} \ve_{j}-p(\vert\alpha\vert+md)H_1+pm}}{\Gamma\left(2p(H_1-\frac{1}{2})\sum\limits_{j=1}^{m} \ve_{j}- 2p(\vert\alpha\vert+md)H_1+2pm+1\right)^{1/2}}
		\\&\quad\times\frac{(\bar t- u)^{p(H_2-\frac{1}{2})\sum_{j=1}^{m} \ve_{j}-p(\vert\alpha\vert+md)H_2+pm}}{\Gamma\left(2p(H_2-\frac{1}{2})\sum\limits_{j=1}^{m} \ve_{j}- 2p(\vert\alpha\vert+md)H_2+2pm+1\right)^{1/2}}.
	\end{align*}
\end{prop}
\begin{proof}
	The proof follows directly from Theorem \ref{RegulariseTheo} and is similar to that of the preceding Proposition.
\end{proof}

\begin{prop}\label{RegulariseProp3}
	Let $(W_{r,u}^H,(r,u)\in[0,T]^2)$ be a $d$-dimensional fractional Brownian sheet with Hurst index $H=(H_1,H_2)\in(0,\frac{1}{2})^2$ under $(\Omega,\mathcal{F},\Pb)$. Let $f$ and $\kappa$ be given as in \eqref{RegulariseDef1} and \eqref{RegulariseDef2} respectively. Let $m,\ell,p\in\N$, $\alpha\in(\N_0^d)^{m+\ell}$ be a multi-index, $\esp>0$, $\esp\leq \bar r<r<\bar s\leq T$, $\esp\leq \bar u<u<\bar t\leq T$, 
	\begin{align*}
		\kappa_j(s,t)= K_H(s_j,t_j,\bar r,\bar u)^{\ve_j}
	\end{align*} 
	for every $s=(s_1,s_2,\ldots,s_{m+\ell})$, $t=(t_1,t_2,\ldots,t_{m+\ell})$, $j\in\{1,2,\dots,m+\ell\}$ and $(\ve_1,\ve_2,\ldots,\ve_{m+\ell})\in\{0,1\}^{m+\ell}$. Define
	\begin{align*}
		\Xi^{m,\ell}_{\bar r,r,\bar s}:=\{(s_1,\ldots,s_{m+\ell}):\,(s_1,\ldots,s_m)\in\nabla_{r,\bar s}^{m}\text{ and }(s_{m+1},\ldots,s_{m+\ell})\in\nabla^{\ell}_{\bar r,s_m}\}
	\end{align*}
	and
	\begin{align*}
		\Delta^{m,\ell}_{\bar u,u,\bar t}=&\{(t_1\ldots,t_{m+\ell}):\,\bar u<t_{m+\ell}<\ldots<t_{m+1}<u<t_{m}<\ldots<t_{1}<\bar t\}.
	\end{align*} 	
	Suppose
	\begin{align}\label{RegulariseCond3}
		\max\{H_1,H_2\}<\frac{1}{2(d-1+\vert\alpha_{j}\vert)}
	\end{align}
	for all $j=1,\ldots,m+\ell$.
	Then there exists a universal constant $C=C_{\ve,H,T,d}$ (depending on $\esp$, $H$, $T$ and $d$, but independent of $m$, $p$, $\{f_i\}_{i=1,\ldots,m}$ and $\alpha$) such that  
	\begin{align*}
		&\Big\vert\E\Big[\Big(\int_{\Xi^{m,\ell}_{\bar r,r,s}}\int_{\Delta^{m,\ell}_{\bar u, u, t}}\prod\limits_{j=1}^{m+\ell}D^{\alpha_j}f_j(s_j,t_j,W^H_{s_j,t_j})\kappa_j(s_j,t_j) \mathrm{d}t_{m+\ell} \ldots\mathrm{d}t_1\mathrm{d}s_{m+\ell}\ldots\mathrm{d}s_1\Big)^p\Big]\Big\vert\\\leq& C^{p[(m+\ell)d+\vert\alpha\vert]}\left(\prod\limits_{l=1}^d(2p\vert\alpha^{(l)}\vert)!\right)^{1/4}\prod\limits_{j=1}^{m+\ell}\Vert f_j\Vert^p_{L^1_{\infty}}\Big(\bar r^{H_1-\frac{1}{2}}\bar u^{H_2-\frac{1}{2}} \Big)^{p\sum_{j=1}^{m+\ell}\ve_{j}}  \\& \times\frac{(\bar s-\bar r)^{p(H_1-\frac{1}{2})\sum_{j=1}^{m+\ell} \ve_{j}-p(\vert\alpha\vert+(m+\ell)d)H_1+p(m+\ell)}}{\Gamma\left(2p(H_1-\frac{1}{2})\sum\limits_{j=1}^{m+\ell} \ve_{j}- 2p(\vert\alpha\vert+(m+\ell)d)H_1+2p(m+\ell)+1\right)^{1/2}}
		\\& \times\sum\limits_{\pi\in\widehat{\mathcal{P}}_{m+\ell\vert 2p}}\Big\{\frac{(\bar t-u)^{\frac{1}{2}(H_2-\frac{1}{2})\sum_{k=1}^{2pm} \bar\ve^{\pi}_{[k]}-(\frac{1}{2}\sum_{k=1}^{2pm}\vert\bar\alpha^{\pi}_{[k]}\vert+pmd)H_2+pm}}{\Gamma\left((H_2-\frac{1}{2})\sum\limits_{k=1}^{2pm} \bar\ve^{\pi}_{[k]}- (\sum\limits_{k=1}^{2pm}\vert\bar\alpha^{\pi}_{[k]}\vert+2pmd)H_2+2pm+1\right)^{1/2}}\\& \times\frac{\vert u-\bar u\vert^{\frac{1}{2}(H_2-\frac{1}{2})\sum_{j=1}^{2p\ell} \bar\ve^{\pi}_{[2pm+j]}-(\frac{1}{2}\sum_{j=1}^{2p\ell}\vert\bar\alpha^{\pi}_{[2pm+j]}\vert+p\ell d)H_2+p\ell}}{\Gamma\left((H_2-\frac{1}{2})\sum\limits_{j=1}^{2p\ell} \bar\ve^{\pi}_{[2pm+j]}- (\sum\limits_{j=1}^{2p\ell}\vert\bar\alpha^{\pi}_{[2pm+j]}\vert+2p\ell d)H_2+2p\ell+1\right)^{1/2}}\Big\}.
	\end{align*}
	where $\bar\ve_j=\ve_{j-i(m+\ell)}$, $\bar\alpha_j=\alpha_{j-i(m+\ell)}$  for $i=0,1,\ldots,p-1$, $j=1+i(m+\ell),\ldots,(i+1)(m+\ell)$,  $\bar\ve_{[j]}=\ve_{j-ip(m+\ell)}$, $\bar\alpha_{[j]}=\alpha_{j-ip(m+\ell)}$ for $i=0,1$, $j=1+ip(m+\ell),\ldots,(i+1)p(m+\ell)$, 
	$\bar\ve^{\pi}_{[j]}=\ve^{}_{[\pi^{-1}(j)]}$,  $\bar\alpha^{\pi}_{[j]}=\alpha^{}_{[\pi^{-1}(j)]}$ for all $j\in\{0,1,2,\ldots,2p(m+\ell)\}$.		
\end{prop} 
\begin{proof}
	We first observe that
	\begin{align*}
		&\Big(\int_{\Xi^{m,\ell}_{\bar r,r,\bar s}}\int_{\Delta^{m,\ell}_{\bar u, u, \bar t}}\prod\limits_{j=1}^{m+\ell}D^{\alpha_j}f_j(s_j,t_j,W^H_{s_j,t_j})\kappa_j(s_j,t_j) \mathrm{d}t_{m+\ell} \ldots\mathrm{d}t_1\mathrm{d}s_{m+\ell}\ldots\mathrm{d}s_1\Big)^p\\=&  \int_{(\Xi_{\bar r,r,\bar s}^{m,\ell})^p} \int_{(\Delta_{\bar u,u,\bar t}^{m,\ell})^p}\prod\limits_{j=1}^{p(m+\ell)}D^{\bar\alpha_{j}}\bar f_{ j}(s_j,t_j,W^H_{s_j,t_j})\bar\kappa_{ j}(s_j,t_j)\mathrm{d}t_1\ldots\mathrm{d}t_{p(m+\ell)}\\&\qquad\qquad\qquad\times\mathrm{d}s_1\ldots\mathrm{d}s_{p(m+\ell)}\\=&\int_{(\nabla_{\bar r,\bar s}^{m+\ell})^p} \int_{(\nabla_{\bar u,\bar t}^{m+\ell})^p}\mathbf{1}_{(\Xi_{\bar r,r,\bar s}^{m,\ell})^p\times(\Delta_{\bar u,u,\bar t}^{m,\ell})^p}(s,t)\prod\limits_{j=1}^{p(m+\ell)}D^{\bar\alpha_{j}}\bar f_{ j}(s_j,t_j,W^H_{s_j,t_j})\bar\kappa_{ j}(s_j,t_j)\\&\qquad\qquad\qquad\times\mathrm{d}t_1\ldots\mathrm{d}t_{p(m+\ell)}\mathrm{d}s_1\ldots\mathrm{d}s_{p(m+\ell)}\\=&\Lambda^{\bar\kappa\bar f,g}_{\bar\alpha,p,m+\ell}(\bar r,\bar s,\bar u,\bar t,z),
	\end{align*}
	where $s=(s_1,\ldots,s_{p(m+\ell)})$, $t=(t_1,\ldots,t_{p(m+\ell)})$,
	$g=\mathbf{1}_{(\Xi_{\bar r,r,\bar s}^{m+\ell})^p\times(\Delta_{\bar u,u,\bar t}^{m+\ell})^p}$,
	$\bar f_j=f_{j-i(m+\ell)}$, $\bar \alpha_j=\alpha_{j-i(m+\ell)}$, $\bar \ve_j=\ve_{j-i(m+\ell)}$ for every $i\in\{0,1,\ldots,p-1\}$,  $j\in\{1+i(m+\ell),\ldots,(1+i)(m+\ell)\}$ and
	\begin{align*}
	\bar\kappa_j(s,t)= K_H(s_j,t_j,\bar r,\bar u)^{\bar\ve_j},\quad\forall\,j\in\{1,2,\ldots,p(m+\ell)\}.
	\end{align*}
	It follows from Theorem \ref{RegulariseTheo} that
	\begin{align*}
		\vert\E[\Lambda^{\bar\kappa\bar f,g}_{\bar\alpha,p,m+\ell}(\bar r,\bar s,\bar u,\bar t,z)]\vert\leq&\E\left[\vert\Lambda^{\bar\kappa\bar f,g}_{\bar\alpha,p,m+\ell}(\bar r,\bar s,\bar u,\bar t,z)\vert^2\right]\\\leq& C_1^{p(m+\ell)/2}(\Psi_{\bar\alpha,p,m+\ell}^{\bar\kappa,g}(\bar r,\bar s,\bar u, \bar t))^{1/2}\prod\limits_{j=1}^{m+\ell}\Vert f_j\Vert^p_{L^1_{\infty}},
	\end{align*}
	where
	\begin{align*}
		&\Psi_{\bar\alpha,p,m+\ell}^{\bar\kappa,g}(\bar r,\bar s,\bar u, \bar t)\\	
		:=&\Big(\prod\limits_{ l=1}^d\sqrt{(2\vert\bar\alpha^{( l)}\vert)!}\Big)\sum\limits_{\sigma,\pi\in\widehat{\mathcal{P}}_{m+\ell\vert2p}} \int_{\nabla_{\bar r,\bar s}^{2p(m+\ell)} }\int_{\nabla_{\bar u,\bar t}^{2p(m+\ell)}} \widetilde g_{\sigma,\pi}(s,t)\prod\limits_{j=1}^{2p(m+\ell)}\vert \bar\kappa_{[j]}(s_{\sigma(j)},t_{\pi(j)})\vert\\&\qquad\times\frac{\mathrm{d}t_1\ldots\mathrm{d}t_{2p(m+\ell)}\mathrm{d}s_{1}\ldots\mathrm{d}s_{2p(m+\ell)}}{\prod\limits_{j=1}^{2p(m+\ell)}\vert s_{j}-s_{j+1}\vert^{H_1(\vert\alpha^{\sigma}_{([j])}\vert+d)}\vert t_{j}-t_{j+1}\vert^{ H_2(\vert\alpha^{\pi}_{([j])}\vert+d)}},
	\end{align*}
$\alpha^{\sigma}_{[j]}=\alpha_{[\sigma^{-1}(j)]}$, $\alpha^{\pi}_{[j]}=\alpha_{[\pi^{-1}(j)]}$ and
\begin{align*}
\widetilde g_{\sigma,\pi}(s,t)=&g(s_{\sigma(1)},\ldots,s_{\sigma(p(m+\ell))},t_{\pi(1)},\ldots,s_{\pi(p(m+\ell))})\\&\times g(s_{\sigma(p(m+\ell)+1)},\ldots,s_{\sigma(2p(m+\ell))},t_{\pi(p(m+\ell)+1)},\ldots,s_{\pi(2p(m+\ell))}).
\end{align*}	
We define
	\begin{align*}
	\Delta^{2pm,2p\ell,\pi}_{\bar u,u,\bar t}=\{(t_1,t_2,\ldots,t_{2p(m+\ell)}):\,(t_{\pi^{-1}(1)},t_{\pi^{-1}(2)},\ldots,t_{\pi^{-1}(2p(m+\ell))})\in\Delta^{2pm,2p\ell}_{\bar u,u,\bar t}\}.
	\end{align*}
	Since 
	\begin{align*}
	(\Xi_{\bar r,r,\bar s}^{m,\ell})^{2p}\cap\nabla^{2p(m+\ell),\sigma}_{\bar r,\bar s}\subset(\nabla_{\bar r,\bar s}^{m+\ell})^{2p}\cap\nabla^{2p(m+\ell),\sigma}_{\bar r,\bar s}=\nabla^{2p(m+\ell),\sigma}_{\bar r,\bar s}
	\end{align*}
	and
	\begin{align*}
	(\Delta^{m,\ell}_{\bar u,u,\bar t})^{2p}\cap\nabla^{2p(m+\ell),\pi}_{\bar u,\bar t}\subset\Delta^{2pm,2p\ell,\pi}_{\bar u,u,\bar t}
	\end{align*}
	for every $\sigma,\pi\in\widehat{\mathcal{P}}_{m+\ell\vert2p}$, we have
	\begin{align*}
	&\Psi_{\bar\alpha,p,m+\ell}^{\bar\kappa,g}(\bar r,\bar s,\bar u, \bar t)\\	
	\leq&\Big(\prod\limits_{ l=1}^d\sqrt{(2\vert\bar\alpha^{( l)}\vert)!}\Big)\sum\limits_{\sigma,\pi\in\widehat{\mathcal{P}}_{m+\ell\vert2p}} \int_{\nabla_{\bar r,\bar s}^{2p(m+\ell)} }\int_{\Delta_{\bar u,u,\bar t}^{2pm,2p\ell}}  \prod\limits_{j=1}^{2p(m+\ell)}\vert \bar\kappa_{[j]}(s_{\sigma(j)},t_{\pi(j)})\vert\\&\qquad\times\frac{\mathrm{d}t_1\ldots\mathrm{d}t_{2p(m+\ell)}\mathrm{d}s_{1}\ldots\mathrm{d}s_{2p(m+\ell)}}{\prod\limits_{j=1}^{2p(m+\ell)}\vert s_{j}-s_{j+1}\vert^{H_1(\vert\alpha^{\sigma}_{([j])}\vert+d)}\vert t_{j}-t_{j+1}\vert^{ H_2(\vert\alpha^{\pi}_{([j])}\vert+d)}}.
	\end{align*}
	Since
	\begin{align*}
		(H_1-\frac{1}{2})\ve_{([j])}-(\vert\alpha_{([j])}\vert+d)H_1>-1,\text{ }	(H_2-\frac{1}{2})\ve_{([j])}-(\vert\alpha_{([j])}\vert+d)H_2>-1,
	\end{align*}
	 we deduce from Lemma \ref{ExistEstLem4} that
	\begin{align*}
		&\Psi_{\bar\alpha,p,m+\ell}^{\bar\kappa,g}(\bar r,\bar s,\bar u,\bar t)\\	
		\leq&\Big(\prod\limits_{l=1}^d\sqrt{(2\vert\widetilde\alpha^{(l)}\vert)!}\Big)C^{m+\ell}_{2}\Pi_{0,2p(m+\ell)}(H_1)(\bar r^{H_1-\frac{1}{2}}\bar u^{H_2-\frac{1}{2}})^{2p\sum_{j=1}^{m+\ell}\ve_j}\\&\notag\times(\bar s-\bar r)^{2p(H_1-\frac{1}{2})\sum_{j=1}^{m+\ell}\ve_j-2(\vert\bar\alpha\vert+p(m+\ell)d)H_1+2p(m+\ell)}\\&\notag\times\sum\limits_{\pi\in\widehat{\mathcal{P}}_{m+\ell\vert 2p}}\widetilde\Pi_{p,m,\ell}(\pi,H_2)\Big\{(t- u)^{(H_2-\frac{1}{2})\sum_{k=1}^{2pm}\bar\ve^{\pi}_{[k]}-2(\sum_{k=1}^{2m}\vert\bar\alpha^{\pi}_{[k]}\vert+pmd)H_2+2pm} \\&\notag\quad\times\vert u-\bar u\vert^{(H_2-\frac{1}{2})\sum_{j=1}^{2p\ell}\bar\ve^{\pi}_{[2pm+j]}-2(\sum_{j=1}^{2p\ell}\vert\bar\alpha^{\pi}_{[2pm+j]}\vert+p\ell d)H_2+2p\ell}\Big\},
	\end{align*} 
	where $\bar\ve^{\pi}_{[j]}=\bar\ve_{[\pi^{-1}(j)]}$, $\bar\alpha^{\pi}_{[j]}=\bar\alpha_{[\pi^{-1}(j)]}$ for all $j\in\{1,2,\ldots,2p(m+\ell)\}$,
	\begin{align*}
		&\Pi_{0,2p(m+\ell)}(H_1)\\:=&\frac{\prod\limits_{j=2}^{2p(m+\ell)}\Gamma\left(1-(\vert\bar\alpha_{[j]}\vert+d)H_1\right)  \prod\limits_{j=1}^{2p(m+\ell)}\Gamma\left((H_1-\frac{1}{2}) \bar\ve_{[j]}-(\vert\bar\alpha_{[j]}\vert+d)H_1+1  \right)}{\Gamma\left(2p(H_1-\frac{1}{2})\sum\limits_{j=1}^{m+\ell} \ve_{j}- 2(\vert\bar\alpha\vert+p(m+\ell)d)H_1+2p(m+\ell)+1\right) }   
	\end{align*}
	and
	\begin{align*}
		&\widetilde\Pi_{p,m,\ell}(\pi,H_2)\\=&\frac{\prod\limits_{j=2}^{2p(m+\ell)}\Gamma\left(1-(\vert\bar\alpha_{[j]}\vert+d)H_2\right) }{\Gamma\left((H_2-\frac{1}{2})\sum\limits_{k=1}^{2pm} \bar\ve^{\pi}_{[k]}- (\sum\limits_{k=1}^{2pm} \vert\bar\alpha^{\pi}_{[k]}\vert+2pmd)H_2+2pm+1)\right)}\\&\times\frac{\prod\limits_{j=1}^{2p(m+\ell)}\Gamma\left((H_2-\frac{1}{2}) \bar\ve_{[j]}-(\vert\bar\alpha_{[j]}\vert+d)H_2+1  \right)}{\Gamma\left((H_2-\frac{1}{2})\sum\limits_{j=1}^{2p\ell} \bar\ve^{\pi}_{[2pm+j]}- (\sum\limits_{j=1}^{2p\ell}\vert\bar\alpha^{\pi}_{[2pm+j]}\vert+2p\ell d)H_2+2p\ell+1\right)}.
	\end{align*}
	Observe that there exists a finite constant $C_3$ such that
	\begin{align}\label{RegularFinalEst}
		C_3^{2p(m+\ell)}\geq&\prod\limits_{j=2}^{2p(m+\ell)}\Gamma\left(1-(\vert\bar\alpha_{[j]}\vert+d)H_1\right) \prod\limits_{j=2}^{2p(m+\ell)}\Gamma\left(1-(\vert\bar\alpha_{[j]}\vert+d)H_2\right) \\&\notag\times\prod\limits_{j=1}^{2p(m+\ell)}\Gamma\left((H_1-\frac{1}{2}) \bar\ve_{[j]}-(\vert\bar\alpha_{[j]}\vert+d)H_1+1  \right) \\&\notag\times\prod\limits_{j=1}^{2p(m+\ell)}\Gamma\left((H_2-\frac{1}{2}) \bar\ve_{[j]}-(\vert\bar\alpha_{[j]}\vert+d)H_2+1  \right).
	\end{align}
	Then, using \eqref{RegularFinalEst}, $\vert\bar\alpha^{(l)}\vert=p\vert\alpha^{(l)}\vert$, $\bar{\alpha}^{(l)}_j=\alpha^{(l)}_{(j)}$ for every $l\in\{1,\ldots,d\}$, $j\in\{1,\ldots,p(m+\ell)\}$ and $\sqrt{a_1+\ldots+a_q}\leq\sqrt{a_1}+\ldots+\sqrt{a_q}$ for arbitrary non-negative numbers $a_1,\,\ldots,\,a_q$ we obtain
	\begin{align*}
		&(\Psi_{\bar\alpha,p,m+\ell}^{\bar\kappa,g}(\bar r,\bar u,\bar s,\bar t,z))^{1/2}\\	
		\leq&C^{(m+\ell)/2}_{2}C_3^{p(m+\ell)}\Big(\prod\limits_{l=1}^d(2p\vert\alpha^{(l)}\vert)!\Big)^{1/4} (\bar r^{H_1-\frac{1}{2}}\bar u^{H_2-\frac{1}{2}})^{p\sum_{j=1}^{m+\ell}\ve_j}\\&\notag\times\frac{(\bar s-\bar r)^{p(H_1-\frac{1}{2})\sum_{j=1}^{m+\ell}\ve_j-p(\vert\alpha\vert+(m+\ell)d)H_1+p(m+\ell)}}{\Gamma\left(2p(H_1-\frac{1}{2})\sum\limits_{j=1}^{m+\ell} \ve_{j}- 2p(\vert\alpha\vert+(m+\ell)d)H_1+2p(m+\ell)+1\right)^{1/2}}\\&\times\sum\limits_{\pi\in\widehat{\mathcal{P}}_{m+\ell\vert 2p}}\Big\{\frac{(\bar t- u)^{\frac{1}{2}(H_2-\frac{1}{2})\sum_{k=1}^{2pm}\bar\ve^{\pi}_{[k]}-(\frac{1}{2}\sum_{k=1}^{2pm} \vert\bar\alpha^{\pi}_{[k]}\vert+pmd)H_2+pm}}{\Gamma\left((H_2-\frac{1}{2})\sum\limits_{k=1}^{2pm} \ve^{\bar\pi}_{[k]}- (\sum\limits_{k=1}^{2pm} \vert\alpha^{\bar\pi}_{[k]}\vert+2pmd)H_2+2pm+1)\right)^{1/2}}\\&\times\frac{\vert u-\bar u\vert^{\frac{1}{2}(H_2-\frac{1}{2})\sum_{j=1}^{2p\ell}\bar\ve^{\pi}_{[2pm+j]}-(\frac{1}{2}\sum_{j=1}^{2p\ell} \vert\bar\alpha^{\pi}_{[2pm+j]}\vert+p\ell d)H_2+p\ell}}{\Gamma\left((H_2-\frac{1}{2})\sum\limits_{j=1}^{2p\ell} \bar\ve^{\pi}_{[2pm+j]}- (\sum\limits_{j=1}^{2p\ell} \vert\bar\alpha^{\pi}_{[2pm+j]}\vert+2p\ell d)H_2+2p\ell+1\right)^{1/2}}\Big\}.
	\end{align*}
	This ends the proof.
\end{proof}

\section{Existence and uniqueness of strong solutions}
{The purpose of this section is to prove existence and uniqueness for a time-inhomogeneous SDE on the plane with additive $d$-dimensional fractional Brownian sheet $W^H$ with Hurst parameter $H=(H_1,H_2)\in(0,1/2)^2$. More precisely, we consider the followind SDE}
\begin{align}\label{eqmainProbS4}
	X_{s,t}=x+\int_0^s\int_0^tb(s_1,t_1,X_{s_1,t_1})\,\mathrm{d}t_1\mathrm{d}s_1+W^H_{s,t},\quad(s,t)\in[0,T]^2, 
\end{align}
where $x\in\R^d$,  {$b:\,[0,T]^2\times\R^d\to\R^d$ belongs to $L^1_{\infty} $}.  
Here is the main result of the current section.  {It also constitutes the main result of this paper.} 
{
	\begin{thm}\label{TheoEqmainPbS4}
		Let $b\in L^1_{\infty}$ and suppose $H_1+H_2<\frac{1}{2(d+1)}$, $d\geq1$. Then there exists a unique strong solution $X^x=(X^x_{s,t},\,(s,t)\in[0,T]^2)$ of \eqref{eqmainProbS4}. Moreover, for any $\esp>0$, $\esp\leq r<s\leq T$ and $\esp\leq u<t\leq T$, $X^x_{s,t}$ is Malliavin differentiable in the direction of $W_{r,u}$, where $W$ is the $d$-dimensional Brownian sheet in \eqref{Bsheet}. 
	\end{thm}
}
The proof of the above result relies on the following steps:
\begin{enumerate}
	\item We start with the construction of a weak solution $X^x$ to \eqref{eqmainProbS4} using Girsanov's theorem. This solution is not adapted to the filtration generated by $W^H$ in general.  {Then we show that joint uniqueness in law holds for \eqref{eqmainProbS4}.}
	\item Next, we consider a sequence $(b_n)_{n\in\N}$ of functions in $C^{\infty}_c([0,T]^2\times\R^d,\R^d)$  {(the space of smooth functions with compact support)} such that $b_n(s,t,y)$ converges to $b(s,t,y)$ for a.e. $(s,t,y)\in[0,T]^2\times\R^d$ with $\sup_{n\geq0}\Vert b_n\Vert_{L^1_{\infty}}<\infty$.  By standard results, one can show that for each smooth coefficient $b_n$, $n\in\N$, there exists a unique strong solution $X^{x,n}$ to the SDE
	\begin{align}\label{eqmainProbS4Xn}
		X_{s,t}=x+\int_0^s\int_0^tb_n(s_1,t_1,X_{s_1,t_1})\,\mathrm{d}t_1\mathrm{d}s_1+W^H_{s,t},\quad(s,t)\in[0,T]^2.
	\end{align}
	{We show } that for each $(s,t)\in[0,T]^2$, the sequence  {$X_{s,t}^{x,n}$} converges weakly to the conditional expectation $\E[X^x_{s,t}\vert\mathcal{F}_{s,t}]$ in the space $L^2(\Omega,\mathcal{F}_{s,t})$ of square integrable $\mathcal{F}_{s,t}$-measurable random variables.
	\item  {It also holds} that for each $n\in\N$ the strong solution $(X^{x,n}_{s,t},(s,t)\in[0,T]^2)$ is Malliavin differentiable, and that the Malliavin derivative $(D_{r,u}X^{x,n}_{s,t},(r,u)\in[0,s]\times[0,t])$ with respect to $W$ in \eqref{Bsheet} satisfies
	\begin{align*}
		D_{r,u}X^{x,n}_{s,t}=K_H(s,t,r,u)I_d+\int_r^s\int_u^t\,b_n^{\prime}(s_1,t_1,X^{x,n}_{s_1,t_1})D_{r,u}X^{x,n}_{s_1,t_1}\,\mathrm{d}t_1\mathrm{d}s_1,
	\end{align*}
	where $b_n^{\prime}$ is the Jacobian of $b_n$ and $I_d$ the identity matrix in $\R^{d\times d}$. In the third step we apply a compactness criterion based on Malliavin calculus to show that for every $(s,t)\in(0,T]^2$ the set of random variables $\{X^{x,n}_{s,t},\,n\in\N\}$ is relatively compact in $L^2(\Omega)$, which then implies that $(X^{x,n}_{s,t},\,n\in\N)$ converges strongly in $L^2(\Omega,\mathcal{F}_{s,t})$ to $\E[X^x_{s,t}\vert\mathcal{F}_{s,t}]$. Moreover $\E[X^x_{s,t}\vert\mathcal{F}_{s,t}]$ is Malliavin differentiable as a consequence of the compactness criterion.
	\item In the fourth step we  {prove} that $\E[X^x_{s,t}\vert\mathcal{F}_{s,t}]=X^x_{s,t}$, which entails that $X^x_{s,t}$ is $\mathcal{F}_{s,t}$-measurable and thus a stong solution on a specific probability space.  {The pathwise uniqueness then follows from a dual Yamada-Watanabe argument.}
\end{enumerate}
The first step of the scheme follows from Lemma  \eqref{GirsanovLemEst} and Theorem \eqref{GirsanovTheo}. Precisely, Let $(\Omega,\widetilde{\mathcal{F}},\widetilde{\Pb})$ be some given probability space and $\widetilde{W}^H$ be a $d$-dimensional fractional Brownian sheet with Hurst parameter $H=(H_1,H_2)\in(0,\frac{1}{2})^2$ defined on this space. Set $X^x_{s,t}:=x+\widetilde W^H_{s,t}$, $(s,t)\in[0,T]^2$, $x\in\R^d$. Define $\ \theta_{s,t}:=(K^{-1}_H(\int_{0}^{\cdot}\int_0^{\cdot}b(s_1,t_1,X_{s_1,t_1})\,\mathrm{d}t_1\mathrm{d}s_1))(s,t)$ and  the Dol\'eans-Dade exponential 
\begin{align*}
	&\mathcal{Z}_{s,t}:=\mathcal{E}(\theta_{\cdot,\cdot})_{s,t}=\exp\Big(\int_0^s\int_0^t\theta_{s_1,t_1}\cdot\mathrm{d}W_{s_1,t_1}-\frac{1}{2}\int_0^s\int_0^t\vert\theta_{s_1,t_1}\vert^2\,\mathrm{d}t_1\mathrm{d}s_1\Big),\\&\quad(s,t)\in[0,T]^2.
\end{align*}

The next lemma shows that the conditions of Theorem \ref{GirsanovTheo} hold. 
\begin{lem}\label{GirsanovLemEst}
	Let $b\in L^1_{\infty} $ and  $(\widetilde{W}_{s,t}^H,(s,t)\in[0,T]^2)$ be a $d$-dimensional fractional Brownian sheet with Hurst index $H=(H_1,H_2)\in(0,\frac{1}{2})^2$ defined on  $(\Omega,\widetilde{\mathcal{F}},\widetilde\Pb)$. Suppose $H_1+H_2<\frac{1}{2(1+d)}$. Then for every $\mu\in\R$, we have
	\begin{align}\label{EqNovikov}	
		&\widetilde\E\Big[\exp\Big\{\mu\int_0^T\int_0^T\Big\vert K^{-1}_H\Big(\int_0^{\cdot}\int_0^{\cdot}\vert b(s_1,t_1,\widetilde{W}_{s_1,t_1}^H)\vert\,\mathrm{d}t_1\mathrm{d}s_1\Big)(s,t)\Big\vert^2\mathrm{d}t\mathrm{d}s\Big\}\Big]\\\leq& C_{H,d,\mu,T}(\Vert b\Vert_{L^{1}_{\infty}}) \notag
	\end{align}
	for some increasing continuous function $C_{H,d,\mu,T}$ depending only on $H$, $d$, $T$ and $\mu$.
	In particular
	\begin{align*}
		&\widetilde\E\Big[\mathcal{E}\Big(\int_0^T\int_0^T K^{-1}_H\Big(\int_0^{\cdot}\int_0^{\cdot}\vert b(s_1,t_1,\widetilde{W}_{s_1,t_1}^H)\vert\,\mathrm{d}t_1\mathrm{d}s_1\Big)^{\ast}(s,t)\mathrm{d}W_{s,t}\Big)^p\Big]\\\leq& C_{H,d,p,T}(\Vert b\Vert_{L^{1}_{\infty}}),
	\end{align*}
	for any $p>1$ and for some increasing continuous function $C_{H,d,p,T}$, where $W$  denotes a $d$-dimensional Brownian sheet, $\widetilde\E$ is the expectation under $\widetilde\Pb$ and $\ast$ denotes the transposition.
\end{lem} 
\begin{proof} 
	Assume without loss of generality that $b\in L^{\infty}([0,T]^2;L^1(\R^d;\R)).$
	Denote by $\theta_{s,t}:=K^{-1}_H\Big(\int_0^{\cdot}\int_0^{\cdot}\vert b(s_1,t_1,\widetilde{W}_{s_1,t_1}^H)\vert\,\mathrm{d}t_1\mathrm{d}s_1\Big)(s,t)$. We deduce from (2.7) 
	and Fubini's theorem that
	\begin{align*}
		\notag\vert\theta_{s,t}\vert=&\left\vert s^{H_1-\frac{1}{2}}t^{H_2-\frac{1}{2}}I^{\frac{1}{2}-H_1,\frac{1}{2}-H_2}s^{\frac{1}{2}-H_1}t^{\frac{1}{2}-H_2}\vert b(s,t,\widetilde{W}_{s,t}^H)\vert\right\vert\\=\notag&\frac{s^{H_1-\frac{1}{2}}t^{H_2-\frac{1}{2}}}{\Gamma(\frac{1}{2}-H_1)\Gamma(\frac{1}{2}-H_2)}\int_0^s\int_0^t(s-s_1)^{-\frac{1}{2}-H_1}(t-t_1)^{-\frac{1}{2}-H_2}s_1^{\frac{1}{2}-H_1}t_1^{\frac{1}{2}-H_2}
		\\&\notag\qquad\qquad\qquad\qquad\qquad\qquad\qquad\times\vert b(s_1,t_1,\widetilde{W}_{s_1,t_1}^H)\vert\,\mathrm{d}t_1\mathrm{d}s_1\\=\notag&\frac{s^{-\frac{1}{2}-H_1}t^{-\frac{1}{2}-H_2}}{\Gamma(\frac{1}{2}-H_1)\Gamma(\frac{1}{2}-H_2)}\int_0^s\int_0^t(1-\frac{s_1}{s})^{-\frac{1}{2}-H_1}(1-\frac{t_1}{t})^{-\frac{1}{2}-H_2}(\frac{s_1}{s})^{\frac{1}{2}-H_1}\\&\notag\qquad\qquad\qquad\qquad\qquad\qquad\times(\frac{t_1}{t})^{\frac{1}{2}-H_2}\vert b(s_1,t_1,\widetilde{W}_{s_1,t_1}^H)\vert\,\mathrm{d}t_1\mathrm{d}s_1\\= \notag&\frac{s^{-\frac{1}{2}-H_1}t^{-\frac{1}{2}-H_2}}{\Gamma(\frac{1}{2}-H_1)\Gamma(\frac{1}{2}-H_2)}\int_0^s\int_0^t(1-\frac{s_1}{s})^{-\frac{1}{2}-H_1}(1-\frac{t_1}{t})^{-\frac{1}{2}-H_2}(\frac{s_1}{s})^{\frac{1}{2}-H_1}\\&\notag\qquad\qquad\qquad\qquad\qquad\qquad\times(\frac{t_1}{t})^{\frac{1}{2}-H_2}\vert b(s_1,t_1,s^Ht^H\widetilde{W}_{\frac{s_1}{s},\frac{t_1}{t}}^H)\vert\,\mathrm{d}t_1\mathrm{d}s_1\\=\notag&\frac{s^{\frac{1}{2}-H_1}t^{\frac{1}{2}-H_2}}{\Gamma(\frac{1}{2}-H_1)\Gamma(\frac{1}{2}-H_2)}\int_0^1\int_0^1(1-s_1)^{-\frac{1}{2}-H_1}(1-t_1)^{-\frac{1}{2}-H_2}s_1^{\frac{1}{2}-H_1}t_1^{\frac{1}{2}-H_2}\\&\notag\qquad\qquad\qquad\qquad\qquad\qquad\times\vert b(ss_1,tt_1,s^Ht^H\widetilde{W}_{s_1,t_1}^H)\vert\,\mathrm{d}t_1\mathrm{d}s_1\\=\notag&\frac{s^{\frac{1}{2}-H_1}t^{\frac{1}{2}-H_2}}{\Gamma(\frac{1}{2}-H_1)\Gamma(\frac{1}{2}-H_2)}\int_0^1\int_0^1\gamma_{-\frac{1}{2}-H_1,\frac{1}{2}-H_1}(1,s_1)\gamma_{-\frac{1}{2}-H_2,\frac{1}{2}-H_2}(1,t_1) \\&\notag\qquad\qquad\qquad\qquad\qquad\qquad\times\vert b(ss_1,tt_1,s^Ht^H\widetilde{W}_{s_1,t_1}^H)\vert\,\mathrm{d}t_1\mathrm{d}s_1,
	\end{align*}
	where
	\begin{align*}
		\gamma_{\xi,\zeta}(s,r):=r^{\zeta}(s-r)^{\xi},\quad s>r.
	\end{align*}
	As a consequence,
	\begin{align}\label{GirsanovTheo5}
		&\notag\vert\theta_{s,t}\vert^{2m}\\=&\Big\vert K^{-1}_H\Big(\int_0^{\cdot}\int_0^{\cdot}\vert b(s_1,t_1,\widetilde{W}_{s_1,t_1}^H)\vert\,\mathrm{d}t_1\mathrm{d}s_1\Big)(s,t)\Big\vert^{2m}\\\notag=&\frac{s^{2m(\frac{1}{2}-H_1)}t^{2m(\frac{1}{2}-H_2)}}{\Gamma(\frac{1}{2}-H_1)^{2m}\Gamma(\frac{1}{2}-H_2)^{2m}}\Big(\int_0^1\int_0^1\gamma_{-\frac{1}{2}-H_1,\frac{1}{2}-H_1}(1,s_1)\gamma_{-\frac{1}{2}-H_2,\frac{1}{2}-H_2}(1,t_1) \\&\notag\qquad\qquad\qquad\qquad\qquad\qquad\qquad\times\vert b(ss_1,tt_1,s^Ht^H\widetilde{W}_{s_1,t_1}^H)\vert\,\mathrm{d}t_1\mathrm{d}s_1\Big)^{2m}  \\\notag=&\frac{s^{2m(\frac{1}{2}-H_1)}t^{2m(\frac{1}{2}-H_2)}(2m)!}{\Gamma(\frac{1}{2}-H_1)^{2m}\Gamma(\frac{1}{2}-H_2)^{2m}}\sum\limits_{\sigma\in\mathcal{P}_{2m}}\int_{(\nabla^{2m}_{0,1})^2} \prod\limits_{j=1}^{2m}\gamma_{-\frac{1}{2}-H_1,\frac{1}{2}-H_1}(1,s_j)\\&\notag\times\prod\limits_{j=1}^{2m}\gamma_{-\frac{1}{2}-H_2,\frac{1}{2}-H_2}(1,t_{\sigma(j)})\vert b(ss_j,tt_{\sigma(j)},s^Ht^H\widetilde{W}_{s_j,t_{\sigma(j)}}^H)\vert \mathrm{d}s_1\mathrm{d}t_1\ldots\mathrm{d}s_{2m}\mathrm{d}t_{2m}.
	\end{align}
	Hence, we obtain
	\begin{align*}
		\widetilde\E[\vert\theta_{s,t}\vert^{2m}]=&\widetilde\E\Big[\Big\vert K^{-1}_H\Big(\int_0^{\cdot}\int_0^{\cdot}\vert b(s_1,t_1,\widetilde{W}_{s_1,t_1}^H)\vert\,\mathrm{d}t_1\mathrm{d}s_1\Big)(s,t)\Big\vert^{2m}\Big]\\\leq& \frac{s^{2m(\frac{1}{2}-H_1)}t^{2m(\frac{1}{2}-H_2)}(2m)!}{\Gamma(\frac{1}{2}-H_1)^{2m}\Gamma(\frac{1}{2}-H_2)^{2m}}\\&\notag\times\sum\limits_{\sigma\in\mathcal{P}_{2m}}\int_{(\nabla^{2m}_{0,1})^2} \prod\limits_{j=1}^{2m}\gamma_{-\frac{1}{2}-H_1,\frac{1}{2}-H_1}(1,s_j)\gamma_{-\frac{1}{2}-H_2,\frac{1}{2}-H_2}(1,t_{\sigma(j)})\\&\notag\times\widetilde\E\Big[\prod\limits_{j=1}^{2m}\vert b(ss_j,tt_{\sigma(j)},s^Ht^H\widetilde{W}_{s_j,t_{\sigma(j)}}^H)\vert\Big] \mathrm{d}s_1\mathrm{d}t_1\ldots\mathrm{d}s_{2m}\mathrm{d}t_{2m}\mathrm{d}t\mathrm{d}s.
	\end{align*}	
	Moreover, for any $\sigma\in\mathcal{P}_{2m}$,
	\begin{align*}
		&\widetilde\E\Big[\prod\limits_{j=1}^{2m}\vert b(ss_j,tt_{\sigma(j)},s^Ht^H\widetilde{W}_{s_j,t_{\sigma(j)}}^H)\vert\Big]\\=&\int_{(\R^d)^{2m}}\vert b(ss_j,tt_{\sigma(j)},s^Ht^Hy_j)\vert\prod\limits_{\ell=1}^{d}\frac{\exp(-\frac{1}{2}(y^{\ell})^{\ast}Q^{\sigma}(\bar s,\bar t)^{-1}y^{(\ell)})}{(2\pi)^m\text{det}( Q^{\sigma}(\bar s,\bar t))^{\frac{1}{2}}}\,\mathrm{d}y_1\ldots\mathrm{d}y_{2m},
	\end{align*}
	where $y_j:=(y^{(1)},\ldots,y^{(d)})^{\ast}$, $j=1,\ldots,2m$, $y^{(\ell)}=(y^{(\ell)}_1,\ldots,y^{(\ell)}_d)$, $\ell=1,\ldots,d$, $Q^{\sigma}(\bar s,\bar t)=(Q^{\sigma}(\bar s,\bar t)_{i,j})_{1\leq i,j\leq 2m}$ with $\bar s=(s_1,\ldots, s_{2m})$, $\bar t=(t_{\sigma(1)},\ldots, t_{\sigma(2m)})$, $$Q^{\sigma}(\bar s,\bar t)_{i,j}=\widetilde\E[\widetilde W^{H,(1)}_{s_i,t_{\sigma(i)}}\widetilde W^{H,(1)}_{s_j,t_{\sigma(j)}}]\,\text{ and }\,\widetilde W^H=(W^{H,(1)},\ldots,W^{H,(d)})^{\ast}.$$\\Thus
	\begin{align*}
		&\widetilde\E\Big[\prod\limits_{j=1}^{2m}\vert b(ss_j,tt_{\sigma(j)},s^Ht^H\widetilde{W}_{s_j,t_{\sigma(j)}}^H)\vert\Big]\\\leq&\frac{1}{(2\pi)^{md}\text{det}( Q^{\sigma}(\bar s,\bar t))^{\frac{d}{2}}}\int_{(\R^d)^{2m}}\prod\limits_{j=1}^{2m}\vert b(ss_j,tt_{\sigma(j)},s^Ht^Hy_j)\vert \mathrm{d}y_1\ldots\mathrm{d}y_{2m}\\=&\frac{s^{-2mdH_1}t^{-2mdH_2}}{(2\pi)^{md}\text{det}( Q^{\sigma}(\bar s,\bar t))^{\frac{d}{2}}}\int_{(\R^d)^{2m}}\prod\limits_{j=1}^{2m}\vert b(ss_j,tt_{\sigma(j)},y_j)\vert \mathrm{d}y_1\ldots\mathrm{d}y_{2m}\\\leq&\frac{s^{-2mdH_1}t^{-2mdH_2}}{(2\pi)^{md}\text{det}( Q^{\sigma}(\bar s,\bar t))^{\frac{d}{2}}}\Big(\sup\limits_{0\leq r,u\leq T}\int_{\R^d}\vert b(r,u,y)\vert\,\mathrm{d}y\Big)^{2m}\\=&\frac{\Vert b\Vert^{2m}_{L^1_{\infty}}}{(2\pi)^{md}\text{det}( Q^{\sigma}(\bar s,\bar t))^{\frac{d}{2}}}s^{-2mdH_1}t^{-2mdH_2}.
	\end{align*}
	Then we have that 
	\begin{align*}
		\widetilde\E[\vert\theta_{s,t}\vert^{2m}]\leq& \Big(\sup\limits_{0\leq r,u\leq T}\int_{\R^d}\vert b(r,u,y)\vert\,\mathrm{d}y\Big)^{2m} \frac{s^{2m(\frac{1}{2}-H_1)}t^{2m(\frac{1}{2}-H_2)}(2m)!}{\Gamma(\frac{1}{2}-H_1)^{2m}\Gamma(\frac{1}{2}-H_2)^{2m}}\\&\notag\times\sum\limits_{\sigma\in\mathcal{P}_{2m}}\int_{(\nabla^{2m}_{0,1})^2} \prod\limits_{j=1}^{2m}\gamma_{-\frac{1}{2}-H_1,\frac{1}{2}-H_1}(1,s_j)\gamma_{-\frac{1}{2}-H_2,\frac{1}{2}-H_2}(1,t_{\sigma(j)})\\&\times\frac{s^{-2mdH_1}t^{-2mdH_2}}{(2\pi)^{md}\text{det}( Q^{\sigma}(\bar s,\bar t))^{\frac{d}{2}}}\,\mathrm{d}s_1\mathrm{d}t_1\ldots\mathrm{d}s_{2m}\mathrm{d}t_{2m} 
	\end{align*}
	We deduce from Lemma \ref{LemOPXZ02} 
	that
	\begin{align*}
		\text{det} Q^{\sigma}(\bar s,\bar t)=&\text{det Cov}(\widetilde W_{s_1,t_{\sigma(1)}}^{H,(1)},\ldots,\widetilde W_{s_{2m},t_{\sigma(2m)}}^{H,(1)})\geq\text{det }Q_1(\bar s)\times\text{det }Q^{\sigma}_2(\bar t),
	\end{align*}
	where $Q_1(\bar s):=\text{Cov}\left(B^{H_1}_{s_1},\,\ldots,\,B^{H_1}_{s_{2m}}\right)$, $Q_2(\bar t):=\text{Cov}\left(B^{H_2}_{t_{\sigma(1)}},\,\ldots,\,B^{H_2}_{t_{\sigma(2m)}}\right)$ and  $(B^{H_i}_t,t\geq0)$, $i=1,2$ is a  one parameter fractional Brownian motion in $\R$ with Hurst index $H_i$. By the strong local non-determinism of the fractional Brownian motion
	\begin{align*}
		&\text{det Cov}\left(B^{H_1}_{s_1},\,\ldots,\,B^{H_1}_{s_{2m}}\right)\geq K(H_1)s^{2H_1}_{2m}(s_{2m-1}-s_{2m})^{2H_1}\ldots(s_{1}-s_{2})^{2H_1}
	\end{align*}
	and
	\begin{align*}
		\text{det Cov}\left(B^{H_2}_{t_{\sigma(1)}},\,\ldots,\,B^{H_2}_{t_{\sigma(2m)}}\right)=&\text{det Cov}\left(B^{H_2}_{t_{1}},\,\ldots,\,B^{H_2}_{t_{2m}}\right)\\\geq& K(H_2)t^{2H_2}_{2m}(t_{2m-1}-t_{2m})^{2H_2}\ldots(t_{1}-t_{2})^{2H_2}
	\end{align*}
	for all  $\sigma\in\mathcal{P}_{2m}$, $(s_1,\ldots,s_{2m}),\,(t_1,\ldots,t_{2m})\in\nabla^{2m}_{0,1}$ and for constants $K(H_i)$, $i=1,2$ not depending on $m$. By Lemma A.5 in \cite{BOPP22}, one has
	\begin{align*}
		&(2m)!\int_{\nabla^{2m}_{0,1}}\prod\limits_{j=1}^{2m}\gamma_{-\frac{1}{2}-H_1,\frac{1}{2}-H_1}(1,s_j) \frac{1}{(2\pi)^{md}\text{det}( Q_1(\bar s))^{\frac{d}{2}}}\,\mathrm{d}s_1\ldots\mathrm{d}s_{2m}\\\leq&C^m_{H_1,d}(m!)^{2H_1(1+d)}
	\end{align*} 
	and
	\begin{align*}
		&(2m)!\int_{\nabla^{2m}_{0,1}}\prod\limits_{j=1}^{2m}\gamma_{-\frac{1}{2}-H_2,\frac{1}{2}-H_2}(1,t_{\sigma(j)}) \frac{1}{(2\pi)^{md}\text{det}( Q^{\sigma}_2(\bar t))^{\frac{d}{2}}}\,\mathrm{d}t_1\ldots\mathrm{d}t_{2m}\\\leq&C^m_{H_2,d}(m!)^{2H_2(1+d)}
	\end{align*} 
	for all   $\sigma\in\mathcal{P}_{2m}$.
	Therefore, we obtain 
	\begin{align*}
		&\widetilde\E[\vert\theta_{s,t}\vert^{2m}]\\\leq&\Vert b\Vert^{2m}_{L^1_{\infty}}\Big( \frac{s^{2m(\frac{1}{2}-H_1)} s^{-2mdH_1}}{\Gamma(\frac{1}{2}-H_1)^{2m} }\\&\notag\times (2m)!\int_{\nabla^{2m}_{0,1}} \prod\limits_{j=1}^{2m}\gamma_{-\frac{1}{2}-H_1,\frac{1}{2}-H_1}(1,s_j)  \frac{1 }{(2\pi)^{md}\text{det}( Q_{1}(\bar s ))^{\frac{d}{2}}}\,\mathrm{d}s_1\ldots\mathrm{d}s_{2m} \Big)\\&\times\Big( \frac{ t^{2m(\frac{1}{2}-H_2)}t^{-2mdH_2}}{(2m)! \Gamma(\frac{1}{2}-H_2)^{2m}}\\&\notag\times\sum\limits_{\sigma\in\mathcal{P}_{2m}}(2m)!\int_{\nabla^{2m}_{0,1}} \prod\limits_{j=1}^{2m} \gamma_{-\frac{1}{2}-H_2,\frac{1}{2}-H_2}(1,t_{\sigma(j)})\frac{ 1}{(2\pi)^{md}\text{det}( Q^{\sigma}_2(\bar t))^{\frac{d}{2}}}\,\mathrm{d}t_1\ldots\mathrm{d}t_{2m}\Big)
		\\\leq&\frac{C_{H_1,H_2,d}^m(m!)^{2(H_1+H_2)(1+d)}\Vert b\Vert^{2m}_{L^1_{\infty}}  }{\Gamma(\frac{1}{2}-H_1)^{2m}\Gamma(\frac{1}{2}-H_2)^{2m}}\,  s^{2m(\frac{1}{2}-(d+1)H_1)} t^{2m(\frac{1}{2}-(d+1)H_2)}
	\end{align*}
	and, as a consequence,
	\begin{align}
		\notag&\widetilde\E\Big[\Big(\int_0^T\int_0^T\Big\vert K^{-1}_H\Big(\int_0^{\cdot}\int_0^{\cdot}\vert b(s_1,t_1,\widetilde{W}_{s_1,t_1}^H)\vert\,\mathrm{d}t_1\mathrm{d}s_1\Big)(s,t)\Big\vert^2\mathrm{d}t\mathrm{d}s\Big)^{m}\Big]\\\leq\notag&T^{2m-2}\int_0^T\int_0^T \widetilde\E\Big[\Big\vert K^{-1}_H\Big(\int_0^{\cdot}\int_0^{\cdot}\vert b(s_1,t_1,\widetilde{W}_{s_1,t_1}^H)\vert\,\mathrm{d}t_1\mathrm{d}s_1\Big)(s,t)\Big\vert^{2m}\Big]\mathrm{d}t\mathrm{d}s\\\notag=&T^{2m-2}\int_0^T\int_0^T \widetilde\E[\vert\chi_{s,t}\vert^{2m}]\mathrm{d}t\mathrm{d}s\\\leq&\notag\frac{C_{H_1,H_2,d}^m(m!)^{2(H_1+H_2)(1+d)}\,\Vert b\Vert^{2m}_{L^1_{\infty}}}{\Gamma(\frac{1}{2}-H_1)^{2m}\Gamma(\frac{1}{2}-H_2)^{2m}}   \\&\notag\times T^{2m-2}\int_0^T\int_0^Ts^{2m(\frac{1}{2}-(d+1)H_1)} t^{2m(\frac{1}{2}-(d+1)H_2)}\,\mathrm{d}t\mathrm{d}s\\\leq&\frac{C_{H_1,H_2,d}^mT^{4m(1-(d+1)(H_1+H_2))}(m!)^{2(H_1+H_2)(1+d)}\,\Vert b\Vert^{2m}_{L^1_{\infty}}}{\Gamma(\frac{1}{2}-H_1)^{2m}\Gamma(\frac{1}{2}-H_2)^{2m}}   .\label{EsttGirsanov}
	\end{align}
	Then
	\begin{align*}
		&\notag\widetilde\E\Big[\exp\Big\{\mu\int_0^T\int_0^T\Big\vert K^{-1}_H\Big(\int_0^{\cdot}\int_0^{\cdot}\vert b(s_1,t_1,\widetilde{W}_{s_1,t_1}^H)\vert\,\mathrm{d}t_1\mathrm{d}s_1\Big)(s,t)\Big\vert^2\mathrm{d}t\mathrm{d}s\Big\}\Big]\\=&\sum\limits_{m=0}^{\infty}\frac{1}{m!}\widetilde\E\Big[\Big(\int_0^T\int_0^T\Big\vert K^{-1}_H\Big(\int_0^{\cdot}\int_0^{\cdot}\vert b(s_1,t_1,\widetilde{W}_{s_1,t_1}^H)\vert\,\mathrm{d}t_1\mathrm{d}s_1\Big)(s,t)\Big\vert^2\mathrm{d}t\mathrm{d}s\Big)^{m}\Big]\\\leq&\sum\limits_{m=0}^{\infty}\frac{\mu^mT^{2m(2-(d+1)(H_1+H_2))}C_{H_1,H_2,d}^m\Vert b\Vert^{2m}_{L^1_{\infty}} }{(m!)^{1-2(H_1+H_2)(1+d)}\Gamma(\frac{1}{2}-H_1)^{2m}\Gamma(\frac{1}{2}-H_2)^{2m}}:= C_{H,d,\mu,T}(\Vert b\Vert_{L^{1}_{\infty}}),
	\end{align*}
	where $C^m_{H_1,H_2,d}=C^m_{H_1,d}C^m_{H_2,d}$.
	This ends the proof.
\end{proof}

We deduce from Lemma \eqref{GirsanovLemEst} and Theorem \eqref{GirsanovTheo} that the process
\begin{align}\label{WeakSolutEq}
	W^H_{s,t}:=X^x_{s,t}-x-\int_0^s\int_0^tb(s_1,t_1,X^x_{s_1,t_1})\,\mathrm{d}t_1\mathrm{d}s_1,\quad(s,t)\in[0,T]^2
\end{align}
is a $d$-dimensional fractional Brownian sheet on $(\Omega,\mathcal{F},\Pb)$ with Hurst parameter $H=(H_1,H_2)\in(0,\frac{1}{2})^2$
, where $\frac{\mathrm{d}\Pb}{\mathrm{d}\widetilde{\Pb}}=\mathcal{Z}_{T,T}$. Then \eqref{WeakSolutEq} means that $(X^x,W^H)$ is a weak solution of \eqref{eqmainProbS4}. Moreover, this solution is unique in law for  {$b\in L^1_{\infty}$} since the estimates of Lemma \ref{GirsanovLemEst} also holds for $X^x$ in place of $\widetilde W^H$. 
{We end the first step of the scheme by showing that  weak joint uniqueness  also holds for \eqref{eqmainProbS4}.
	\begin{lem}\label{lem:weakjointu}
		Let $(X^{1,x},W^{1,H})$ and $(X^{2,x},W^{2,H})$ be two weak solutions of \eqref{eqmainProbS4} defined on two possibly different filtered probability spaces $(\Omega^1,\mathcal{F}^1,(\mathcal{F}^1_{s,t}),\Pb^1)$ and $(\Omega^2,\mathcal{F}^2,(\mathcal{F}^2_{s,t}),\Pb^2)$. Then $(X^{1,x},W^{1,H})$ and $(X^{2,x},W^{2,H})$ have the same law.
	\end{lem}
	\begin{proof}
		Define the map $\mathfrak{M}$ from $\mathcal{C}([0,T]^2;\R^d)$ to $\mathcal{C}([0,T]^2;\R^d)$ by
		\begin{align*}
			\mathfrak{M}(z)(s,t)=z(s,t)-x-\int_0^s\int_0^tb(s_1,t_1,z(s_1,t_1))\mathrm{d}t_1\mathrm{d}s_1,\,\forall\,z\in\mathcal{C}([0,T]^2;\R^d).
		\end{align*}
		Since $(X^{1,x},W^{1,H})$ and $(X^{2,x},W^{2,H})$ are two weak solutions of \eqref{eqmainProbS4}, we have $W^{1,H}=\mathfrak{M}(X^{1,x})$ $\Pb^1$-a.s. and $W^{2,H}=\mathfrak{M}(X^{2,x})$ $\Pb^2$-a.s. It follows from the uniqueness in law for \eqref{eqmainProbS4}
		that $X^{1,x}$ and $X^{2,x}$ have the same law. As a consequence, $(X^{1,x},\mathfrak{M}(X^{1,x}))$ and $(X^{2,x},\mathfrak{M}(X^{2,x}))$ have the same law, which means that $(X^{1,x},W^{1,H})$ and $(X^{2,x},W^{2,H})$ have the same law.
	\end{proof}
}

Therefore, we consider a filtered probability space $(\Omega,\widetilde{\mathcal{F}},(\widetilde{\mathcal{F}}_{s,t}),\Pb)$ on which the weak solution $(X^x,W^H)$ is defined.  {Also, for every $(s,t)\in[0,T]^2$, we denote by $\mathcal{F}^H_{s,t}$ the $\sigma$-algebra generated by the random variables $\{W^H_{u,v};u\leq s,v\leq t\}$.}


The second step of the scheme is a consequence of the next result which may be seen as  a counterpart of \cite[Lemma 4.6]{BNP18}. 
\begin{lem}\label{LemWeakConverge}
	Let $(b_n)_{n\in\N}$ be a sequence of functions
	in $C^{\infty}_c([0,T]^2\times\R^d,\R^d)$ such that $b_n(s,t,z)$ converges to $b(s,t,z)$ for a.e. $(s,t,z)\in[0,T]^2\times\R^d$ with $\sup_{n\geq0}\Vert b_n\Vert_{L^1_{\infty} }<\infty$. Denote by $X^{x,n}$ the solution of \eqref{eqmainProbS4Xn}. Then for every $(s,t)\in[0,T]^2$ and bounded continuous function $\varphi:\,\R^d\to\R$, $\varphi(X^{x,n}_{s,t})$ converges weakly in  {$L^2(\Omega,\mathcal{F}^H_{s,t})$ to $\E[\varphi(X^x_{s,t})\vert\mathcal{F}^H_{s,t}]$}, where $X^x$ satisfies \eqref{eqmainProbS4}.
\end{lem}
\begin{proof} 
	We suppose without loss of generality $x=0$. Let us start by showing that
	\begin{align}\label{WeakL2Converge}
		&\lim\limits_{n\to\infty}\E\Big[\Big\vert\mathcal{E}\Big(\int_0^s\int_0^t K^{-1}_H\Big(\int_0^{\cdot}\int_0^{\cdot}\vert b_n(s_2,t_2,{W}_{s_2,t_2}^H)\vert\,\mathrm{d}t_2\mathrm{d}s_2\Big)^{\ast}(s_1,t_1)\cdot\mathrm{d}W_{s_1,t_1}\Big)\\&\quad-\mathcal{E}\Big(\int_0^s\int_0^t K^{-1}_H\Big(\int_0^{\cdot}\int_0^{\cdot}\vert b(s_2,t_2,{W}_{s_2,t_2}^H)\vert\,\mathrm{d}t_2\mathrm{d}s_2\Big)^{\ast}(s_1,t_1)\cdot\mathrm{d}W_{s_1,t_1}\Big)\Big\vert^p\Big]=0\notag
	\end{align}
	for all $p\geq1$. To obtain \eqref{WeakL2Converge} we first observe that
	\begin{align}\label{WeakL1Converge}
		\lim\limits_{n\to\infty}&\E\Big[\Big\vert K^{-1}_H\Big(\int_0^{\cdot}\int_0^{\cdot}\vert b_n(s_2,t_2,{W}_{s_2,t_2}^H)\vert\,\mathrm{d}t_2\mathrm{d}s_2\Big)^{}(s_1,t_1)\\&\quad-K^{-1}_H\Big(\int_0^{\cdot}\int_0^{\cdot}\vert b(s_2,t_2,{W}_{s_2,t_2}^H)\vert\,\mathrm{d}t_2\mathrm{d}s_2\Big)^{}(s_1,t_1)\Big\vert\Big]=0\notag
	\end{align}
	for all $(s_1,t_1)$. Indeed we deduce from \eqref{GirsanovTheo3}
	, Fubini's theorem and the dominated convergence theorem that
	\begin{align*}
		&\lim\limits_{n\to\infty}\E\Big[\Big\vert K^{-1}_H\Big(\int_0^{\cdot}\int_0^{\cdot}\vert b_n(s_2,t_2,{W}_{s_2,t_2}^H)\vert\,\mathrm{d}t_2\mathrm{d}s_2\Big)^{}(s_1,t_1)\\&\quad-K^{-1}_H\Big(\int_0^{\cdot}\int_0^{\cdot}\vert b(s_2,t_2,{W}_{s_2,t_2}^H)\vert\,\mathrm{d}t_2\mathrm{d}s_2\Big)^{}(s_1,t_1)\Big\vert\Big]\\\leq&\lim\limits_{n\to\infty}\frac{s_1^{H_1-\frac{1}{2}}t_1^{H_2-\frac{1}{2}}}{\Gamma(\frac{1}{2}-H_1)\Gamma(\frac{1}{2}-H_2)}\int_0^{s_1}\int_0^{t_1}(s_1-s_2)^{-\frac{1}{2}-H_1}(t_1-t_2)^{-\frac{1}{2}-H_2}s_2^{\frac{1}{2}-H_1}t_2^{\frac{1}{2}-H_2}\\&\notag\qquad\qquad\qquad\qquad\qquad\qquad\times\E[\vert b_n(s_2,t_2,{W}_{s_2,t_2}^H)-b(s_2,t_2,{W}_{s_2,t_2}^H)\vert]\,\mathrm{d}t_2\mathrm{d}s_2\\=&\lim\limits_{n\to\infty}\frac{s_1^{H_1-\frac{1}{2}}t_1^{H_2-\frac{1}{2}}}{\Gamma(\frac{1}{2}-H_1)\Gamma(\frac{1}{2}-H_2)}\int_0^{s_1}\int_0^{t_1}(s_1-s_2)^{-\frac{1}{2}-H_1}(t_1-t_2)^{-\frac{1}{2}-H_2}s_2^{\frac{1}{2}-H_1}t_2^{\frac{1}{2}-H_2}\\&\notag\qquad\qquad\qquad\qquad\qquad\qquad\times\int_{\R^d}\vert b_n(s_2,t_2,y)-b(s_2,t_2,y)\vert(2\pi s_2^{2H}t_2^{2H})^{-d/2}\\&\notag\qquad\qquad\qquad\qquad\qquad\qquad\qquad\times\exp\Big(-\frac{y^2}{2s_2^{2H}t_2^{2H}}\Big)\,\mathrm{d}y\mathrm{d}s_2\mathrm{d}t_2=0
	\end{align*}
	since $\lim_{n\to\infty}b_n(s_2,t_2,y)=b(s_2,t_2,y)$ for a.e. $(s_2,t_2,y)$.\\
	Since  $\left(K^{-1}_H(\int_0^{\cdot}b_n(s_2,t_2,W^H_{s_2,t_2})\,\mathrm{d}s_2\mathrm{d}t_2)\right)_{n\in\N}$ is bounded in $L^p(\Omega\times[0,s]\times[0,t],\R^d)$ for any $p\geq1$ (see \eqref{EsttGirsanov}), we deduce from \eqref{WeakL1Converge} that
	\begin{align}\label{StrLpConverge}
		\lim\limits_{n\to\infty}&\E\Big[\Big\vert K^{-1}_H\Big(\int_0^{\cdot}\int_0^{\cdot}\vert b_n(s_2,t_2,{W}_{s_2,t_2}^H)\vert\,\mathrm{d}t_2\mathrm{d}s_2\Big)^{}(s_1,t_1)\\&\quad-K^{-1}_H\Big(\int_0^{\cdot}\int_0^{\cdot}\vert b(s_2,t_2,{W}_{s_2,t_2}^H)\vert\,\mathrm{d}t_2\mathrm{d}s_2\Big)^{}(s_1,t_1)\Big\vert^p\Big]=0\notag
	\end{align}
	for any $p\geq1$.
	As a consequence,
	\begin{align*}
		\lim\limits_{n\to\infty}&\E\Big[\Big\vert\int_0^s\int_0^t K^{-1}_H\Big(\int_0^{\cdot}\int_0^{\cdot}\vert b_n(s_2,t_2,{W}_{s_2,t_2}^H)\vert\,\mathrm{d}t_2\mathrm{d}s_2\Big)^{\ast}(s_1,t_1)\mathrm{d}W_{s_1,t_1}\\&\quad-\int_0^s\int_0^t K^{-1}_H\Big(\int_0^{\cdot}\int_0^{\cdot}\vert b(s_2,t_2,{W}_{s_2,t_2}^H)\vert\,\mathrm{d}t_2\mathrm{d}s_2\Big)^{\ast}(s_1,t_1)\mathrm{d}W_{s_1,t_1}\Big\vert^p\Big]=0
	\end{align*}
	and
	\begin{align*}
		\lim\limits_{n\to\infty}&\E\Big[\Big(\int_0^s\int_0^t \Big\vert K^{-1}_H\Big(\int_0^{\cdot}\int_0^{\cdot}\vert b_n(s_2,t_2,{W}_{s_2,t_2}^H)\vert\,\mathrm{d}t_2\mathrm{d}s_2\Big)^{}(s_1,t_1)\Big\vert^2\mathrm{d}t_1\mathrm{d}s_1\\&\quad-\int_0^s\int_0^t \Big\vert K^{-1}_H\Big(\int_0^{\cdot}\int_0^{\cdot}\vert b_n(s_2,t_2,{W}_{s_2,t_2}^H)\vert\,\mathrm{d}t_2\mathrm{d}s_2\Big)^{}(s_1,t_1)\Big\vert^2\mathrm{d}t_1\mathrm{d}s_1\Big)^p\,\Big]=0
	\end{align*}
	for any $p\geq1$. The equality \eqref{WeakL2Converge} then follows from the estimate $\vert e^y-e^z\vert\leq  (e^y+e^z)\vert y-z\vert$, H\"older inequality and the bounds in Lemma \ref{GirsanovLemEst}
	Similarly, one also shows that
	\begin{align*}
		\lim\limits_{n\to\infty}&\E\Big[\Big\vert\exp\Big(\sum\limits_{j=1}^d\int_0^{s_1}\int_0^{t_1}f_j(s_2,t_2)b_n^{(j)}(s_2,t_2,x+W^H_{s_2,t_2})\mathrm{d}t_2\mathrm{d}s_2\Big)\\&\quad-\exp\Big(\sum\limits_{j=1}^d\int_0^{s_1}\int_0^{t_1}f_j(s_2,t_2)b^{(j)}(s_2,t_2,x+W^H_{s_2,t_2})\mathrm{d}t_2\mathrm{d}s_2\Big)\Big\vert^p\Big]=0
	\end{align*}
	for all $p\geq1$, $0< s_1\leq s\leq T$, $0< t_1\leq t$, $f_j\in L^{\infty}([0,s]\times[0,t])$.\\
	We  {observe} that the set%
	\begin{align*}
		\Pi _{s,t} :=&\Big\{ \exp \Big(\sum_{j=1}^{d}\int_{0}^{s_1}\int_0^{t_1}f_{j}(s_{2},t_{2})\mathrm{d}W_{s_{2},t_{2}}^{(j),H}\Big):\,0\leq s_{1}\leq
		s,\,0\leq t_1\leq s,\\&\qquad\,f_{j}\in L^{\infty }(\left[ 0,s\right] \times \left[ 0,t%
		\right] ),j=1,\ldots,d\Big\} 
	\end{align*}%
	is a total subspace of $L^{p}(\Omega ;\mathcal{F}^H_{s,t})$. To conclude the proof, it is
	then sufficient to show that%
	\begin{equation*}
		\E\left[ \varphi (X_{s,t}^{x,n})\xi \right] \underset{%
			n\longrightarrow \infty }{\longrightarrow }\E\left[ \E\left[ \varphi
		(X_{s,t}^{x})\right. \left\vert \mathcal{F}^H_{s,t}\right] \xi %
		\right] 
	\end{equation*}%
	for all $\xi \in \Pi _{s,t}$.  {For $n\geq1$, define the new measures $\Q_n$ by } 
	\begin{equation*}
		\mathrm{d}\Q_{n}=\mathcal{E}\Big(\int_0^s\int_0^t K^{-1}_H\Big(\int_0^{\cdot}\int_0^{\cdot}\vert b_n(s_2,t_2,{X}_{s_2,t_2}^{x,n})\vert\,\mathrm{d}t_2\mathrm{d}s_2\Big)^{\ast}(s_1,t_1)\cdot\mathrm{d}W_{s_1,t_1}\Big)\mathrm{d}\Pb.
	\end{equation*}%
	Then  {the} Girsanov`s theorem for multiparameter Wiener processes
	implies that  {the process $\widetilde{W}^{H,n}$ defined by } %
	\begin{equation*}
		\widehat W_{\cdot }^{H ,n}:=W^H_{\cdot }-\int_{0}^{\cdot }\int_{0}^{\cdot
		}b_{n}(s_{2},t_{2},X_{s_{2},t_{2}}^{x,n})\mathrm{d}t_{2}\mathrm{d}s_{2}
	\end{equation*}%
	is a $\mathbb{Q}_{n}$-Wiener process on the plane and one has
	\begin{align*}
	&\E\Big[\varphi(X^{x,n}_{s,t})\exp \Big(\sum_{j=1}^{d}\int_{0}^{s_1}\int_0^{t_1}f_{j}(s_{2},t_{2})\mathrm{d}W_{s_{2},t_{2}}^{(j),H}\Big)\Big]\\=& \E_{\Q_{n}}\Big[\varphi(x+\widehat W^{H,n}_{s,t})\exp \Big(\sum_{j=1}^{d}\int_{0}^{s_1}\int_0^{t_1}f_{j}(s_{2},t_{2})\,\mathrm{d}\widehat W_{s_{2},t_{2}}^{(j),H,n}\Big)\\&\times\exp\Big(\sum\limits_{j=1}^d\int_0^{s_1}\int_0^{t_1}f_j(s_2,t_2)b_n^{(j)}(s_2,t_2,x+\widehat W^{H,n}_{s_2,t_2})\,\mathrm{d}t_2\mathrm{d}s_2\Big)\\&\times\mathcal{E}\Big(\int_0^s\int_0^t K^{-1}_H\Big(\int_0^{\cdot}\int_0^{\cdot}\vert b_n(s_2,t_2,x+\widehat{W}_{s_2,t_2}^{H,n})\vert\,\mathrm{d}t_2\mathrm{d}s_2\Big)^{\ast}(s_1,t_1)\cdot\mathrm{d}\widehat W^{H,n}_{s_1,t_1}\Big)\Big]\\=& \E_{}\Big[\varphi(x+ W^{H}_{s,t})\exp \Big(\sum_{j=1}^{d}\int_{0}^{s_1}\int_0^{t_1}f_{j}(s_{2},t_{2})\,\mathrm{d} W_{s_{2},t_{2}}^{(j),H}\Big)\\&\times\exp\Big(\sum\limits_{j=1}^d\int_0^{s_1}\int_0^{t_1}f_j(s_2,t_2)b_n^{(j)}(s_2,t_2,x+ W^{H}_{s_2,t_2})\,\mathrm{d}t_2\mathrm{d}s_2\Big)\\&\times\mathcal{E}\Big(\int_0^s\int_0^t K^{-1}_H\Big(\int_0^{\cdot}\int_0^{\cdot}\vert b_n(s_2,t_2,x+{W}_{s_2,t_2}^{H})\vert\,\mathrm{d}t_2\mathrm{d}s_2\Big)^{\ast}(s_1,t_1)\cdot\mathrm{d} W^{H}_{s_1,t_1}\Big)\Big].
	\end{align*}
	As a consequence,
	\begin{align*}
		&\lim\limits_{n\to\infty}\E\Big[\varphi(X^{x,n}_{s,t})\exp \Big(\sum_{j=1}^{d}\int_{0}^{s_1}\int_0^{t_1}f_{j}(s_{2},t_{2})\mathrm{d}W_{s_{2},t_{2}}^{(j),H}\Big)\Big]\\=&\E_{}\Big[\varphi(x+ W^{H}_{s,t})\exp \Big(\sum_{j=1}^{d}\int_{0}^{s_1}\int_0^{t_1}f_{j}(s_{2},t_{2})\,\mathrm{d} W_{s_{2},t_{2}}^{(j),H}\Big)\\&\times\exp\Big(\sum\limits_{j=1}^d\int_0^{s_1}\int_0^{t_1}f_j(s_2,t_2)b^{(j)}(s_2,t_2,x+ W^{H}_{s_2,t_2})\,\mathrm{d}t_2\mathrm{d}s_2\Big)\\&\times\mathcal{E}\Big(\int_0^s\int_0^t K^{-1}_H\Big(\int_0^{\cdot}\int_0^{\cdot}\vert b(s_2,t_2,x+{W}_{s_2,t_2}^{H})\vert\,\mathrm{d}t_2\mathrm{d}s_2\Big)^{\ast}(s_1,t_1)\cdot\mathrm{d} W^{H}_{s_1,t_1}\Big)\Big]\\=&\E\Big[\varphi(X^x_{s,t})\exp \Big(\sum_{j=1}^{d}\int_{0}^{s_1}\int_0^{t_1}f_{j}(s_{2},t_{2})\,\mathrm{d} W_{s_{2},t_{2}}^{(j),H}\Big)\Big]\\=&\E\Big[\E[\varphi(X^x_{s,t})\vert\mathcal{F}^H_{s,t}]\exp \Big(\sum_{j=1}^{d}\int_{0}^{s_1}\int_0^{t_1}f_{j}(s_{2},t_{2})\,\mathrm{d} W_{s_{2},t_{2}}^{(j),H}\Big)\Big].
	\end{align*}
	This completes the proof.
\end{proof}

We turn now to the third step of the scheme.  {The following result show that
	the compactness criterion for subsets of $L^2(\Omega)$ which we recall in the Appendix is
	satisfied in this case.}

\begin{thm}\label{RelCompacTheo}
	Let $\{b_n\}_{n\in\N_0}\subset C^{\infty}_0([0,T]^2\times\R^d)^d$ be a  sequence that converges to $b$ in $L^1_{\infty}$. For any $n\in\N_0$, denote by $X^{x,n}=(X^{x,n}_{s,t},(s,t)\in[0,T]^2)$ the corresponding solution of \eqref{MainPb} if we replace $b$ by $b_n$. Then there exists $\beta\in(0,1/2)$ such that for any  $(s,t)\in(0,T]^2$ and $\esp\in(0,\min\{s,t\})$, 
	\begin{align}\label{TheoExistEstim1}
		&\sup\limits_{n\in\N_0}\int_{[\esp,s]^2\times[\esp,t]^2}\frac{\E\left[\Vert D_{r,u}X^{x,n}_{s,t}-D_{\bar r,\bar u}X^{x,n}_{s,t}\Vert^2\right]}{(\vert r-\bar r\vert+\vert u-\bar u\vert)^{2+2\beta}}\,\mathrm{d}u\mathrm{d}\bar u\mathrm{d}r\mathrm{d}\bar r\\&\notag\qquad\leq\sup\limits_{n\in\N_0}C_{\esp,H,T,d}(\Vert b_n\Vert_{L^1_{\infty}})<\infty
	\end{align}
	and 
	\begin{align}\label{TheoExistEstim2}
		\sup\limits_{n\in\N_0}\Vert D_{\cdot,\cdot}X^{x,n}_{s,t}\Vert_{L^2(\Omega\times[\esp,s]\times[\esp,t],\R^{d\times d})}\leq\sup\limits_{n\in\N_0}C_{\esp,H,T,d}(\Vert b_n\Vert_{L^1_{\infty}})<\infty
	\end{align}
	for some continuous function $C_{\esp,H,,d,T}:\,\R_+^2\to\R_+$, where $\Vert\cdot\Vert$ denotes the maximum norm in $\R^{d\times d}$.
\end{thm}
\begin{proof}
	We only show \eqref{TheoExistEstim1} since the proof of \eqref{TheoExistEstim2} follows analogeously.\\
	Let $(s,t)\in[\esp,T]^2$ and $(r,\bar r,u,\bar u)\in\R_+^4$ such that $\esp<\bar r,\,r<s$ and $\esp<\bar u,\,u<t$. Suppose without loss of generality we suppose $\bar u<u$ and $\bar r<r$. By the chain rule for the Malliavin derivative, one has
	\begin{align*}
		D_{r,u}X^{x,n}_{s,t}=K_H(s,t,r,u)I_d+\int_r^s\int_u^tb_n^{\prime}(s_1,t_1,X^{x,n}_{s_1,t_1})D_{r,u}X^{x,n}_{s_1,t_1}\,\mathrm{d}t_1\mathrm{d}s_1,
	\end{align*}
	where the above equality holds in $L^p([\esp,T]^2)$,   $b^{\prime}(s,t,z)=\Big(\frac{\partial}{\partial z_j}b^{(i)}_n(s,t,z)\Big)_{i,j=1,\ldots,d}$ is the Jacobian  matrix of $b_n$ and $I_d$ is the identity matrix in $\R^{d\times d}$. Hence one has
	\begin{align*}
		&D_{r,u}X^{x,n}_{s,t}-D_{\bar r,\bar u}X^{x,n}_{s,t}\\=&
K_H(s,t,r,u)I_d-K_H(s,t,\bar r,\bar u)I_d-(D_{\bar r,\bar u}X^{x,n}_{s,u}-K_H(s,u,\bar r,\bar u)I_d)\\&-(D_{\bar r,\bar u}X^{x,n}_{r,t}-K_H(r,t,\bar r,\bar u)I_d)+(D_{\bar r,\bar u}X^{x,n}_{r,u}-K_H(r,u,\bar r,\bar u)I_d)\\&+\int_r^s\int_u^tb_n^{\prime}(s_1,t_1,X^{x,n}_{s_1,t_1})(D_{r,u}X^{x,n}_{s_1,t_1}-D_{\bar r,\bar u}X^{x,n}_{s_1,t_1})\,\mathrm{d}t_1\mathrm{d}s_1.
	\end{align*}
	Then Picard iteration applied to the above equation yields
	\begin{align*}
		&D_{r,u}X^{x,n}_{s,t}-D_{\bar r,\bar u}X^{x,n}_{s,t}\\=&(K_H(s,t,r,u)I_d-K_H(s,t,\bar r,\bar u)I_d)-(D_{\bar r,\bar u}X^{x,n}_{s,u}-K_H(s,u,\bar r,\bar u)I_d)\\&-(D_{\bar r,\bar u}X^{x,n}_{r,t}-K_H(r,t,\bar r,\bar u)I_d)\\&+\sum\limits_{m=1}^{\infty}\int_{\nabla^{m}_{r,s}}\int_{\nabla^{m}_{u,t}}\prod\limits_{j=1}^mb_n^{\prime}(s_j,t_j,X^{x,n}_{s_j,t_j})(K_H(s_m,t_m,r,u)I_d-K_H(s_m,t_m,\bar r,\bar u)I_d)\\&\qquad\qquad\qquad\qquad\times\mathrm{d}t_1\ldots\mathrm{d}t_m\mathrm{d}s_1\ldots\mathrm{d}s_m\\&-\sum\limits_{m=1}^{\infty}\int_{\nabla^{m}_{r,s}}\int_{\nabla^{m}_{u,t}}\prod\limits_{j=1}^mb_n^{\prime}(s_j,t_j,X^{x,n}_{s_j,t_j})(D_{\bar r,\bar u}X_{s_m,u}-K_H(s_m,u,\bar r,\bar u)I_d)\,\mathrm{d}t_j\mathrm{d}s_j\\&-\sum\limits_{m=1}^{\infty}\int_{\nabla^{m}_{r,s}}\int_{\nabla^{m}_{u,t}}\prod\limits_{j=1}^mb_n^{\prime}(s_j,t_j,X^{x,n}_{s_j,t_j})(D_{\bar r,\bar u}X_{r,t_m}-K_H(r,t_m,\bar r,\bar u)I_d)\,\mathrm{d}t_j\mathrm{d}s_j\\&+\Big(I_d+\sum\limits_{m=1}^{\infty}\int_{\nabla^{m}_{r,s}}\int_{\nabla^{m}_{u,t}}\prod\limits_{j=1}^mb_n^{\prime}(s_j,t_j,X^{x,n}_{s_j,t_j})\,\mathrm{d}s_j\mathrm{d}t_j\Big)\\&\quad\times(D_{\bar r,\bar u}X^{x,n}_{r,u}-K_H(r,u,\bar r,\bar u)I_d)\\:=&I_1(r,u,\bar r,\bar u)-I_{2,n}(r,u,\bar r,\bar u)-I_{3,n}(r,u,\bar r,\bar u)+I_{4,n}(r,u,\bar r,\bar u)-I_{5,n}(r,u,\bar r,\bar u)\\&-I_{6,n}(r,u,\bar r,\bar u)+I_{7,n}(r,u,\bar r,\bar u).
	\end{align*}
	Observe that for any $(\esp,\esp)\preceq(r,u)\prec(s,t)\preceq(T,T)$, one has
	\begin{align*}
		&D_{r,u}X^{x,n}_{s,t}-K_H(s,t,r,u)I_d\\=&\sum\limits_{\ell=1}^{\infty}\int_{\nabla_{r,s}^{\ell}}\int_{\nabla_{u,t}^{\ell}}\prod\limits_{j=1}^{\ell}b^{\prime}(s_j,t_j,X^{x,n}_{s_j,t_j})K_H(s_{\ell},t_{\ell},r,u)I_d\,\mathrm{d}t_j\mathrm{d}s_j.
	\end{align*}
	Hence we may write
	\begin{align*}
		I_{2,n}(r,u,\bar r,\bar u)=\sum\limits_{\ell=1}^{\infty}\int_{\nabla_{\bar r,s}^{\ell}}\int_{\nabla_{\bar u,u}^{\ell}}\prod\limits_{j=1}^{\ell}b^{\prime}(s_j,t_j,X^{x,n}_{s_j,t_j})K_H(s_{\ell},t_{\ell},\bar r,\bar u)I_d\,\mathrm{d}t_j\mathrm{d}s_j,
	\end{align*}
	\begin{align*}
		I_{3,n}(r,u,\bar r,\bar u)=\sum\limits_{\ell=1}^{\infty}\int_{\nabla_{\bar r,r}^{\ell}}\int_{\nabla_{\bar u,t}^{\ell}}\prod\limits_{j=1}^{\ell}b^{\prime}(s_j,t_j,X^{x,n}_{s_j,t_j})K_H(s_{\ell},t_{\ell},\bar r,\bar u)I_d\,\mathrm{d}t_j\mathrm{d}s_j,
	\end{align*}
	\begin{align}\label{CompactEqIn5}
		&I_{5,n}(r,u,\bar r,\bar u)\\=&\notag\sum\limits_{m=1}^{\infty}\sum\limits_{\ell=1}^{\infty}\int_{\nabla^{m}_{r,s}}\int_{\nabla^{\ell}_{\bar r,s_m}}\int_{\nabla^{m}_{u,t}}\int_{\nabla^{\ell}_{\bar u,u}}\prod\limits_{j=1}^{m+\ell}b_n^{\prime}(s_j,t_j,X^{x,n}_{s_j,t_j}) K_H(s_{m+\ell},t_{m+\ell},\bar r,\bar u)I_d\\&\notag\qquad\times\mathrm{d}t_{m+\ell}\ldots\mathrm{d}t_{m+1}\mathrm{d}t_m\ldots\mathrm{d}t_1\mathrm{d}s_{m+\ell}\ldots\mathrm{d}s_{m+1}\mathrm{d}s_m\ldots\mathrm{d}s_1,
	\end{align}
	\begin{align}\label{CompactEqIn6}
		&I_{6,n}(r,u,\bar r,\bar u)\\=&\notag\sum\limits_{m=1}^{\infty}\sum\limits_{\ell=1}^{\infty}\int_{\nabla^{m}_{r,s}}\int_{\nabla^{\ell}_{\bar r,r}}\int_{\nabla^{m}_{u,t}}\int_{\nabla^{\ell}_{\bar u,t_m}}\prod\limits_{j=1}^{m+\ell}b_n^{\prime}(s_j,t_j,X^{x,n}_{s_j,t_j}) K_H(s_{m+\ell},t_{m+\ell},\bar r,\bar u)I_d\\&\notag\qquad\times\mathrm{d}t_{m+\ell}\ldots\mathrm{d}t_{m+1}\mathrm{d}t_m\ldots\mathrm{d}t_1\mathrm{d}s_{m+\ell}\ldots\mathrm{d}s_{m+1}\mathrm{d}s_m\ldots\mathrm{d}s_1
	\end{align}
	and
	\begin{align*}
		I_{7,n}(r,u,\bar r,\bar u)=&\Big(I_d+\sum\limits_{m=1}^{\infty}\int_{\nabla^{m}_{r,s}}\int_{\nabla^{m}_{u,t}}\prod\limits_{j=1}^mb_n^{\prime}(s_j,t_j,X^{x,n}_{s_j,t_j})\,\mathrm{d}s_j\mathrm{d}t_j\Big)\\&\times\Big(\sum\limits_{\ell=1}^{\infty}\int_{\nabla_{\bar r,r}^{\ell}}\int_{\nabla_{\bar u,u}^{\ell}}\prod\limits_{j=1}^{\ell}b^{\prime}(s_j,t_j,X^{x,n}_{s_j,t_j})K_H(s_{\ell},t_{\ell},\bar r,\bar u)I_d\,\mathrm{d}t_j\mathrm{d}s_j\Big).
	\end{align*}
	For any $m,\,\ell\in\N$ and $0\leq\bar r<r<s\leq T$, define the sets
	\begin{align*}
		\Xi^{m,\ell}_{\bar r,r,s}:=\{(s_1,\ldots,s_{m+\ell}):\,(s_1,\ldots,s_m)\in\nabla_{r,s}^{m}\text{ and }(s_{m+1},\ldots,s_{m+\ell})\in\nabla^{\ell}_{\bar r,s_m}\}
	\end{align*}
	and 
	\begin{align*}
		\Delta^{m,\ell}_{\bar r,r,s}:=&\{(s_1,\ldots,s_{m+\ell}):\,(s_1,\ldots,s_m)\in\nabla_{r,s}^{m}\text{ and }(s_{m+1},\ldots,s_{m+\ell})\in\nabla^{\ell}_{\bar r,r}\}\\=&\{(s_1,\ldots,s_{m+\ell}):\,\bar r<s_{m+\ell}<\ldots<s_{m+1}<r<s_m<\ldots<s_1<s\}.
	\end{align*}
	We have $\Xi^{m,\ell}_{\bar r,r,s}\subset\nabla^{m+\ell}_{\bar r,s}$ and $\Delta^{m,\ell}_{\bar r,r,s}\subset\nabla^{m+\ell}_{\bar r,s}$. One can see that \eqref{CompactEqIn5} and \eqref{CompactEqIn6} rewrite
	\begin{align*}
		&I_{5,n}(r,u,\bar r,\bar u)\\=&\notag\sum\limits_{m=1}^{\infty}\int_{\Xi^{m,\ell}_{\bar r,r,s}} \int_{\Delta^{m,\ell}_{\bar u,u,t}} \prod\limits_{j=1}^{m+\ell}b_n^{\prime}(s_j,t_j,X^{x,n}_{s_j,t_j}) K_H(s_{m+\ell},t_{m+\ell},\bar r,\bar u)I_d\,\mathrm{d}t_j\mathrm{d}s_j 
	\end{align*}
	and
	\begin{align*}
		&I_{6,n}(r,u,\bar r,\bar u)\\=&\notag\sum\limits_{m=1}^{\infty}\int_{\Delta^{m,\ell}_{\bar r,r,s}} \int_{\Xi^{m,\ell}_{\bar u,u,t}} \prod\limits_{j=1}^{m+\ell}b_n^{\prime}(s_j,t_j,X^{x,n}_{s_j,t_j}) K_H(s_{m+\ell},t_{m+\ell},\bar r,\bar u)I_d\,\mathrm{d}t_j\mathrm{d}s_j. 
	\end{align*}
	We deduce from Lemma \ref{ExistEstLem1} that there exists $\beta_1\in(0,\frac{1}{2})$ such that for any $(0,0)\prec(\esp,\esp) \prec(s,t)\preceq(T,T)$,
	\begin{align}\label{ExistEstLemPr1}
		&\notag\int_{[\esp,s]^2}\int_{[\esp,t]^2}\frac{\Vert I_1(r,u,\bar r,\bar u)\Vert^2}{(\vert r-\bar r\vert+\vert u-\bar u\vert)^{2+2\beta_1}}\mathrm{d}u\mathrm{d}\bar u\mathrm{d}r\mathrm{d}\bar r\\
		\notag=&\int_{[\ve,s]^2}\int_{[\ve,t]^2}\frac{\vert K_H(s,t,r,u)-K_H(s,t,\bar r,\bar u)\vert^2}{(\vert r-\bar r\vert+\vert u-\bar u\vert)^{2+2\beta_1}}\mathrm{d}u\mathrm{d}\bar u\mathrm{d}r\mathrm{d}\bar r
		\\\leq&	
		\int_{[0,s]^2}\int_{[0,t]^2}\frac{\vert K_H(s,t,r,u)-K_H(s,t,\bar r,\bar u)\vert^2}{(\vert r-\bar r\vert+\vert u-\bar u\vert)^{2+2\beta_1}}\mathrm{d}u\mathrm{d}\bar u\mathrm{d}r\mathrm{d}\bar r<\infty.
	\end{align}
	Next we estimate the term $I_{4,n}(r,u,\bar r,\bar u)$. It follows from  {Theorem \ref{GirsanovTheo}}, Cauchy-Schwarz inequality and Lemma  \ref{GirsanovLemEst} that
	\begin{align}
		&\notag\E[\Vert I_{4,n}(r,u,\bar r,\bar u)\Vert^2]\\\leq&\notag C_{H,d,T}(\Vert b_n\Vert_{L^{\infty}_{\infty}})\E\Big[\Big\Vert \sum\limits_{m=1}^{\infty}\int_{\nabla^{m}_{r,s}}\int_{\nabla^{m}_{u,t}}\prod\limits_{j=1}^mb_n^{\prime}(s_j,t_j,x+W^{H}_{s_j,t_j})(K_H(s_m,t_m,r,u)I_d\\&\qquad\qquad\qquad\qquad\notag-K_H(s_m,t_m,\bar r,\bar u)I_d)\mathrm{d}t_1\ldots\mathrm{d}t_m\mathrm{d}s_1\ldots\mathrm{d}s_m\Big\Vert^4\Big]\\\leq&C_1\Big(\sum\limits_{m=1}^{\infty}\sum\limits_{i,j=1}^d\sum\limits_{l_1,\ldots,l_{m-1}=1}^d\Big\Vert\int_{\nabla^{m}_{r,s}}\int_{\nabla^{m}_{u,t}}\frac{\partial b_n^{(i)}}{\partial x_{l_1}}(s_1,t_1,x+W^H_{s_1,t_1})\label{ExistEstimI4n}\\&\notag\qquad\times\frac{\partial b_n^{(l_1)}}{\partial x_{l_2}}(s_2,t_2,x+W^H_{s_2,t_2})\cdots\frac{\partial b_n^{(l_{m-1})}}{\partial x_{j}}(s_m,t_m,x+W^H_{s_m,t_m})\\&\qquad\notag\times(K_H(s_m,t_m,r,u)-K_H(s_m,t_m,\bar r,\bar u))\mathrm{d}t_1\ldots\mathrm{d}t_m\mathrm{d}s_1\ldots\mathrm{d}s_m\Big\Vert_{L^4(\Omega)}\Big)^2,
	\end{align}
	where $C_{H,d,T}=(C_{H,d,2,T})^{1/2}$, $C_{H,d,2,T}$ being the function from Lemma \ref{GirsanovLemEst} with $p=2$ and $$C_1:=\sup_{n\in\N_0}C_{H,d,T}(\Vert b_n\Vert_{L^{1}_{\infty}})<\infty.$$
	We define
	\begin{align*}
		J_{4,n}(r,u,\bar r,\bar u,W^H):=&\int_{\nabla^{m}_{r,s}}\int_{\nabla^{m}_{u,t}}\frac{\partial b_n^{(i)}}{\partial x_{l_1}}(s_1,t_1,x+W^H_{s_1,t_1})\\&\times\frac{\partial b_n^{(l_1)}}{\partial x_{l_2}}(s_2,t_2,x+W^H_{s_2,t_2})\cdots\frac{\partial b_n^{(l_{m-1})}}{\partial x_{j}}(s_m,t_m,x+W^H_{s_m,t_m})\\&\times(K_H(s_m,t_m,r,u)-K_H(s_m,t_m,\bar r,\bar u))\mathrm{d}t_1\ldots\mathrm{d}t_m\mathrm{d}s_1\ldots\mathrm{d}s_m.
	\end{align*}
	It follows from 
	Proposition \ref{RegulariseProp1} with $p=4$, $\sum_{j=1}^{m}\ve_{j}=1$ (since $\ve_m=1$ and $\ve_j=0$ for $j=1,\ldots,m-1$), $\vert\alpha_j\vert=\sum_{l=1}^d\alpha^{(l)}_j=1$ for all $j$, $\vert\alpha\vert=m$ and from the inequality $\prod_{l=1}^d(2p\vert\alpha^{(l)}\vert)!\leq(2p\sum_{l=1}^d\vert\alpha^{(l)}\vert)!=(2p\vert\alpha\vert)!$ that
	\begin{align*}
		&\E[(J_{4,n}(r,u,\bar r,\bar u,W^H))^4]\leq\Big[r^{H_1-\frac{1}{2}-\gamma}\bar u^{H_2-\frac{1}{2}-\gamma}\Big(\frac{\vert u-\bar u\vert^{\gamma}}{(u\bar u)^{\gamma}} +\frac{\vert r-\bar r\vert^{\gamma}}{(r\bar r)^{\gamma}} \Big)\Big]^{4}\\&\quad\times C^{4m}\Vert b_n\Vert^{4m}_{L^1_{\infty}}A^{4,\gamma}_{d,m}(H,\vert s-r\vert,\vert t-u\vert),
	\end{align*}
	where
	\begin{align*}
	A^{4,\gamma}_{d,m}(H,\vert s-r\vert,\vert t-u\vert)    :=&\frac{((8m)!)^{\frac{1}{4}}(s-r)^{4(H_1-\frac{1}{2}-\gamma) -4m(1+d)H_1+4m}}{\Gamma\left(8(H_1-\frac{1}{2}-\gamma)- 8m(1+d)H_1+8m+1\right)^{1/2}}
		\\&\times\frac{(t-u)^{4(H_2-\frac{1}{2}-\gamma) -4m(1+d)H_2+4m}}{\Gamma\left(8(H_2-\frac{1}{2}-\gamma) -8m(1+d)H_2+8m+1\right)^{1/2}}.
	\end{align*}
	Then using \eqref{ExistEstimI4n}, we have
	\begin{align}
		\E[\Vert I_{4,n}(r,u,\bar r,\bar u)\Vert^2]\leq&\Big[r^{H_1-\frac{1}{2}-\gamma}\bar u^{H_2-\frac{1}{2}-\gamma}\Big(\frac{\vert u-\bar u\vert^{\gamma}}{(u\bar u)^{\gamma}} +\frac{\vert r-\bar r\vert^{\gamma}}{(r\bar r)^{\gamma}} \Big)\Big]^{2}\\&\notag\times\Big(\sum\limits_{m=1}^{\infty}d^{m+1} C^{m}\Vert b_n\Vert^{m}_{L^1_{\infty}}(A^{4,\gamma}_{d,m}(H,\vert s-r\vert,\vert t-u\vert))^{\frac{1}{4}}\Big)^2.
	\end{align}
	Since $\max\{H_1,H_2\}<\frac{1}{2(d+1)}$ one can see that the above sum  {converges}. Hence, there exists a continuous function $C^{(4)}_{\esp,H,d,T}:\,[0,\infty)^2\to[0,\infty)$ such that
	\begin{align*}
		&\sup\limits_{n\in\N_0}\E[\Vert I_{4,n}(r,u,\bar r,\bar u)\Vert^2]\\\leq&\sup_{n\in\N_0}C^{(4)}_{\esp,H,d,T}(\Vert b_n\Vert_{L^1_{\infty}}) \Big[r^{H_1-\frac{1}{2}-\gamma}\bar u^{H_2-\frac{1}{2}-\gamma}\Big(\frac{\vert u-\bar u\vert^{\gamma}}{(u\bar u)^{\gamma}} +\frac{\vert r-\bar r\vert^{\gamma}}{(r\bar r)^{\gamma}} \Big)\Big]^{2}
	\end{align*}
	for any  $\gamma\in(0,\min\{H_1,H_2\})$,
	where  $\max\{H_1,H_2\}<\frac{1}{2(d+1)}$. Moreover, applying the Young type inequality, we have
	\begin{align}\label{EqLemC4g}
		\frac{\vert r-\bar r\vert+\vert u-\bar u\vert}{\min\{\eta,\delta\}}\geq\frac{\vert r-\bar r\vert}{\delta}+\frac{\vert u-\bar u\vert}{\eta} \geq\vert r-\bar r\vert^{1/\delta}\vert u-\bar u\vert^{1/\eta},
	\end{align}
	where $\delta>1$ and $\eta=\frac{\delta}{\delta-1}$,
	we can choose $\gamma\in(0,\min\{H_1,H_2\})$ and a suitably small $\beta_4\in(0,\frac{1}{2})$ satisfying $0<\beta_4<\gamma<\min\{H_1,H_2\}<1/2$ such that
	\begin{align*}
		&\int_{[\esp,s]^2}\int_{[\esp,t]^2}\Big[r^{H_1-\frac{1}{2}-\gamma}\bar u^{H_2-\frac{1}{2}-\gamma}\Big(\frac{\vert u-\bar u\vert^{\gamma}}{(u\bar u)^{\gamma}} +\frac{\vert r-\bar r\vert^{\gamma}}{(r\bar r)^{\gamma}} \Big)\Big]^{2}\\&\qquad\times(\vert r-\bar r\vert+\vert u-\bar u\vert)^{-2-2\beta_4}\,\mathrm{d}u\mathrm{d}\bar u\mathrm{d}r\mathrm{d}\bar r\\\leq&2\int_{[0,s]^2}\int_{[0,t]^2}r^{2H_1-1-2\gamma}\bar u^{2H_2-1-2\gamma}\Big(\frac{\vert u-\bar u\vert^{2\gamma}}{(u\bar u)^{2\gamma}} +\frac{\vert r-\bar r\vert^{2\gamma}}{(r\bar r)^{2\gamma}} \Big) \\&\qquad\qquad\qquad\times(\vert r-\bar r\vert+\vert u-\bar u\vert)^{-2-2\beta_4}\,\mathrm{d}u\mathrm{d}\bar u\mathrm{d}r\mathrm{d}\bar r:=C^{(2)}_{H,\gamma,\beta_4}<\infty
	\end{align*}
	for every $(s,t)\in[\ve,T]^2$.\\ As a consequence, there exists $\beta_4\in(0,\frac{1}{2})$ such that for any $\esp>0$ and $(s,t)\in[\esp,T]^2$ it holds
	\begin{align*}
		\int_{[\esp,s]^2}\int_{[\esp,t]^2}\frac{\E\left[\Vert I_{4,n}(r,u,\bar r,\bar u)\Vert^2\right]}{(\vert r-\bar r\vert+\vert u-\bar u\vert)^{2+2\beta_4}}\mathrm{d}u\mathrm{d}\bar u\mathrm{d}r\mathrm{d}\bar r<\infty.
	\end{align*}    
	Let us turn to the term $I_{5,n}(r,u,\bar r,\bar u)$. Using again Girsanov's theorem, Cauchy-Schwarz inequality and Lemma \ref{GirsanovLemEst}, we may write
	\begin{align}
		&\notag\E\left[\Vert I_{5,n}(r,u,\bar r,\bar u)\Vert^2\right]\\\leq&\notag C_{H,d,T}(\Vert b_n\Vert_{L^{\infty}_{\infty}})\E\Big[\Big\Vert \sum\limits_{m=1}^{\infty}\int_{\Xi^{m,\ell}_{\bar r,r,s}} \int_{\Delta^{m,\ell}_{\bar u,u,t}} \prod\limits_{j=1}^{m+\ell}b_n^{\prime}(s_j,t_j,x+W^{H}_{s_j,t_j})\\&\notag\qquad\times K_H(s_{m+\ell},t_{m+\ell},\bar r,\bar u)I_d\,\mathrm{d}t_{m+\ell}\ldots\mathrm{d}t_1\mathrm{d}s_{m+\ell}\ldots\mathrm{d}s_1\Big\Vert^4\Big]\\\leq&C_1\Big(\sum\limits_{m,\ell=1}^{\infty}\sum\limits_{i,j=1}^d\sum\limits_{l_1,\ldots,l_{m+\ell-1}=1}^d\Vert J_{5,n}(r,u,\bar r,\bar u,W^H)\Vert_{L^4(\Omega)}\Big)^2,\label{ExistEstimI5n}
	\end{align}
	where
	\begin{align*}
		&J_{5,n}(r,u,\bar r,\bar u,W^H)\\:=&
		\int_{\Xi^{m,\ell}_{\bar r,r,s}}\int_{\Delta^{m,\ell}_{\bar u,u,t}}\frac{\partial b_n^{(i)}}{\partial x_{l_1}}(s_1,t_1,x+W^H_{s_1,t_1})\\&\quad\times\frac{\partial b_n^{(l_1)}}{\partial x_{l_2}}(s_2,t_2,x+W^H_{s_2,t_2})\cdots\frac{\partial b_n^{(l_{m+\ell-1})}}{\partial x_{j}}(s_{m+\ell},t_{m+\ell},x+W^H_{s_{m+\ell},t_{m+\ell}})\\&\quad\times K_H(s_{m+\ell},t_{m+\ell},\bar r,\bar u)\mathrm{d}t_{m+\ell}\ldots\mathrm{d}t_1\mathrm{d}s_{m+\ell}\ldots\mathrm{d}s_1.
	\end{align*}
	Using Proposition \ref{RegulariseProp3} $p=4$, $\sum_{j=1}^{m+\ell}\ve_{j}=1$ (since $\ve_{m+\ell}=1$ and $\ve_j=0$ for $j=1,\ldots,m+\ell-1$), $\vert{\alpha}_{([j])}\vert:=\sum_{l=1}^d\alpha^{(l)}_{([j])}=1$ for all $j$, $\vert\alpha\vert=m+\ell$ and the inequality $\prod_{l=1}^d(2p\vert\alpha^{(l)}\vert)!\leq (2p\vert\alpha\vert)!$, we have
	\begin{align*}
		&\E[(J_{5,n}(r,u,\bar r,\bar u,W^H))^4]\leq\bar r^{4H_1-2}\bar u^{4H_2-2}C_2^{4(m+\ell)}\Vert b_n\Vert^{4(m+\ell)}_{L^1_{\infty} }\\&\qquad\times A^{5,0}_{d,m}(H,\vert s-\bar r\vert,\vert t-u\vert,\vert u-\bar u\vert),
	\end{align*}
	where
	\begin{align*}
		&A^{5,0}_{d,m}(H,\vert s-\bar r\vert,\vert t-u\vert,\vert u-\bar u\vert)\\
		:=&\frac{(8(m+\ell)!)^{\frac{1}{4}}(s-\bar r)^{4(H_1-\frac{1}{2}) -4(m+\ell)(1+d)H_1+4(m+\ell)}}{\Gamma\left(8(H_1-\frac{1}{2})- 8(m+\ell)(1+d)H_1+8(m+\ell)+1\right)^{1/2}}
		\\&\quad\times\sum\limits_{\pi\in\widehat{\mathcal{P}}_{m+\ell\vert 8}}\Big\{\frac{(t-u)^{\frac{1}{2}(H_2-\frac{1}{2})\sum_{k=1}^{8m} \ve_{([\pi^{-1}(k)])}-4m(1+d)H_2+4m}}{\Gamma\left((H_2-\frac{1}{2})\sum\limits_{k=1}^{8m} \ve_{([\pi^{-1}(k)])}- 8m(1+d)H_2+8m+1\right)^{1/2}}\\&\qquad\qquad\times\frac{\vert u-\bar u\vert^{\frac{1}{2}(H_2-\frac{1}{2})\sum_{j=1}^{8\ell} \ve_{([\pi^{-1}(8m+j)])}-4\ell(1+ d)H_2+4\ell}}{\Gamma\left((H_2-\frac{1}{2})\sum\limits_{j=1}^{8\ell} \ve_{([\pi^{-1}(8m+j)])}- 8\ell(1+ d)H_2+8\ell+1\right)^{1/2}}\Big\}.
	\end{align*}
	Now, since $H_2-\frac{1}{2}<0$, $\sum_{k=1}^{8m} \ve_{([\pi^{-1}(k)])}\leq\sum_{k=1}^{8(m+\ell)}\ve_{([k])}$, $\sum_{j=1}^{8\ell} \ve_{([\pi^{-1}(8m+j)])}\leq\sum_{k=1}^{8(m+\ell)}\ve_{([k])}$, $\sum_{k=1}^{8(m+\ell)}\ve_{([k])}=8$, $0\leq t-u\leq T$, $0\leq u-\bar u\leq T$ and $\Gamma$ is nondecreasing on $[2,\infty)$,
	then 
	for $m$ and $\ell$ large enough, 
	\begin{align*}
		A^{5,0}_{d,m}(H,\vert s-\bar r\vert,\vert t-u\vert,\vert u-\bar u\vert)\leq C_{T}^{m+\ell}\widetilde{A}^{5,0}_{d,m}(H,\vert s-\bar r\vert,\vert t-u\vert,\vert u-\bar u\vert),
	\end{align*}
	where
	\begin{align*}
		&\widetilde{A}^{5,0}_{d,m}(H,\vert s-\bar r\vert,\vert t-u\vert,\vert u-\bar u\vert)\\:=&  \frac{((8(m+\ell))!)^{\frac{1}{4}}(s-\bar r)^{4(H_1-\frac{1}{2}) -4(m+\ell)(1+d)H_1+4(m+\ell)}}{\Gamma\left(8(H_1-\frac{1}{2})- 8(m+\ell)(1+d)H_1+8(m+\ell)+1\right)^{1/2}}
		\\&\quad\times \frac{(t-u)^{4(H_2-\frac{1}{2}) -4m(1+d)H_2+4m}}{\Gamma\left(8(H_2-\frac{1}{2})- 8m(1+d)H_2+8m+1\right)^{1/2}}\\&\quad\times\frac{\vert u-\bar u\vert^{4(H_2-\frac{1}{2}) -4\ell(1+ d)H_2+4\ell}}{\Gamma\left(8(H_2-\frac{1}{2})- 8\ell(1+ d)H_2+8\ell+1\right)^{1/2}}.
	\end{align*}
	Hence, by \eqref{ExistEstimI5n}, we obtain
	\begin{align}\label{ExistEstimI5nb}
		&\E[\Vert I_{5,n}(r,u,\bar r,\bar u)\Vert^2]\leq\bar r^{2H_1-1}\bar u^{2H_2-1}\\&\notag\times\Big(\sum\limits_{m,\ell=1}^{\infty}d^{m+\ell+1} C_T^{m+\ell}\Vert b_n\Vert^{m+\ell}_{L^1_{\infty}}(\widetilde{A}^{5,0}_{d,m}(H,\vert s-\bar r\vert,\vert t-u\vert,\vert u-\bar u\vert))^{\frac{1}{4}}\Big)^2.
	\end{align}
	We deduce from the inequality $(8m)!(8\ell)!\leq(8(m+\ell))!\leq2^{8(m+\ell)}(8m)!(8\ell)!$ that
	\begin{align*}
		&\widetilde{A}^{5,0}_{d,m}(H,\vert s-\bar r\vert,\vert t-u\vert,\vert u-\bar u\vert)\\\leq&  \frac{2^{8(m+\ell)}((8m)!(8\ell)!)^{\frac{1}{8}}(s-\bar r)^{4(H_1-\frac{1}{2}) -4(m+\ell)(1+d)H_1+4(m+\ell)}}{\Gamma\left(8(H_1-\frac{1}{2})- 8(m+\ell)(1+d)H_1+8(m+\ell)+1\right)^{1/2}}
		\\&\quad\times \frac{((8m)!)^{\frac{1}{8}}(t-u)^{4(H_2-\frac{1}{2}) -4m(1+d)H_2+4m}}{\Gamma\left(8(H_2-\frac{1}{2})- 8m(1+d)H_2+8m+1\right)^{1/2}}\\&\quad\times\frac{((8\ell)!)^{\frac{1}{8}}\vert u-\bar u\vert^{4(H_2-\frac{1}{2}) -4\ell(1+ d)H_2+4\ell}}{\Gamma\left(8(H_2-\frac{1}{2})- 8\ell(1+ d)H_2+8\ell+1\right)^{1/2}}.
	\end{align*}
	Thus as $\max\{H_1,H_2\}<\frac{1}{2(d+1)}$, we have  
	\begin{align*}
		\Gamma\left(8(H_1-1/2)- 8(m+\ell)(1+d)H_1+8(m+\ell)+1\right)>\Gamma(4(m+\ell)+1),
	\end{align*} 
	\begin{align*}
		\Gamma\left(8(H_2-1/2)- 8m(1+d)H_2+8m+1\right)>\Gamma(4m+1)
	\end{align*}
	and
	\begin{align*}
		\Gamma\left(8(H_2-\frac{1}{2})- 8\ell(1+ d)H_2+8\ell+1\right)>\Gamma(4\ell+1)
	\end{align*}
	for $m$ and $\ell$ large enough. As a consequence, we can show (by ratio test) that the sum in \eqref{ExistEstimI5nb}  {converges} and there exists $\gamma>0$ and a continuous function $C^{(5)}_{\ve,H,d,T}:\,[0,\infty)^2\to[0,\infty)$ such that
	\begin{align*}
		\sup\limits_{n\in\N_0}\E[\Vert I_{5,n}(r,u,\bar r,\bar u)\Vert^2]\leq&\sup_{n\in\N_0}C^{(5)}_{\esp,H,d,T}(\Vert b_n\Vert_{L^1_{\infty}}) \,\bar r^{2H_1-1}\bar u^{2H_2-1}\vert u-\bar u\vert^{\gamma}.
	\end{align*}
	It follows from Young inequality that there exists $\beta_5\in(0,\gamma)$ such that for any $0<\esp<s,t<T$,
	\begin{align*}
		&\int_{[\esp,s]^2}\int_{[\esp,t]^2}\bar r^{2H_1-1}\bar u^{2H_2-1}\vert u-\bar u\vert^{\gamma}(\vert r-\bar r\vert+\vert u-\bar u\vert)^{-2-2\beta_5}\\\leq&\int_{[0,s]^2}\int_{[0,t]^2}\bar r^{2H_1-1}\bar u^{2H_2-1}\vert u-\bar u\vert^{\gamma}(\vert r-\bar r\vert+\vert u-\bar u\vert)^{-2-2\beta_5}<\infty.
	\end{align*}
	Then we can choose $\beta_5\in(0,\frac{1}{2})$ such that for any $0<\esp<s,t\leq T$,
	\begin{align*}
		\int_{[\esp,s]^2}\int_{[\esp,t]^2}\frac{\E\left[\Vert I_{5,n}(r,u,\bar r,\bar u)\Vert^2\right]}{(\vert r-\bar r\vert+\vert u-\bar u\vert)^{2+2\beta_5}}\mathrm{d}u\mathrm{d}\bar u\mathrm{d}r\mathrm{d}\bar r<\infty.
	\end{align*}
	Similarly we can find $\beta_6\in(0,\frac{1}{2})$ such that for any $0<\esp<s,t\leq T$,
	\begin{align*}
		\int_{[\esp,s]^2}\int_{[\esp,t]^2}\frac{\E\left[\Vert I_{6,n}(r,u,\bar r,\bar u)\Vert^2\right]}{(\vert r-\bar r\vert+\vert u-\bar u\vert)^{2+2\beta_6}}\mathrm{d}u\mathrm{d}\bar u\mathrm{d}r\mathrm{d}\bar r<\infty.
	\end{align*} 
	The estimates of the terms $I_{2,n}(r,u,\bar r,\bar u)$ and $I_{3,n}(r,u,\bar r,\bar u)$ are analogeous to those obtained above. Precisely, we have 
	\begin{align*}
		&\notag\E[\Vert I_{2,n}(r,u,\bar r,\bar u)\Vert^2]\\\leq&\bar r^{2H_1-1}\bar u^{2H_2-1}  \Big(\sum\limits_{\ell=1}^{\infty}d^{\ell+1} C^{\ell}\Vert b_n\Vert^{\ell}_{L^1_{\infty} }(A^{2,0}_{d,\ell}(H,\vert s-\bar r\vert,\vert u-\bar u\vert))^{\frac{1}{4}}\Big)^2
	\end{align*}
	and
	\begin{align*}
		&\notag\E[\Vert I_{3,n}(r,u,\bar r,\bar u)\Vert^2]\\\leq&\bar r^{2H_1-1}\bar u^{2H_2-1}  \Big(\sum\limits_{\ell=1}^{\infty}d^{\ell+1} C^{\ell}\Vert b_n\Vert^{\ell}_{L^1_{\infty} }(A^{3,0}_{d,\ell}(H,\vert r-\bar r\vert,\vert t-\bar u\vert))^{\frac{1}{4}}\Big)^2,
	\end{align*}
	where
	\begin{align*}
		A^{2,0}_{d,\ell}(H,\vert s-\bar r\vert,\vert u-\bar u\vert) =&\frac{((8\ell)!)^{\frac{1}{4}}(s-\bar r)^{4(H_1-\frac{1}{2}) -4\ell(1+d)H_1+4\ell}}{\Gamma\left(8(H_1-\frac{1}{2})- 8\ell(1+d)H_1+8\ell+1\right)^{1/2}}
		\\&\quad\times\frac{\vert u-\bar u\vert^{4(H_2-\frac{1}{2}) -4\ell(1+d)H_2+4\ell}}{\Gamma\left(8(H_2-\frac{1}{2}) -8\ell(1+d)H_2+8\ell+1\right)^{1/2}}
	\end{align*}
	and
	\begin{align*}
		A^{3,0}_{d,\ell}(H,\vert r-\bar r\vert,\vert t-\bar u\vert) =&\frac{((8\ell)!)^{\frac{1}{4}}\vert r-\bar r\vert^{4(H_1-\frac{1}{2}) -4\ell(1+d)H_1+4\ell}}{\Gamma\left(8(H_1-\frac{1}{2})- 8\ell(1+d)H_1+8\ell+1\right)^{1/2}}
		\\&\quad\times\frac{(t-\bar u)^{4(H_2-\frac{1}{2}) -4\ell(1+d)H_2+4\ell}}{\Gamma\left(8(H_2-\frac{1}{2}) -8\ell(1+d)H_2+8\ell+1\right)^{1/2}}.
	\end{align*}
	Then, as $\max\{H_1,H_2\}<\frac{1}{2(d+1)}$, there exist $\beta_2,\,\beta_3\in(0,\frac{1}{2})$ such that for any $0<\esp<s,t\leq T$,
	\begin{align*}
		\int_{[\esp,s]^2}\int_{[\esp,t]^2}\frac{\E\left[\Vert I_{i,n}(r,u,\bar r,\bar u)\Vert^2\right]}{(\vert r-\bar r\vert+\vert u-\bar u\vert)^{2+2\beta_i}}\mathrm{d}u\mathrm{d}\bar u\mathrm{d}r\mathrm{d}\bar r<\infty,\quad i=2,3.
	\end{align*} 
	It remains to estimate $I_{7,n}(r,u,\bar r,\bar u)$. Using  {once more} Girsanov's theorem, Cauchy-Schwarz inequality  {repeatedly} and Proposition \ref{RegulariseProp2}, we have
	\begin{align*}
		&\E\left[\Vert I_{7,n}(r,u,\bar r,\bar u)\Vert^2\right]\\\leq&C_1\Big\Vert\Big(I_d+\sum\limits_{m=1}^{\infty}\int_{\nabla^{m}_{r,s}}\int_{\nabla^{m}_{u,t}}\prod\limits_{j=1}^mb_n^{\prime}(s_j,t_j,X^{x,n}_{s_j,t_j})\,\mathrm{d}s_j\mathrm{d}t_j\Big)\Big\Vert^2_{L^8(\Omega)}\\&\times\Big\Vert\Big(\sum\limits_{m=1}^{\infty}\int_{\nabla_{\bar r,r}^{m}}\int_{\nabla_{\bar u,u}^{m}}\prod\limits_{j=1}^{m}b^{\prime}(s_j,t_j,X^{x,n}_{s_j,t_j})K_H(s_{m},t_{m},\bar r,\bar u)I_d\,\mathrm{d}t_j\mathrm{d}s_j\Big)\Big\Vert^2_{L^8(\Omega)},
	\end{align*}
	which implies that
	\begin{align*}
		&\E\left[\Vert I_{7,n}(r,u,\bar r,\bar u)\Vert^2\right]\\\leq&C_{12}\Big(1+\sum\limits_{m=1}^{\infty}\sum\limits_{i,j=1}^d\sum\limits_{l_1,\ldots,l_{m-1}=1}^d\Big\Vert\int_{\nabla^{m}_{r,s}}\int_{\nabla^{m}_{u,t}}\frac{\partial b_n^{(i)}}{\partial x_{l_1}}(s_1,t_1,x+W^H_{s_1,t_1}) \\&\notag\qquad\times\frac{\partial b_n^{(l_1)}}{\partial x_{l_2}}(s_2,t_2,x+W^H_{s_2,t_2})\cdots\frac{\partial b_n^{(l_{m-1})}}{\partial x_{j}}(s_m,t_m,x+W^H_{s_m,t_m})\\&\qquad\notag\times \mathrm{d}t_1\ldots\mathrm{d}t_m\mathrm{d}s_1\ldots\mathrm{d}s_m\Big\Vert_{L^8(\Omega)}\Big)^2\\&\times\Big(\sum\limits_{m=1}^{\infty}\sum\limits_{i,j=1}^d\sum\limits_{l_1,\ldots,l_{m-1}=1}^d\Big\Vert\int_{\nabla^{m}_{\bar r,r}}\int_{\nabla^{m}_{\bar u,u}}\frac{\partial b_n^{(i)}}{\partial x_{l_1}}(s_1,t_1,x+W^H_{s_1,t_1}) \\&\notag\qquad\times\frac{\partial b_n^{(l_1)}}{\partial x_{l_2}}(s_2,t_2,x+W^H_{s_2,t_2})\cdots\frac{\partial b_n^{(l_{m-1})}}{\partial x_{j}}(s_m,t_m,x+W^H_{s_m,t_m})\\&\qquad\notag\times K_H(s_m,t_m,\bar r,\bar u)\mathrm{d}t_1\ldots\mathrm{d}t_m\mathrm{d}s_1\ldots\mathrm{d}s_m\Big\Vert_{L^8(\Omega)}\Big)^2.
	\end{align*}
	Applying Proposition \ref{RegulariseProp1} with $p=8$,   $\sum_{l=1}^d\alpha^{(l)}_j=1$ for all $j$, $\vert\alpha\vert=m+\ell$ and the inequality $\prod_{l=1}^d(2p\vert\alpha^{(l)}\vert)!\leq (2p\vert\alpha\vert)!$, we have
	\begin{align}\label{ExistEstim7n}
		&\E\left[\Vert I_{7,n}(r,u,\bar r,\bar u)\Vert^2\right]\leq C_{12}\bar r^{2H_1-1}\bar u^{2H_2-1}\\&\notag\qquad\times \Big(1+\sum\limits_{m=1}^{\infty}d^{m+1} C^{m}\Vert b_n\Vert^{m}_{L^1_{\infty} }(A^{7,0}_{d,m}(H,\vert s-r\vert,\vert t-u\vert))^{\frac{1}{8}}\Big)^2\\&\qquad\times\Big(\sum\limits_{m=1}^{\infty}d^{m+1} C^{m}\Vert b_n\Vert^{m}_{L^1_{\infty} }(A^{7,0}_{d,m}(H,\vert r-\bar r\vert,\vert u-\bar u\vert))^{\frac{1}{8}}\Big)^2\notag,
	\end{align}
	where, for any $\bar s,\,\bar t>0$,
	\begin{align*}
		A^{7,0}_{d,m}(H,\bar s,\bar t)  :=&\frac{((16\ell)!)^{\frac{1}{4}}\bar s^{\,8(H_1-\frac{1}{2}-\gamma) -8\ell(1+d)H_1+8\ell}}{\Gamma\left(16(H_1-\frac{1}{2})- 16\ell(1+d)H_1+16\ell+1\right)^{1/2}}
		\\&\quad\times\frac{\bar t^{\,8(H_2-\frac{1}{2}) -8\ell(1+d)H_2+8\ell}}{\Gamma\left(16(H_2-\frac{1}{2}) -16\ell(1+d)H_2+16\ell+1\right)^{1/2}}.
	\end{align*}
	Since $\max\{H_1,H_2\}<\frac{1}{2(d+1)}$, both sums in \eqref{ExistEstim7n} converge and there exist $\gamma_1,\,\gamma_2>0$ and a continuous function $C^{(7)}_{\ve,H,d,T}:\,[0,\infty)^2\to[0,\infty)$ such that
	\begin{align*}
		&\sup\limits_{n\in\N_0}\E[\Vert I_{7,n}(r,u,\bar r,\bar u)\Vert^2]\\\leq&\sup_{n\in\N_0}C^{(7)}_{\esp,H,d,T}(\Vert b_n\Vert_{L^1_{\infty}}) \,\bar r^{2H_1-1}\bar u^{2H_2-1}\vert r-\bar r\vert^{\gamma_1}\vert u-\bar u\vert^{\gamma_2}.
	\end{align*}
	Then the Young inequality allows to find $\beta_7\in(0,\frac{1}{2})$ such that for any $0<\esp<s,t\leq T$,
	\begin{align*}
		\int_{[\esp,s]^2}\int_{[\esp,t]^2}\frac{\E\left[\Vert I_{7,n}(r,u,\bar r,\bar u)\Vert^2\right]}{(\vert r-\bar r\vert+\vert u-\bar u\vert)^{2+2\beta_7}}\mathrm{d}u\mathrm{d}\bar u\mathrm{d}r\mathrm{d}\bar r<\infty.
	\end{align*}
	The proof of \eqref{TheoExistEstim1} is completed by choosing $\beta=\min\{\beta_i,\,i=1,\ldots,7\}$ since
	\begin{align*}
		\E\left[\Vert D_{r,u}X^{x,n}_{s,t}-D_{\bar r,\bar u}X^{x,n}_{s,t}\Vert^2\right]\leq7\sum\limits_{i=1}^7\E\left[\Vert I_{i,n}(r,u,\bar r,\bar u)\Vert^2\right].
	\end{align*}
\end{proof}
\begin{cor}\label{CorolStrongExists}
	Let $(b_n)_{n\in\N}$ be a sequence of functions
	in $C^{\infty}_c([0,T]^2\times\R^d,\R^d)$ such that $b_n(s,t,z)$ converges to $b(s,t,z)$ for a.e. $(s,t,z)\in[0,T]^2\times\R^d$ with $\sup_{n\geq0}\Vert b_n\Vert_{L^1_{\infty} }<\infty$. Denote by $X^{x,n}$ the solution of \eqref{eqmainProbS4Xn}. Then for every $(s,t)\in[0,T]^2$ and bounded continuous function $\varphi:\,\R^d\to\R$, $\varphi(X^{x,n}_{s,t})$ converges strongly in $L^2(\Omega,\mathcal{F}^H_{s,t})$ to $\varphi(\E[X^x_{s,t}\vert\mathcal{F}^H_{s,t}])$. Moreover, for every $(s,t)\in(0,T]^2$, $\E[X^x_{s,t}\vert\mathcal{F}^H_{s,t}]$ is Malliavin differentiable.
\end{cor}	
\begin{proof}
	By Theorem \ref{RelCompacTheo}, we can apply  Lemma \ref{CorolCompact} to deduce the relative compactness of $\{X_{s,t}^{x,n},n\in\N\}$ in $L^2(\Omega)$ for every $(s,t)\in(0,T]$. Moreover, it follows from Lemma \ref{LemWeakConverge} that we can identify the limit as being $\E[X^x_{s,t}\vert\mathcal{F}^H_{s,t}]$. Then the convergence still holds for any continuous function. The Malliavin differentiability of $\E[X^x_{s,t}\vert\mathcal{F}^H_{s,t}]$ is obtained by taking $\varphi=I_d$ and applying estimate \eqref{TheoExistEstim2} together with \cite[Proposition 1.2.3]{Nua06}.
\end{proof}
We are now ready to prove the main result of this section.
\begin{proof}[Proof of Theorem \ref{TheoEqmainPbS4}]
	We first show that $X^x_{s,t}$ is $\mathcal{F}^H_{s,t}$-measurable for every $(s,t)\in(0,T]^2$. It follows from Corollary \ref{CorolStrongExists} that for any  globally Lipschitz continuous function $\varphi$, there exists a sequence $(k_n)$ such that $\varphi(X^{x,k_n})$ converges $\Pb$-a.s. to $\varphi(\E[X^x_{s,t}\vert\mathcal{F}^H_{s,t}])$ as $n$ goes to $\infty$. Moreover, by Lemma \ref{LemWeakConverge}, $\varphi(X^{x,k_n}_{s,t})$ converges weakly in $L^2(\Omega,\mathcal{F}^H_{s,t})$ and then $\Pb$-a.s. to $\E[\varphi(X^x_{s,t})\vert\mathcal{F}^H_{s,t}]$. By the uniqueness of the limit, we have
	\begin{align*}
		\varphi(\E[X^x_{s,t}\vert\mathcal{F}^H_{s,t}])=\E[\varphi(X^x_{s,t})\vert\mathcal{F}^H_{s,t}]\quad\Pb\text{-a.s., }\forall\,(s,t)\in(0,T]^2,
	\end{align*}
	which yields that $X^x_{s,t}$ is $\mathcal{F}^H_{s,t}$-measurable for every $(s,t)\in(0,T]^2$, and then $X^x$ is a strong solution.\\
	We turn now to pathwise uniqueness.  {Since weak joint uniqueness holds for \eqref{eqmainProbS4} (see Lemma \ref{lem:weakjointu}), then the desired result follows by  using a similar dual Yamada-Watanabe argument as in \cite[Theorem 3.14]{Ku07}  (see also \cite[Proof of Corollary 2.5]{CrRi22}) which states that joint weak uniqueness and strong existence imply pathwise uniqueness. Indeed, let $(X^{1,x},W^H)$ and $(X^{2,x},W^H)$ be two strong solutions with the same $W^H$. Then there exist two measurable maps $F^{1,x}$ and $F^{2,x}$ on $\mathcal{C}([0,T]^2,\R^d)$ into $\mathcal{C}([0,T]^2,\R^d)$ such that $X^{1,x}=F^{1,x}(W^H)$ and $X^{2,x}=F^{2,x}(W^H)$ a.s. By  joint uniqueness in law  $(X^{1,x},W^H)$ and $(X^{2,x},W^H)$ have the same law $\mu^x$. Let $\nu$ denote the law of $W^H$. It follows from the desintegration theorem that there exists a transition function $\eta^x$ such that $\mu^x(\mathrm{d}y,\mathrm{d}z)=\eta^x(\mathrm{d}y, z)\nu(\mathrm{d}z)$.
		As a consequence,  for any Borel subset $A$ of $\mathcal{C}([0,T]^2,\R^d)$, we have
		\begin{align*}
			&\delta_{F^{1,x}(W^H)}(A)=\Pb(F^{1,x}(W^H)\in A\vert W^H)=\Pb(X^{1,x}\in A\vert W^H)=\eta^x(A,W^H)\\=&\Pb(X^{2,x}\in A\vert W^H)=\Pb(F^{2,x}(W^H)\in A\vert W^H)=\delta_{F^{2,x}(W^H)}(A)\quad\Pb\text{-a.s.},
		\end{align*}
		where $\delta_z$ denotes the usual Dirac distribution at point $z$. Then
		\begin{align*}
			X^{1,x}=F^{1,x}(W^H)=F^{2,x}(W^H)=X^{2,x}\quad\Pb\text{-a.s.}
		\end{align*}
		and pathwise uniqueness follows.
	}
\end{proof}

{
	\begin{rem}
		Denote by $L^{\infty}_{\infty}:=L^{\infty}([0,T]^2;L^{\infty}(\R^d;\R^d))$.
		If $b\in L^{1}_{\infty}\cap L^{\infty}_{\infty}$, then Theorem \ref{TheoEqmainPbS4} still holds when   $\max\{H_1,H_2\}<\frac{1}{2(d+1)}$. Indeed, using similar computations as those in the proof of \cite[Lemma 4.3]{BNP18}, one shows that the conditions of Theorem \ref{GirsanovTheo} hold.
	\end{rem}	
}

\appendix
\section{Shuffles}\label{Appen1}

\begin{lem}\label{lemma:Shuffle}
	Let $k,\,m\in\N$, $T>0$, $\N_k=\{1,\ldots,k\}$, $\mathcal{T}=[0,T]$, $0\leq r<s\leq T$, $0\leq u<t\leq T$. Let $\mathcal{P}_k$  denote the set of permutations on $\N_k$ and $\widehat{\mathcal{P}}_{m\vert k}$ the subset of $\mathcal{P}_{mk}$ given by
	\begin{align*}
	\widehat{\mathcal{P}}_{m\vert k} =\Big\{\sigma\in\mathcal{P}^{}_{km}:\,\sigma(1+im)<\ldots<\sigma((1+i)m),\text{ for all }i\in\{0,1,\ldots,k-1\}\Big\}.
	\end{align*}  
	For every $\sigma\in\mathcal{P}_m$, we define
	\begin{align*}
		\nabla^{m,\sigma}_{r, s					}=
		\{(s_1,\ldots,s_{m})\in\mathcal{T}^{m}:\,r<s_{\sigma^{-1}(m)}<\ldots<s_{\sigma^{-1}(1)}<s\}.
	\end{align*}
	When $\sigma=Id$ is the identity , we denote
	\begin{align*}
		\nabla^{m}_{r, s					}:=\nabla_{r,s}^{m,Id}=\{(s_1,\ldots,s_m)\in\mathcal{T}^{m}:\,r<s_m<\ldots<s_1<s\}.
	\end{align*}
	Then, for any $f\in L^1(\mathcal{T})$ and $m\in\N$,
	\begin{align*}
		&\Big(\int_{
			\nabla^{m}_{ r, s			}}\int_{\nabla^{m}_{ u, t			}
		}\,\prod\limits_{i=1}^mf(s_i,t_{i})\,\mathrm{d}s_1\ldots\mathrm{d}s_m \mathrm{d}t_1\ldots\mathrm{d}t_m\Big)^k\\=&\sum\limits_{\sigma,\gamma\in\widehat{\mathcal{P}}_{m\vert k}}\int_{\nabla^{mk,\sigma}_{
				r, s
		}}\int_{\nabla^{mk,\gamma}_{
				u, t
		}}\,\prod\limits_{i=1}^{mk}f(s_i,t_{i})\,\mathrm{d}s_1\ldots\mathrm{d}s_{mk} \mathrm{d}t_1\ldots\mathrm{d}t_{mk}.
	\end{align*}	
\end{lem}
\begin{proof} 
	It  {is enough} to prove that  $\{\nabla_{r,s}^{mk,\sigma}\}_{\sigma\in\widehat{\mathcal{P}}_{m\vert k}}$ (respectively  $\{\nabla_{u,t}^{mk,\gamma}\}_{\gamma\in\widehat{\mathcal{P}}_{m\vert k}}$) is a partition of $(\nabla^{m}_{r,s})^k$ (respectively $(\nabla^{m}_{u,t})^k$). We first show that $\nabla_{r,s}^{mk,\sigma}\subset(\nabla_{r,s}^{m})^k$ (respectively $\nabla_{u,t}^{mk,\gamma}\subset(\nabla_{u,t}^{m})^k$). Let $\sigma\in\widehat{\mathcal{P}}_{m\vert k}$ and  $(s_1,\ldots,s_{mk})\in\nabla_{r,s}^{mk,\sigma}$. Since $\sigma((i+1)m)<\ldots<\sigma(1+im)$, then, by definition of $\nabla_{r,s}^{mk,\sigma}$, we have 
	\begin{align*}
		r<s_{\sigma^{-1}\circ\sigma((i+1)m)}<\ldots<s_{\sigma^{-1}\circ\sigma(1+im)}<s,\,\forall\,i\in\{0,1,\ldots,k-1\},
	\end{align*}
	which can be rewritten as
	\begin{align*}
		r<s_{ (i+1)m}<\ldots<s_{ 1+im}<s,\,\forall\,i\in\{0,1,\ldots,k-1\}.
	\end{align*}
	This means that $(s_{1+im},\ldots,s_{ (1+i)m})\in\nabla^{m}_{r,s}$ for every $i$ and then $(s_1,\ldots,s_{mk})\in(\nabla^{m}_{r,s})^k$. The proof of $\nabla_{u,t}^{mk,\gamma}\subset(\nabla_{u,t}^{m})^k$ follows analogously. Moreover, $\{\nabla_{r,s}^{mk,\sigma}\}_{\sigma}$ and $\{ \nabla_{u,t}^{mk,\gamma}\}_{\gamma}$ are clearly families of disjoint nonempty sets. Hence, it remains to show that
	\begin{align}\label{Partition1}
		(\nabla_{r,s}^{m})^k\subset\bigcup\limits_{\sigma\in\widehat{\mathcal{P}}_{m\vert k}}\nabla_{r,s}^{mk,\sigma}\quad(\text{respectively } (\nabla_{u,t}^{m})^k\subset\bigcup\limits_{\gamma\in\widehat{\mathcal{P}}_{m\vert k}}\nabla_{u,t}^{mk,\gamma}).	
	\end{align}	
	Let $(s_1,\ldots,s_{mk})\in(\nabla_{r,s}^{m})^k$.   Then
	\begin{align}\label{Partition11}
		s_{(i+1)m}<\ldots<s_{1+im}\,\text{ for all }i\in\{0,1,\ldots,k-1\}
	\end{align}
	and there exists $\sigma\in\mathcal{P}_{mk}$ such that
	\begin{align}\label{Partition12}
		s_{\sigma^{-1}(mk)}<\ldots<s_{\sigma^{-1}(1)}.
	\end{align}
	Since \eqref{Partition11} is equivalent to
	\begin{align*}
		s_{\sigma^{-1}\circ\sigma((i+1)m)}<\ldots<s_{\sigma^{-1}\circ\sigma(1+im)}\quad\text{for all }i\in\{0,1,\ldots,k-1\},
	\end{align*}
	then it follows from \eqref{Partition12} that
	\begin{align*}
		\sigma((i+1)m)<\ldots<\sigma(1+im)\, \text{ for all }i\in\{0,1,\ldots,k-1\},
	\end{align*}
	and therefore $\sigma\in\widehat{\mathcal{P}}_{m\vert k}$.
	The proof of $(\nabla_{u,t}^{m})^k\subset\bigcup\limits_{\gamma\in\widehat{\mathcal{P}}_{m\vert k}}\nabla_{u,t}^{mk,\gamma}$ follows the same lines.
	This ends the proof.
\end{proof}

\section{Some Gaussian estimates}\label{Appen2}
The next inequality may be found in \cite{LW12} (see also \cite[Lemma A.9]{BNP18})
\begin{lem}\label{LemLW12}
	Assume that $X_1,\,\ldots,\,X_n$ are real valued centered jointly Gaussian random variables, and $\Sigma=(\sigma_{i,j}:=\E[X_iX_j])_{1\leq i,j\leq n}$ is the covariance matrix, then
	\begin{align*}
		\E[\vert X_1\vert\ldots\vert X_n\vert]\leq\sqrt{\text{perm}(\Sigma)},
	\end{align*}
	where perm$(\Sigma)$ is the permanent of  $\Sigma$ defined by
	\begin{align*}
		\text{perm}(\Sigma)=\sum\limits_{\pi\in\mathcal{P}_n}\prod\limits_{j=1}^n\sigma_{j,\pi(j)},
	\end{align*}
	where $\mathcal{P}_n$ is the set of permutations on $\N_n=\{1,\ldots,n\}$.
\end{lem}

The next result can be found in \cite{CD82} (see also \cite[Lemma 2.4]{Xi97} and \cite[Lemma 4.2]{XZ02}):
\begin{lem}\label{LemCD82}
	Let $Z_1,\,\ldots,\,Z_n$ be mean zero Gaussian variables which are linearly independent and let $g$ be real valued measurable function on $\R$ such that
	\begin{align*}
		\int_{\R}g(v)\exp(-\ve v^2)\,\mathrm{d}v<\infty
	\end{align*}
	for all $\ve>0$. Then
	\begin{align*}
		&\int_{\R^n}g(v_1)\exp\Big(-\frac{1}{2}\text{Var}\Big[\sum\limits_{j=1}^nv_jZ_j\Big]\Big)\,\mathrm{d}v_1\ldots\mathrm{d}v_n\\=&\frac{(2\pi)^{(n-1)/2}}{(\text{det }\text{Cov}(Z_1,\ldots,Z_n))^{1/2}}\int_{\R}g\Big(\frac{v}{\sigma_1}\Big)\exp\Big(-\frac{1}{2}v^2\Big)\,\mathrm{d}v,
	\end{align*}
	where $\sigma_1^2:=\text{Var}[Z_1\vert Z_2,\ldots,Z_n]$ is the conditional variance of $Z_1$ given $Z_2,\,\ldots,\,Z_n$.
\end{lem}
The following estimate is a special case of \cite[Lemma 4.5]{XZ02}. This result follows from a classical theorem of Oppenheim about the Hadamard product of positive semidefinite Hermitian matrices.	
\begin{lem}\label{LemOPXZ02}
	Let $(W^{H,(0)}_{s,t},(s,t)\in\R_+^2)$ be a real valued fractional Brownian sheet with Hurst index $H=(H_1,H_2)$. For any $n\in\N$, $n\geq2$, and $(s_1,t_1),\,\ldots,(s_n,t_n)\in\R_+^2$,
	\begin{align*}
		&\text{det Cov}\left(W^{H,(0)}_{s_1,t_1},\,\ldots,\,W^{H,(0)}_{s_n,t_n}\right)\\
		&\geq\text{det Cov}\left(B^{H_1}_{s_1},\,\ldots,\,B^{H_1}_{s_n}\right)\times\text{det Cov}\left(B^{H_2}_{t_1},\,\ldots,\,B^{H_2}_{t_n}\right),
	\end{align*}
	where $(B^{H_i}_t,t\geq0)$, $i=1,2$ is a  one parameter fractional Brownian motion in $\R$ with Hurst index $H_i$.
\end{lem}	

\section{Further technical results}\label{Appen3}
Here is a compactness criterion for subsets of $L^2(\Omega)$ that has been proved in \cite[Theorem 1]{DaPMN92}.
\begin{thm}\label{theoCompactApp}
	Let $\{(\Omega,\mathcal{A},\Pb);H\}$ be a Gaussian probability space, that is $(\Omega,\mathcal{A},\Pb)$ is a probability space and $H$ a separable closed subspace of Gaussian random variables of $L^2(\Omega)$, which generate the $\sigma$-field $\mathcal{A}$. Denote by $\mathbf{D}$ the derivative operator acting on elementary smooth random variables in the sense that
	\begin{align*}
		\mathbf{D}(f(h_1,\ldots,h_n))=\sum\limits_{i=1}^n\partial_if(h_1,\ldots,h_n)h_i,\quad h_i\in H,\,f\in C^{\infty}_b(\R^n).
	\end{align*}
	Further let $\mathbb{D}^{1,2}$ be the closure of the family of elementary smooth random variables with respect to the norm
	\begin{align*}
		\Vert F\Vert_{1,2}:=\Vert F\Vert_{L^2(\Omega)}+\Vert\mathbf{D}F\Vert_{L^2(\Omega;H)}.
	\end{align*}
	Assume that $C$ is a self-adjoint compact operator on $H$ with dense image. Then for any $c>0$ the set
	\begin{align*}
		\mathcal{G}=\left\{G\in\mathbb{D}^{1,2}:\,\Vert G\Vert_{L^2(\Omega)}+\Vert C^{-1}\mathbf{D}G\Vert_{L^2(\Omega;H)}\leq c\right\}
	\end{align*}
	is relatively compact in $L^2(\Omega)$.
\end{thm}

In order to apply the above result, we consider the fractional Sobolev space: 
$$
\mathbb{G}^2_{\beta}(U;\R):=\Big\{g\in L^2(U,\R):\,\int_U\int_U\frac{\vert g(u)-g(u')\vert^2}{\vert u-u'\vert^{p+2\beta}}\mathrm{d}u\mathrm{d}u'<\infty\Big\},
$$ 
where $U$ is a domain of $\mathbb{R}^p$, $p\geq1$ and the norm is given by 
$$
\Vert g\Vert_{\mathbb{G}^2_{\beta}(U;\R)}:=\Vert g\Vert_{L^2(U;\R)}+\Big(\int_U\int_U\frac{\vert g(u)-g(u')\vert^2}{\vert u-u'\vert^{p+2\beta}}\mathrm{d}u\mathrm{d}u'\Big)^{1/2}.
$$
We need the next compact embedding result  from \cite[Lemma 10]{PSV13} (see also \cite[Theorem 7.1]{DNPV12}).
\begin{lem}\label{lemCompact}
	Let $p\geq1$, $U\subset\mathbb{R}^p$ be a Lipschitz bounded open set and $\mathcal{J}$ be a bounded subset of $L^2(U;\R)$. Suppose that
	$$
	\sup\limits_{g\in\mathcal{J}}\int_U\int_U\frac{\vert g(u)-g(u')\vert^2}{\vert u-u'\vert^{p+2\beta}}\mathrm{d}u\mathrm{d}u'<\infty
	$$
	for some $\beta\in(0,1)$. Then $\mathcal{J}$ is relatively compact in $L^2(U;\R)$.
\end{lem}
As a consequence of Theorem \ref{theoCompactApp} and Lemma \ref{lemCompact}, we have the following compactness criterion for subsets in the space $L^2(\Omega,\mathbb{R}^d)$. 
\begin{cor}\label{CorolCompact}
	Denote by $\mathcal{F}^{}_{T}$ the $\sigma$-algebra generated by the $d$-dimensional fractional Brownian sheet $W^H=(W^{H,(1)},\ldots,W^{H,(d)})$ with Hurst index $H=(H_1,H_2)$. Let $(X^{(n)},n\in\N)$ be a sequence of $(\mathcal{F}^{}_{T},\mathcal{B}(\mathbb{R}^d))$-measurable random variables and let $D_{r,u}$ be the Malliavin derivative associated with the random vector $W^H_{r,u}=(W_{r,u}^{H,(1)},\ldots,W_{r,u}^{H,(d)})$. Suppose 
	\begin{align}\label{EqSeqXn1}
		\sup\limits_{n\in\N}\Vert X^{(n)}\Vert_{L^2(\Omega,\mathbb{R}^d)}<\infty\,\text{ and }\,\sup\limits_{n\in\N}\Vert D_{\cdot,\cdot}X^{(n)}\Vert_{L^2(\Omega\times[\esp,s]\times[\esp,t],\mathbb{R}^{d\times d})}<\infty
	\end{align}
	as well as
	\begin{align}\label{EqSeqXn2}
		\sup\limits_{n\in\N}\int_{[\esp,s]^2\times[\esp,t]^2}\frac{\E\left[\vert D_{r,u}X^{(n)}-D_{\bar r,\bar u}X^{(n)}\vert^2\right]}{(\vert{}r-\bar r\vert{}+\vert{}u-\bar u\vert{})^{2+2\beta}}\mathrm{d}u\mathrm{d}\bar u\mathrm{d} r\mathrm{d}\bar r<\infty
	\end{align}
	for some $(s,t)\in(0,T]^2$ and $\esp\in(0,\min\{s,t\})$, where $\beta\in(0,1)$ is independent of $\esp,\,s,\,t$ and $\vert\cdot\vert$ denotes any matrix norm. Then $(X^{(n)},n\in\N)$ is relatively compact in $L^2(\Omega,\mathbb{R}^d)$.
\end{cor}
\begin{proof} 
	The proof is inspired from \cite[Section 5]{HP14}. We consider the symmetric form $\mathcal{L}$ on $L^2([\esp,s]\times[\esp,t],\R^{d\times d})$ defined as
	\begin{align*}
		\mathcal{L}(f,g)=&\int_{[\esp,s]\times[\esp,t]} f(r,u)\cdot g(r,u)\mathrm{d}r\mathrm{d}u\\&+\int_{([\esp,s]\times[\esp,t])^2} \frac{(f(r,u)-f(\bar r,\bar u))\cdot(g(r,u)-g(\bar r,\bar u))}{(|r-\bar r|+|u-\bar u|)^{2+2\beta}}\mathrm{d}r\mathrm{d}\bar r\mathrm{d}u\mathrm{d}\bar u
	\end{align*}
	for functions $f$, $g$ in the dense domain $ \mathcal{D}(\mathcal{L})\subset L^2([\esp,s]\times[\esp,t],\R^{d\times d})$ and a fixed $\beta\in(0,1)$, where
	\begin{align*}
		\mathcal{D}(\mathcal{L})=&\Big\{g:\,\Vert g\Vert^2_{L^2([\esp,s]\times[\esp,t],\R^{d\times d})}\\&\qquad+\int_{([\esp,s]\times[\esp,t])^2}\frac{|g(r,u)-g(\bar r,\bar u)|^2}{(|r-\bar r|+|u-\bar u|)^{2+2\beta}}\mathrm{d}r\mathrm{d}\bar r\mathrm{d}u\mathrm{d}\bar u<\infty\Big\}.
	\end{align*}
	Then $\mathcal{L}$ is a positive symmetric closed form and, by Kato's first representation theorem (see e.g. \cite{Ka13}), one can find a positive self-adjoint operator $T_{\mathcal{L}}$ such that
	\begin{align*}
		\mathcal{L}(f,g)=\langle f,T_{\mathcal{L}}g\rangle_{L^2([\esp,s]\times[\esp,t],\R^{d\times d})}
	\end{align*}
	for all $g\in\mathcal{D}(T_{\mathcal{L}})$ and $f\in\mathcal{D}(\mathcal{L})$. Further, one may observe that the form $\mathcal{L}$ is bounded from below by a positive number. Indeed, 
	\begin{align}\label{EqOpMatcalE}
		\mathcal{L}(g,g)\geq\Vert g\Vert_{L^2([\esp,s]\times[\esp,t],\R^{d\times d})}
	\end{align}
	for all $g\in\mathcal{D}(\mathcal{L})$. Hence, we have that $\mathcal{D}(\mathcal{L})=\mathcal{D}(T^{1/2}_{\mathcal{L}})$. 
	Now, define the operator $A$ as $A=T^{1/2}_{\mathcal{L}}$. It follows from Lemma 1 in \cite[Section 1]{Le82} (see also \cite[Lemma 9]{HP14}) and Lemma \ref{lemCompact} 
	applied to $p=2$, $U=[\esp,s]\times[\esp,t]$ that $A$ has a discrete spectrum and a compact inverse $A^{-1}$. Then, using (C1)
	and (C2)
	, the operator $A$ and the sequence $(X^{(n)},n\in\N)$ satisfy the assumptions of Theorem \ref{theoCompactApp}.
	.
\end{proof}

In order to apply the preceding compact criterion we will need the following estimates.
\begin{lem}\label{ExistEstLem1}
	Let $H=(H_1,H_2)\in(0,1/2)$ and $(s,t)\in[0,T]^2$ be fixed. Then, there exists $\beta\in(0,1/2)$ such that
	\begin{align}
		\int_{[0,s]^2}\int_{[0,t]^2}\frac{\vert K_H(s,t,r,u)-K_H(s,t,\bar r,\bar u)\vert^2}{(\vert r-\bar r\vert+\vert u-\bar u\vert)^{2+2\beta}}\mathrm{d}u\mathrm{d}\bar u\mathrm{d}r\mathrm{d}\bar r<\infty.
	\end{align}
\end{lem}
\begin{proof} 
	Let $H=(H_1,H_2)\in(0,1/2)$, $0\leq \bar r,r\leq s \leq T$ and $0\leq\bar u,u\leq t\leq T$, we have
	\begin{align}\label{EqLemC4a}
		&\notag\vert K_H(s,t,r,u)-K_H(s,t,\bar r,\bar u)\vert^2\\=&\notag\vert K_{H_1}(s,r)K_{H_2}(t,u)-K_{H_1}(s,\bar r)K_{H_2}(t,\bar u)\vert^2\\=&\notag\left\vert K_{H_1}(s,r)(K_{H_2}(t,u)-K_{H_2}(t,\bar u))+(K_{H_1}(s,r)-K_{H_1}(s,\bar r))K_{H_2}(t,\bar u)\right\vert^2\\ \leq&2\left(\vert K_{H_1}(s,r)\vert^2\vert K_{H_2}(t,u)-K_{H_2}(t,\bar u)\vert^2+\vert K_{H_1}(s,r)-K_{H_1}(s,\bar r)\vert^2\vert K_{H_2}(t,\bar u)\vert^2\right).
	\end{align} 
	By Lemma A.4 in \cite{BNP18}, there exists $\beta_1,\,\beta_2\in(0,1/2)$ such that 
	\begin{align*}
		&	\int_{[0,s]^2}\frac{\vert K_{H_1}(s,r)-K_{H_1}(s,\bar r)\vert^2}{\vert r-\bar r\vert^{1+2\beta_1}}\,\mathrm{d}r\mathrm{d}\bar r<\infty
	\end{align*}
	and
	\begin{align*}
		\int_{[0,t]^2}\frac{\vert K_{H_2}(t,u)-K_{H_2}(t,\bar u)\vert^2}{\vert u-\bar u\vert^{1+2\beta_2}}\,\mathrm{d}u\mathrm{d}\bar u	<\infty.
	\end{align*}
	Let $\beta\in(0,\min\{\beta_1,\beta_2\})$. Applying Young's inequality $xy\leq\frac{1}{\eta}x^{\eta}+\frac{1}{\delta}x^{\delta}$ with $\eta=\frac{2+2\beta}{1+2\beta_1}$, $\delta=\frac{\eta}{\eta-1}=\frac{2+2\beta}{1-2(\beta_1-\beta)}$, $x=\vert r-\bar r\vert^{1/\delta}$, $y=\vert u-\bar u\vert^{1/\eta}$, we have
	
	\begin{align}\label{EqLemC4b}
		\frac{\vert r-\bar r\vert+\vert u-\bar u\vert}{\min\{\eta,\delta\}}\geq\frac{\vert r-\bar r\vert}{\delta}+\frac{\vert u-\bar u\vert}{\eta} \geq\vert r-\bar r\vert^{1/\delta}\vert u-\bar u\vert^{1/\eta}.
	\end{align}
	We also have the following estimates (see e.g. \cite[Page 45]{BNP18})
	\begin{align}\label{EqLemC4c}
		\vert K_{H_1}(s,r)\vert\leq C_{H_1,T}\vert s-r\vert^{H_1-1/2}r^{H_1-1/2}
	\end{align}
	and
	\begin{align}\label{EqLemC4d}
		\vert K_{H_2}(t,\bar u)\vert\leq C_{H_2,T}\vert t-\bar u\vert^{H_1-1/2}\bar u^{H_1-1/2}
	\end{align}
	for every $0<r<s<T$, $0<\bar u<t<T$ and some positive constants $C_{H_1,T},\,C_{H_2,T}$.
	We deduce from \eqref{EqLemC4b} and \eqref{EqLemC4c} that
	\begin{align}
		&\notag\int_{[0,s]^2}\int_{[0,t]^2}\frac{\vert K_{H_1}(s,r)\vert^2\vert K_{H_2}(t,u)-K_{H_2}(t,\bar u)\vert^2}{(\vert r-\bar r\vert+\vert u-\bar u\vert)^{2+2\beta}}\mathrm{d}u\mathrm{d}\bar u\mathrm{d}r\mathrm{d}\bar r\\ \notag\leq&C_{\eta,\delta}\Big(\int_{[0,s]^2}\frac{\vert K_{H_1}(s,r)\vert^2}{\vert r-\bar r\vert^{(2+2\beta)/\delta}}\,\mathrm{d}r\mathrm{d}\bar r\Big)\Big(\int_{[0,t]^2}\frac{\vert K_{H_2}(t,u)-K_{H_2}(t,\bar u)\vert^2}{\vert u-\bar u\vert^{(2+2\beta)/\eta}}\,\mathrm{d}u\mathrm{d}\bar u\Big)\\ \notag=&C_{\eta,\delta}\Big(\int_{[0,s]^2}\frac{\vert K_{H_1}(s,r)\vert^2}{\vert r-\bar r\vert^{1-2(\beta_1-\beta)}}\,\mathrm{d}r\mathrm{d}\bar r\Big)\Big(\int_{[0,t]^2}\frac{\vert K_{H_2}(t,u)-K_{H_2}(t,\bar u)\vert^2}{\vert u-\bar u\vert^{1+2\beta_1}}\,\mathrm{d}u\mathrm{d}\bar u\Big)
		\\ \leq&C_{\eta,\delta}C^2_{H_1,T}\Big(\int_{[0,s]^2}\vert s-r\vert^{2H_1-1}r^{2H_1-1}\vert r-\bar r\vert^{2(\beta_1-\beta)-1}\,\mathrm{d}\bar r\mathrm{d}r\Big)\label{EqLemC4e}\\&\qquad\times\Big(\int_{[0,t]^2}\frac{\vert K_{H_2}(t,u)-K_{H_2}(t,\bar u)\vert^2}{\vert u-\bar u\vert^{1+2\beta_1}}\,\mathrm{d}u\mathrm{d}\bar u\Big)
		<\infty.\notag
	\end{align}
	Similarly, using \eqref{EqLemC4d} we show that
	\begin{align}\label{EqLemC4f}
		\int_{[0,s]^2}\int_{[0,t]^2}\frac{\vert K_{H_1}(s,r)-K_{H_1}(s,\bar r)\vert^2\vert K_{H_2}(t,\bar u)\vert^2}{(\vert r-\bar r\vert+\vert u-\bar u\vert)^{2+2\beta}}\mathrm{d}u\mathrm{d}\bar u\mathrm{d}r\mathrm{d}\bar r<\infty.
	\end{align}
	It follows from \eqref{EqLemC4a}, \eqref{EqLemC4e} and \eqref{EqLemC4f} that
	\begin{align*}
		&\int_{[0,s]^2}\int_{[0,t]^2}\frac{\vert K_H(s,t,r,u)-K_H(s,t,\bar r,\bar u)\vert^2}{(\vert r-\bar r\vert+\vert u-\bar u\vert)^{2+2\beta}}\mathrm{d}u\mathrm{d}\bar u\mathrm{d}r\mathrm{d}\bar r\\ \leq&2\Big\{
		\int_{[0,s]^2}\int_{[0,t]^2}\frac{\vert K_{H_1}(s,r)\vert^2\vert K_{H_2}(t,u)-K_{H_2}(t,\bar u)\vert^2}{(\vert r-\bar r\vert+\vert u-\bar u\vert)^{2+2\beta}}\mathrm{d}u\mathrm{d}\bar u\mathrm{d}r\mathrm{d}\bar r\\&\qquad+\int_{[0,s]^2}\int_{[0,t]^2}\frac{\vert K_{H_1}(s,r)-K_{H_1}(s,\bar r)\vert^2\vert K_{H_2}(t,\bar u)\vert^2}{(\vert r-\bar r\vert+\vert u-\bar u\vert)^{2+2\beta}}\mathrm{d}u\mathrm{d}\bar u\mathrm{d}r\mathrm{d}\bar r
		\Big\}<\infty,
	\end{align*}
	which completes the proof.
\end{proof}

\begin{lem}\label{ExistEstLem2}
	Let $\esp>0$, $H=(H_1,H_2)\in(0,1/2)^2$, $\esp\leq r,\bar r<s\leq T$, $\esp\leq u,\bar u<t\leq T$ $\sigma,\pi\in\mathcal{P}_m$ and $(\ve_1,\ldots,\ve_m)\in\{0,1\}^m$ be fixed. Let $v_j,\,w_j$ be real numbers such that $v_j +(H_1-\frac{1}{2}-\gamma)\ve_{\sigma^{-1}(j)}>-1$ and   $w_j +(H_2-\frac{1}{2}-\gamma)\ve_{\pi^{-1}(j)}>-1$ for all $j=1,\ldots,m$. Then there exists a positive constant $C=C_{H,T}$ such that
	\begin{align*}
		&\int_{\nabla^{m}_{r,s}}\int_{\nabla^{m}_{u,t}}\prod\limits_{j=1}^m\vert K_H(s_{\sigma(j)},t_{\pi(j)},r,u)-K_H(s_{\sigma(j)},t_{\pi(j)},\bar r,\bar u)\vert^{\ve_j}\\&\qquad\qquad\times\vert s_j-s_{j+1}\vert^{v_j}\vert t_j-t_{j+1}\vert^{w_j}\mathrm{d}t_1\ldots\mathrm{d}t_m\mathrm{d}s_1\ldots\mathrm{d}s_m\\&\leq C^m  \Pi_{\gamma,m}(v,H_1)\Pi_{\gamma,m}(w,H_2)\\&\quad\times\Big[r^{H_1-\frac{1}{2}-\gamma}\bar u^{H_2-\frac{1}{2}-\gamma}\Big( \frac{\vert u-\bar u\vert^{\gamma}}{(u\bar u)^{\gamma}} +\frac{\vert r-\bar r\vert^{\gamma}}{(r\bar r)^{\gamma}} \Big)\Big]^{\sum_{j=1}^m\ve_j}\\&\quad\times(s-r)^{(H_1-\frac{1}{2}-\gamma)\sum_{j=1}^m \ve_{j}+\sum_{j=1}^mv_j+m}(t-u)^{(H_2-\frac{1}{2}-\gamma)\sum_{j=1}^m \ve_{j}+\sum_{j=1}^mw_j+m},
	\end{align*}
	where    
	\begin{align*}
		\Pi_{\gamma,m}(v,H_1):=\frac{\prod\limits_{j=2}^{m}\Gamma\left(1+v_j\right) \prod\limits_{j=1}^{m}\Gamma\left((H_1-\frac{1}{2}-\gamma) \ve_{\sigma^{-1}(j)}+ v_j+1\right)}{\Gamma\Big((H_1-\frac{1}{2}-\gamma)\sum\limits_{j=1}^{m} \ve_{j}+ \sum\limits_{j=1}^{m}v_{j}+m+1\Big)}   
	\end{align*}
	and
	\begin{align*}
		\Pi_{\gamma,m}(w,H_2):=\frac{\prod\limits_{j=2}^{m}\Gamma\left(1+w_j\right) \prod\limits_{j=1}^{m}\Gamma\left((H_2-\frac{1}{2}-\gamma) \ve_{\pi^{-1}(j)}+ w_j+1\right)}{\Gamma\Big((H_2-\frac{1}{2}-\gamma)\sum\limits_{j=1}^{m} \ve_{j}+ \sum\limits_{j=1}^{m}w_{j}+m+1\Big)}.   
	\end{align*}
	
\end{lem}
\begin{proof} 
	Set
	\begin{align*}
		\mathcal{J}^{(m)}_{\sigma,\pi}:=&\int_{\nabla^{m}_{r,s}}\int_{\nabla^{m}_{u,t}}\prod\limits_{j=1}^m\vert K_H(s_{\sigma(j)},t_{\pi(j)},r,u)-K_H(s_{\sigma(j)},t_{\pi(j)},\bar r,\bar u)\vert^{\ve_j}\\&\qquad\qquad\times\vert s_j-s_{j+1}\vert^{v_j}\vert t_j-t_{j+1}\vert^{w_j}\mathrm{d}t_1\ldots\mathrm{d}t_m\mathrm{d}s_1\ldots\mathrm{d}s_m.
	\end{align*}	
	We start by recalling the following estimates that have been proved in \cite[Appendix]{BNP18}. For any $0\leq \bar r,r\leq s\leq T$, $0\leq \bar u,u\leq t\leq T$ and $\gamma\in(0,\min\{H_1,H_2\})$,
	\begin{align}\label{BanApp1a}
		\vert K_{H_1}(s,r)\vert\leq C_{H_1,T}(  s-r)^{H_1-\frac{1}{2}}r^{H_1-\frac{1}{2}},
	\end{align}
	\begin{align}\label{BanApp1b}
		\vert K_{H_2}(t,u)\vert\leq C_{H_2,T}(  t-u)^{H_2-\frac{1}{2}}u^{H_2-\frac{1}{2}},
	\end{align}
	\begin{align}\label{BanApp1c}
		\vert K_{H_1}(s,r)-K_{H_1}(s,\bar r)\vert\leq C_{H_1,T}\frac{\vert r-\bar r\vert^{\gamma}}{(r \bar r)^{\gamma}}r^{H_1-\frac{1}{2}-\gamma}(s-r)^{H_1-\frac{1}{2}-\gamma}
	\end{align}	
	and
	\begin{align}\label{BanApp1d}
		\vert K_{H_2}(t,u)-K_{H_2}(t,\bar u)\vert\leq C_{H_2,T}\frac{\vert u-\bar u\vert^{\gamma}}{(u \bar u)^{\gamma}}u^{H_2-\frac{1}{2}-\gamma}(t-u)^{H_2-\frac{1}{2}-\gamma}.
	\end{align}
	It follows from \eqref{BanApp1a} and \eqref{BanApp1b} that for any $\gamma\in(0,H_1\wedge H_2)$ and $j\in\{1,\ldots,m\}$,
	\begin{align*}
		&\vert K_H(s_{\sigma(j)},t_{\pi(j)},r,u)-K_H(s_{\sigma(j)},t_{\pi(j)},\bar r,\bar u)\vert\\=&\vert K_{H_1}(s_{\sigma(j)},r)(K_{H_2}(t_{\pi(j)},u)-K_{H_2}(t_{\pi(j)},\bar u))\\&\quad+K_{H_2}(t_{\pi(j)},\bar u)(K_{H_1}(s_{\sigma(j)},r)-K_{H_1}(s_{\sigma(j)},\bar r))\vert\\ \leq&C^{(1)}_{T,H}\Big(r^{H_1-\frac{1}{2}}u^{H_2-\frac{1}{2}-\gamma}\frac{\vert u-\bar u\vert^{\gamma}}{(u\bar u)^{\gamma}}+r^{H_1-\frac{1}{2}-\gamma}\bar u^{H_2-\frac{1}{2}}\frac{\vert r-\bar r\vert^{\gamma}}{(r\bar r)^{\gamma}}\Big)\\&\times\left[(s_{\sigma(j)}-r)^{H_1-\frac{1}{2}}(t_{\pi(j)}-u)^{H_2-\frac{1}{2}-\gamma}+(s_{\sigma(j)}-r)^{H_1-\frac{1}{2}-\gamma}(t_{\pi(j)}-\bar u)^{H_2-\frac{1}{2}}\right]\\ \leq&C^{(2)}_{T,H}\Big(r^{H_1-\frac{1}{2}}u^{H_2-\frac{1}{2}-\gamma}\frac{\vert u-\bar u\vert^{\gamma}}{(u\bar u)^{\gamma}}+r^{H_1-\frac{1}{2}-\gamma}\bar u^{H_2-\frac{1}{2}}\frac{\vert r-\bar r\vert^{\gamma}}{(r\bar r)^{\gamma}}\Big)\\&\times(s_{\sigma(j)}-r)^{H_1-\frac{1}{2}-\gamma}(t_{\pi(j)}-u)^{H_2-\frac{1}{2}-\gamma}.	
	\end{align*}
	As a consequence,
	\begin{align*}
		\mathcal{J}^{(m)}_{\sigma,\pi}\leq& (C^{(2)}_{T,H})^m\Big(r^{H_1-\frac{1}{2}}u^{H_2-\frac{1}{2}-\gamma}\frac{\vert u-\bar u\vert^{\gamma}}{(u\bar u)^{\gamma}}+r^{H_1-\frac{1}{2}-\gamma}\bar u^{H_2-\frac{1}{2}}\frac{\vert r-\bar r\vert^{\gamma}}{(r\bar r)^{\gamma}}\Big)^{\sum_{j=1}^m\ve_j}\\&\quad\times\Big(\int_{\nabla^{m}_{r,s}}\prod\limits_{j=1}^m(s_{\sigma(j)}-r)^{(H_1-\frac{1}{2}-\gamma)\ve_j}\vert s_j-s_{j+1}\vert^{v_j}\,\mathrm{d}s_1\ldots\mathrm{d}s_m\Big)\\&\quad\times\Big(\int_{\nabla^{m}_{u,t}}\prod\limits_{j=1}^m(t_{\pi(j)}-r)^{(H_2-\frac{1}{2}-\gamma)\ve_j}\vert t_j-t_{j+1}\vert^{w_j}\,\mathrm{d}t_1\ldots\mathrm{d}t_m\Big)\\=&(C^{(2)}_{T,H})^m\Big(r^{H_1-\frac{1}{2}}u^{H_2-\frac{1}{2}-\gamma}\frac{\vert u-\bar u\vert^{\gamma}}{(u\bar u)^{\gamma}}+r^{H_1-\frac{1}{2}-\gamma}\bar u^{H_2-\frac{1}{2}}\frac{\vert r-\bar r\vert^{\gamma}}{(r\bar r)^{\gamma}}\Big)^{\sum_{j=1}^m\ve_j}\\&\quad\times\Big(\int_{\nabla^{m}_{r,s}}\prod\limits_{j=1}^m(s_{j}-r)^{(H_1-\frac{1}{2}-\gamma)\ve^{\sigma}_j}\vert s_j-s_{j+1}\vert^{v_j}\,\mathrm{d}s_1\ldots\mathrm{d}s_m\Big)\\&\quad\times\Big(\int_{\nabla^{m}_{u,t}}\prod\limits_{j=1}^m(t_{j}-r)^{(H_2-\frac{1}{2}-\gamma)\ve^{\pi}_j}\vert t_j-t_{j+1}\vert^{w_j}\,\mathrm{d}t_1\ldots\mathrm{d}t_m\Big),
	\end{align*}
	where $\ve^{\sigma}_j=\ve_{\sigma^{-1}(j)}$ and $\ve^{\pi}_j=\ve_{\pi^{-1}(j)}$.
	
	Now one can see that
	\begin{align}
		&\notag r^{H_1-\frac{1}{2}}u^{H_2-\frac{1}{2}-\gamma}\frac{\vert u-\bar u\vert^{\gamma}}{(u\bar u)^{\gamma}}+r^{H_1-\frac{1}{2}-\gamma}\bar u^{H_2-\frac{1}{2}}\frac{\vert r-\bar r\vert^{\gamma}}{(r\bar r)^{\gamma}}\\\leq& C^{(3)}_{T,H}r^{H_1-\frac{1}{2}-\gamma}\bar u^{H_2-\frac{1}{2}-\gamma}\Big(\frac{\vert u-\bar u\vert^{\gamma}}{(u\bar u)^{\gamma}}+\frac{\vert r-\bar r\vert^{\gamma}}{(r\bar r)^{\gamma}}\Big).
	\end{align}
	Further, one has
	\\
	\textbf{Claim 1. }
	For any $m\geq2$,
	\begin{align}
		&\notag\int_{\nabla^{m}_{r,s}}\prod\limits_{j=1}^m( s_{j}-r)^{(H_1-\frac{1}{2}-\gamma)\ve^{\sigma}_{j}}\vert  s_{j}-s_{j+1}\vert^{v_j}\,\mathrm{d} s_{m}\ldots\mathrm{d} s_{1}\\=&\notag\int_r^s\int_r^{ s_1}\cdots\int_{r}^{ s_{m-1}}\prod\limits_{j=1}^m(\ s_{j}-r)^{(H_1-\frac{1}{2}-\gamma)\ve^{\sigma}_{j}}\vert  s_{j}-s_{j+1}\vert^{ v_j}\,\mathrm{d} s_{m}\ldots\mathrm{d} s_{1}\\=&\label{BetaFormula2a}\frac{\prod\limits_{j=2}^m\Gamma\left(1+v_j\right) \prod\limits_{j=1}^{m}\Gamma\left((H_1-\frac{1}{2}-\gamma)\sum\limits_{\ell=1}^{j}\ve^{\sigma}_{m-\ell+1}+\sum\limits_{\ell=1}^{j}v_{m-\ell+1} +j\right)}{\prod\limits_{j=2}^{m}\Gamma\left((H_1-\frac{1}{2}-\gamma)\sum\limits_{\ell=1}^{j-1}\ve^{\sigma}_{m-\ell+1}+ \sum\limits_{\ell=1}^{j}v_{m-\ell+1}+j\right)} \\&\qquad\times\frac{( s_{}-r)^{(H_1-\frac{1}{2}-\gamma)\sum\limits_{\ell=0}^{m-1}\ve^{\sigma}_{m-\ell}+\sum\limits_{\ell=0}^{m-1}v_{m-\ell}+m }}{\Gamma\left((H_1-\frac{1}{2}-\gamma)\sum\limits_{\ell=0}^{m-1}\ve^{\sigma}_{m-\ell}+\sum\limits_{\ell=0}^{m-1}v_{m-\ell}+m+1 \right)}\notag
	\end{align}
	and
	\begin{align}
		&\notag\int_{\nabla^{m}_{u,t}}\prod\limits_{j=1}^m( t_{j}-u)^{(H_2-\frac{1}{2}-\gamma)\ve^{\pi}_{j}}\vert  t_{j}- t_{j+1}\vert^{ w_j}\,\mathrm{d} t_{m}\ldots\mathrm{d} t_{1}\\=&\label{BetaFormula2b}\frac{\prod\limits_{j=2}^m\Gamma\left(1+w_j\right) \prod\limits_{j=1}^{m}\Gamma\left((H_2-\frac{1}{2}-\gamma)\sum\limits_{\ell=1}^{j}\ve^{\pi}_{m-\ell+1}+\sum\limits_{\ell=1}^{j}w_{m-\ell+1} +j\right)}{\prod\limits_{j=2}^{m}\Gamma\left((H_2-\frac{1}{2}-\gamma)\sum\limits_{\ell=1}^{j-1}\ve^{\pi}_{m-\ell+1}+ \sum\limits_{\ell=1}^{j}w_{m-\ell+1}+j\right)} \\&\qquad\times\frac{( t_{}-u)^{(H_2-\frac{1}{2}-\gamma)\sum\limits_{\ell=0}^{m-1}\ve^{\pi}_{m-\ell}+\sum\limits_{\ell=0}^{m-1}w_{m-\ell}+m }}{\Gamma\left((H_2-\frac{1}{2}-\gamma)\sum\limits_{\ell=0}^{m-1}\ve^{\pi}_{m-\ell}+\sum\limits_{\ell=0}^{m-1}w_{m-\ell}+m+1 \right)}\notag,
	\end{align}
	where $s_{m+1}=r$ and $t_{m+1}=u$.
	\begin{proof}[Proof of the claim 1.]
		We first prove by induction on $k=1,\ldots,m-1$ that
		\begin{align}
			&\notag\int_r^{ s_{m-k}}\cdots\int_{r}^{ s_{m-1}}(s_m-r)^{v_m}\prod\limits_{j=m-k+1}^m( s_{j}-r)^{(H_1-\frac{1}{2}-\gamma)\ve^{\sigma}_{j}}\vert  s_{j-1}- s_{j}\vert^{v_{j-1}}\\&\notag\qquad\qquad\qquad\times\mathrm{d} s_{m}\ldots\mathrm{d} s_{m-k+1}\\=&\prod\limits_{j=1}^{k}\frac{\Gamma\left((H_1-\frac{1}{2}-\gamma)\sum\limits_{\ell=0}^{j-1}\ve^{\sigma}_{m-\ell}+ \sum\limits_{\ell=0}^{j-1}v_{m-\ell}+j\right)\Gamma\left(1+v_{m-j+1}\right)}{\Gamma\left((H_1-\frac{1}{2}-\gamma)\sum\limits_{\ell=0}^{j-1}\ve^{\sigma}_{m-\ell}+ \sum\limits_{\ell=0}^{j}v_{m-\ell}+j+1\right)}\label{ClaimBeta1}\\&\qquad\times( s_{m-k}-r)^{(H_1-\frac{1}{2}-\gamma)\sum_{\ell=0}^{k-1}\ve^{\sigma}_{m-\ell}+\sum_{\ell=0}^kv_{m-\ell}+k}.\notag
		\end{align}
		It follows from the well-known formula 
		\begin{align*}
			&\int_r^{ s_{m-1}}( s_m-r)^{(H_1-\frac{1}{2}-\gamma)\ve^{\sigma}_m+v_m}(   s_{m-1}- s_{m})^{ v_{m-1}}\,\mathrm{d} s_m\\=&\frac{\Gamma\left((H_1-\frac{1}{2}-\gamma)\ve^{\sigma}_m+v_m+1\right)\Gamma( v_{m-1}+1)}{\Gamma\left((H_1-\frac{1}{2}-\gamma)\ve^{\sigma}_m+v_{m-1}+v_m+2\right)}( s_{m-1}-r)^{(H_1-\frac{1}{2}-\gamma)\ve^{\sigma}_m+v_{m-1}+v_m+1}
		\end{align*}
		that \eqref{ClaimBeta1} holds for $k=1$. Moreover, if we suppose that \eqref{ClaimBeta1} is valid for some $k=1,\ldots,m-2$, then
		\begin{align}
			&\notag\int_r^{ s_{m-k-1}}\cdots\int_{r}^{ s_{m-1}}(s_m-r)^{v_m}\prod\limits_{j=m-k}^m( s_{j}-r)^{(H_1-\frac{1}{2}-\gamma)\ve^{\sigma}_{j}}(  s_{j-1}- s_{j})^{ v_{j-1}}\\&\qquad\times\mathrm{d} s_{m}\ldots\mathrm{d} s_{m-k}\\
			&\notag=\notag\int_r^{ s_{m-k-1}}\Big(\int_r^{ s_{m-k}}\cdots\int_{r}^{ s_{m-1}}(s_m-r)^{v_m}\\&\notag\quad\times\prod\limits_{j=m-k+1}^m( s_{j}-r)^{(H_1-\frac{1}{2}-\gamma)\ve^{\sigma}_{j}} (s_{j-1}- s_{j})^{ v_{j-1}}\mathrm{d} s_{m}\ldots\mathrm{d} s_{m-k+1}\Big)\\&\notag\quad\times( s_{m-k}-r)^{(H_1-\frac{1}{2}-\gamma)\ve^{\sigma}_{m-k}}  (s_{m-k-1}- s_{m-k})^{ v_{m-k-1}}\,\mathrm{d} s_{m-k}\\=&\prod\limits_{j=1}^{k}\frac{\Gamma\left((H_1-\frac{1}{2}-\gamma)\sum\limits_{\ell=0}^{j-1}\ve^{\sigma}_{m-\ell}+ \sum\limits_{\ell=0}^{j-1}v_{m-\ell}+j\right)\Gamma\left(v_{m-k+1}+1\right)}{\Gamma\left((H_1-\frac{1}{2}-\gamma)\sum\limits_{\ell=0}^{j-1}\ve^{\sigma}_{m-\ell}+ \sum\limits_{\ell=0}^{j}v_{m-\ell}+j+1 \right)}\label{Claim1Beta2}\\&\notag\times\int_r^{ s_{m-k-1}}( s_{m-k}-r)^{(H_1-\frac{1}{2}-\gamma)\sum\limits_{\ell=0}^{k}\ve^{\sigma}_{m-\ell}+ \sum\limits_{\ell=0}^{k}v_{m-\ell}+k}\\&\notag\qquad\qquad\times( s_{m-k-1}- s_{m-k-1})^{ v_{m-k}}\mathrm{d} s_{m-k}.\notag
		\end{align}
		We deduce from \eqref{Claim1Beta2} and the identity
		\begin{align*}
			&\int_r^{ s_{m-k-1}}( s_{m-k}-r)^{(H_1-\frac{1}{2}-\gamma)\sum\limits_{\ell=0}^{k}\ve^{\sigma}_{m-\ell}+ \sum\limits_{\ell=0}^{k}v_{m-\ell}+k}( s_{m-k-1}- s_{m-k})^{ v_{m-k-1}}\\&\qquad\qquad\times\mathrm{d} s_{m-k}\\=&
			\frac{\Gamma\left((H_1-\frac{1}{2}-\gamma)\sum\limits_{\ell=0}^{k}\ve^{\sigma}_{m-\ell}+ \sum\limits_{\ell=0}^{k}v_{m-\ell}+k+1\right)\Gamma\left( v_{m-k-1}+1\right)}{\Gamma\left((H_1-\frac{1}{2}-\gamma)\sum\limits_{\ell=0}^{k}\ve^{\sigma}_{m-\ell}+ \sum\limits_{\ell=0}^{k+1}v_{m-\ell}+k+2 \right)}\\&\quad\times( s_{m-k-1}-r)^{(H_1-\frac{1}{2}-\gamma)\sum_{\ell=0}^{k}\ve^{\sigma}_{m-\ell}+\sum_{\ell=0}^{k+1}v_{m-\ell}+k+1}
		\end{align*}
		that
		\begin{align*}
			&\notag\int_r^{ s_{m-k-1}}\cdots\int_{r}^{ s_{m-1}}(s_m-r)^{v_m}\prod\limits_{j=m-k}^m( s_{j}-r)^{(H_1-\frac{1}{2}-\gamma)\ve^{\sigma}_{j}}(  s_{j-1}- s_{j})^{ v_{j-1}}\\&\qquad\qquad\times\mathrm{d} s_{m}\ldots\mathrm{d} s_{m-k}\\=&\prod\limits_{j=1}^{k+1}\frac{\Gamma\left((H_1-\frac{1}{2}-\gamma)\sum\limits_{\ell=0}^{j-1}\ve^{\sigma}_{m-\ell}+ \sum\limits_{\ell=0}^{j-1}v_{m-\ell}+j\right)\Gamma\left(v_{m-j+1}+1\right)}{\Gamma\left((H_1-\frac{1}{2}-\gamma)\sum\limits_{\ell=0}^{j-1}\ve^{\sigma}_{m-\ell}+\sum\limits_{\ell=0}^{j}v_{m-\ell}+j+1  \right)}\\&\quad\times( s_{m-k-1}-r)^{(H_1-\frac{1}{2}-\gamma)\sum_{\ell=0}^{k+1}\ve^{\sigma}_{m-\ell}+k+1},
		\end{align*}
		which means that \eqref{ClaimBeta1} is still valid for $k+1$. 
		Applying \eqref{ClaimBeta1} with $k=m-1$, we have
		\begin{align*}
			&\int_r^{ s_1}\cdots\int_{r}^{ s_{m-1}}(s_m-r)^{v_m}\prod\limits_{j=2}^m(\ s_{j}-r)^{(H_1-\frac{1}{2}-\gamma)\ve^{\sigma}_{j}}\vert  s_{j-1}-s_{j}\vert^{ v_{j-1}}\,\mathrm{d} s_{m}\ldots\mathrm{d} s_{2}\\=&\prod\limits_{j=1}^{m-1}\frac{\Gamma\left((H_1-\frac{1}{2}-\gamma)\sum\limits_{\ell=0}^{j-1}\ve^{\sigma}_{m-\ell}+\sum\limits_{\ell=0}^{j}v_{m-\ell}+j \right)\Gamma\left(1+v_{m-j+1}\right)}{\Gamma\left((H_1-\frac{1}{2}-\gamma)\sum\limits_{\ell=0}^{j-1}\ve^{\sigma}_{m-\ell}+\sum\limits_{\ell=0}^{j}v_{m-\ell}+j+1  \right)} \\&\qquad\times( s_{1}-r)^{(H_1-\frac{1}{2}-\gamma)\sum\limits_{\ell=0}^{m-2}\ve^{\sigma}_{m-\ell}+ \sum\limits_{\ell=0}^{m-1}v_{m-\ell}+m-1}.
		\end{align*}
		Hence,
		\begin{align*}
			&\int_r^s\int_r^{ s_1}\cdots\int_{r}^{ s_{m-1}}\prod\limits_{j=1}^m(\ s_{j}-r)^{(H_1-\frac{1}{2}-\gamma)\ve^{\sigma}_{j}}\vert  s_{j}-s_{j+1}\vert^{ v_j}\,\mathrm{d} s_{m}\ldots\mathrm{d} s_{1}\\=&\int_r^s\Big(\int_r^{ s_1}\cdots\int_{r}^{ s_{m-1}}(s_m-r)^{v_m}\prod\limits_{j=2}^m(\ s_{j}-r)^{(H_1-\frac{1}{2}-\gamma)\ve^{\sigma}_{j}}\vert  s_{j-1}-s_{j}\vert^{ v_{j-1}}\\&\qquad\qquad\times\mathrm{d} s_{m}\ldots\mathrm{d} s_{2}\Big)\,(s_1-r)^{(H_1-\frac{1}{2}-\gamma)\ve^{\sigma}_1}\,\mathrm{d}s_1\\=&\prod\limits_{j=1}^{m-1}\frac{\Gamma\left((H_1-\frac{1}{2}-\gamma)\sum\limits_{\ell=0}^{j-1}\ve^{\sigma}_{m-\ell}+\sum\limits_{\ell=0}^{j-1}v_{m-\ell}+j  \right)\Gamma\left(1+v_{m-j+1}\right)}{\Gamma\left((H_1-\frac{1}{2}-\gamma)\sum\limits_{\ell=0}^{j-1}\ve^{\sigma}_{m-\ell}+ \sum\limits_{\ell=0}^{j}v_{m-\ell}+j+1 \right)} \\&\qquad\times\int_r^s( s_{1}-r)^{(H_1-\frac{1}{2}-\gamma)\sum\limits_{\ell=0}^{m-1}\ve^{\sigma}_{m-\ell}+ \sum\limits_{\ell=0}^{m-1}v_{m-\ell}+m-1}\,\mathrm{d}s_1\\=&\frac{\prod\limits_{j=2}^m\Gamma\left(1+v_j\right) \prod\limits_{j=1}^{m}\Gamma\left((H_1-\frac{1}{2}-\gamma)\sum\limits_{\ell=1}^{j}\ve^{\sigma}_{m-\ell+1}+\sum\limits_{\ell=1}^{j}v_{m-\ell+1} +j\right)}{\prod\limits_{j=2}^{m}\Gamma\left((H_1-\frac{1}{2}-\gamma)\sum\limits_{\ell=1}^{j-1}\ve^{\sigma}_{m-\ell+1}+ \sum\limits_{\ell=1}^{j}v_{m-\ell+1}+j\right)} \\&\qquad\times\frac{( s_{}-r)^{(H_1-\frac{1}{2}-\gamma)\sum\limits_{\ell=0}^{m-1}\ve^{\sigma}_{m-\ell}+\sum\limits_{\ell=0}^{m-1}v_{m-\ell}+m }}{\Gamma\left((H_1-\frac{1}{2}-\gamma)\sum\limits_{\ell=0}^{m-1}\ve^{\sigma}_{m-\ell}+\sum\limits_{\ell=0}^{m-1}v_{m-\ell}+m+1 \right)},
		\end{align*}
		where in the last equality we have used the exact value of the integral and the formula $\frac{1}{\alpha}=\frac{\Gamma(\alpha)}{\Gamma(\alpha+1)}$, $\alpha>0$.
		The proof of \eqref{BetaFormula2a} is completed. The proof of \eqref{BetaFormula2b} follows analogously.
	\end{proof}
	
	\textbf{Claim 2.} Define
	\begin{align*}
		\check \Pi_{\gamma}(H_1,m):=\frac{  \prod\limits_{j=1}^{m}\Gamma\left((H_1-\frac{1}{2}-\gamma)\sum\limits_{\ell=1}^{j}\ve^{\sigma}_{m-\ell+1}+\sum\limits_{\ell=1}^{j}v_{m-\ell+1}+j \right)}{\prod\limits_{j=2}^{m}\Gamma\left((H_1-\frac{1}{2}-\gamma)\sum\limits_{\ell=1}^{j-1}\ve^{\sigma}_{m-\ell+1}+\sum\limits_{\ell=1}^{j}v_{m-\ell+1} +j\right)}
	\end{align*}
	and
	\begin{align*}
		\hat \Pi_{\gamma}(H_2,m):=\frac{  \prod\limits_{j=1}^{m}\Gamma\left((H_2-\frac{1}{2}-\gamma)\sum\limits_{\ell=1}^{j}\ve^{\pi}_{m-\ell+1}+\sum\limits_{\ell=1}^{j}v_{m-\ell+1}+j  \right)}{\prod\limits_{j=2}^{m}\Gamma\left((H_2-\frac{1}{2}-\gamma)\sum\limits_{\ell=1}^{j-1}\ve^{\pi}_{m-\ell+1}+ \sum\limits_{\ell=1}^{j}v_{m-\ell+1}+j\right)}.
	\end{align*}
	Then there exists a positive constant $C_1$ such that
	\begin{align}
		&\check \Pi_{\gamma}(H_1,m)\leq C_1^m\prod\limits_{j=1}^m\Gamma\left((H_1-\frac{1}{2}-\gamma)\ve^{\sigma}_{j}+v_{j}+1\right)\label{BetaFormula4}
	\end{align}
	and
	\begin{align}
		\hat\Pi_{\gamma}(H_2,m)\leq C_1^m\prod\limits_{j=1}^m\Gamma\left((H_2-\frac{1}{2}-\gamma)\ve^{\pi}_{j}+v_{j}+1\right).\label{BetaFormula4b}
	\end{align}
	
	\begin{proof}[Proof of Claim 2.]
		We only show \eqref{BetaFormula4} since the proof of \eqref{BetaFormula4b} follows similarly.   We first show that there exists a positive constant $C_2$ such that for any $j=2,\ldots,m$, 
		\begin{align}\label{GammaVarIneq}
			&\notag\frac{\Gamma\left((H_1-\frac{1}{2}-\gamma)\sum\limits_{\ell=1}^{j}\ve^{\sigma}_{m-\ell+1}+\sum\limits_{\ell=1}^{j}v_{m-\ell+1}+j\right)}{\Gamma\left((H_1-\frac{1}{2}-\gamma)\sum\limits_{\ell=1}^{j-1}\ve^{\sigma}_{m-\ell}+\sum\limits_{\ell=1}^{j}v_{m-\ell+1}+j\right)}\\\leq& C_2\,\Gamma\left((H_1-\frac{1}{2}-\gamma)\ve^{\sigma}_{m-j+1}+v_{m-j+1}+1\right).
		\end{align} 
		Since the function $\Gamma$ is strictly convex on $(0,\infty)$ and $\Gamma(1)=1=\Gamma(2)$, there exists $\tau\in(1,2)$ such that the function $\Gamma$ is non-increasing on $(0,\tau)$ and non-decreasing on $(\tau,\infty)$. Suppose
		\begin{align*}
			(H_1-\frac{1}{2}-\gamma)\sum\limits_{\ell=1}^{j}\ve^{\sigma}_{m-\ell+1}+\sum\limits_{\ell=1}^{j}v_{m-\ell+1}+j<\tau.
		\end{align*}
		Then we have
		\begin{align}
			&\notag\Gamma\left((H_1-\frac{1}{2}-\gamma)\sum\limits_{\ell=1}^{j}\ve^{\sigma}_{m-\ell+1}+\sum\limits_{\ell=1}^{j}v_{m-\ell+1}+j\right)\\<&\Gamma\left((H_1-\frac{1}{2}-\gamma)\ve^{\sigma}_{m-j+1}+v_{m-j+1}+1\right)\label{GammaVariation1}
		\end{align}
		since
		\begin{align*}
			(H_1-\frac{1}{2}-\gamma)\sum\limits_{\ell=1}^{j-1}\ve^{\sigma}_{m-\ell+1}+\sum\limits_{\ell=1}^{j-1}v_{m-\ell+1}+j-1>0
		\end{align*}
		and the function $\Gamma$ is non-increasing on $(0,\tau)$. We deduce from \eqref{GammaVariation1} and the inequality
		\begin{align*}
			\Gamma\left((H_1-\frac{1}{2}-\gamma)\sum\limits_{\ell=1}^{j-1}\ve^{\sigma}_{m-\ell+1}+\sum\limits_{\ell=1}^{j}v_{m-\ell+1}+j\right)>\Gamma(\tau)
		\end{align*}
		that
		\begin{align*}
			&\frac{\Gamma\left((H_1-\frac{1}{2}-\gamma)\sum\limits_{\ell=1}^{j}\ve^{\sigma}_{m-\ell+1}+\sum\limits_{\ell=1}^{j}v_{m-\ell+1}+j\right)}{\Gamma\left((H_1-\frac{1}{2}-\gamma)\sum\limits_{\ell=1}^{j-1}\ve^{\sigma}_{m-\ell+1}+\sum\limits_{\ell=1}^{j}v_{m-\ell+1}+j\right)}\\\leq& \frac{1}{\Gamma(\tau)}\Gamma\left((H_1-\frac{1}{2}-\gamma)\ve^{\sigma}_{m-j+1}+v_{m-j+1}+1\right).
		\end{align*}
		\\ Suppose on the contrary that
		\begin{align*}
			(H_1-\frac{1}{2}-\gamma)\sum\limits_{\ell=1}^{j}\ve^{\sigma}_{m-\ell+1}+\sum\limits_{\ell=1}^{j}v_{m-\ell+1}+j\geq\tau.
		\end{align*}
		Then
		\begin{align*}
			&\Gamma\left((H_1-\frac{1}{2}-\gamma)\sum\limits_{\ell=1}^{j}\ve^{\sigma}_{m-\ell+1}+\sum\limits_{\ell=1}^{j}v_{m-\ell+1}+j\right)\\\leq&\Gamma\left((H_1-\frac{1}{2}-\gamma)\sum\limits_{\ell=1}^{j}\ve^{\sigma}_{m-\ell+1}+\sum\limits_{\ell=1}^{j-1}v_{m-\ell+1}+j\right)
		\end{align*}
		since $(H_1-\frac{1}{2}-\gamma)\ve^{\sigma}_{m-j+1}<0$ and the function $\Gamma$ is non-decreasing on $(\tau,\infty)$. As a consequence, we obtain
		\begin{align*}
			&\frac{\Gamma\left((H_1-\frac{1}{2}-\gamma)\sum\limits_{\ell=1}^{j}\ve^{\sigma}_{m-\ell+1}+\sum\limits_{\ell=1}^{j}v_{m-\ell+1}+j\right)}{\Gamma\left((H_1-\frac{1}{2}-\gamma)\sum\limits_{\ell=1}^{j}\ve^{\sigma}_{m-\ell+1}+\sum\limits_{\ell=1}^{j-1}v_{m-\ell+1}+j\right)}\\\leq&1<\Gamma\left((H_1-\frac{1}{2}-\gamma)\ve^{\sigma}_{m-j+1}+v_{m-j+1}+1\right).
		\end{align*}
		The proof of \eqref{GammaVarIneq} is completed by choosing $C_2=1/\Gamma(\tau)$. 
		
		It follows from \eqref{GammaVarIneq} that
		\begin{align*}
			&\check \Pi_{\gamma}(H_1,m)\\:=&\Gamma\left((H_1-\frac{1}{2}-\gamma)\ve^{\sigma}_m+1+v_m\right)\\&\times\frac{  \prod\limits_{j=2}^{m}\Gamma\left((H_1-\frac{1}{2}-\gamma)\sum\limits_{\ell=1}^{j}\ve^{\sigma}_{m-\ell+1}+\sum\limits_{\ell=1}^{j}v_{m-\ell+1}+j\right)}{\prod\limits_{j=2}^{m}\Gamma\left((H_1-\frac{1}{2}-\gamma)\sum\limits_{\ell=1}^{j}\ve^{\sigma}_{m-\ell+1}+\sum\limits_{\ell=1}^{j-1}v_{m-\ell+1}+j\right)}\\\leq&C_2^m\prod\limits_{j=1}^m\Gamma\left((H_1-\frac{1}{2}-\gamma) \ve^{\sigma}_{m-j+1}+ v_{m-j+1}+1\right)\\=&C_2^m\prod\limits_{j=1}^m\Gamma\left((H_1-\frac{1}{2}-\gamma) \ve^{\sigma}_{j}+ v_{j}+1\right).
		\end{align*}
		Then one can choose $C_1=C_2$ to obtain \eqref{BetaFormula4}. 
		This ends the proof of Claim 2.
	\end{proof}
	{Thanks to  Claims 1 and 2, we have}
	\begin{align*}
		&\int_{\nabla^{(m)}_{r,s}}\prod\limits_{j=1}^m(s_j-r)^{(H_1-\frac{1}{2}-\gamma)\ve^{\sigma}_j}( s_j-s_{j+1})^{v_j}\,\mathrm{d}s_1\ldots\mathrm{d}s_m\\ \leq& C_1^m\frac{\prod\limits_{j=2}^{m}\Gamma\left(1+v_j\right) \prod\limits_{j=1}^{m}\Gamma\left((H_1-\frac{1}{2}-\gamma) \ve^{\sigma}_{j}+ v_j+1\right)}{\Gamma\left((H_1-\frac{1}{2}-\gamma)\sum\limits_{j=1}^{m}\ve^{\sigma}_{j}+ \sum\limits_{j=1}^{m}v_{j}+m+1\right)}\\&\times  ( s_{}-r)^{(H_1-\frac{1}{2}-\gamma)\sum\limits_{j=1}^{m}\ve^{\sigma}_{j}+ \sum\limits_{j=1}^{m}v_{j}+m}\frac{}{}\\=&
		C_1^m\frac{\prod\limits_{j=2}^{m}\Gamma\left(1+v_j\right) \prod\limits_{j=1}^{m}\Gamma\left((H_1-\frac{1}{2}-\gamma) \ve^{\sigma}_{j}+ v_j+1\right)}{\Gamma\left((H_1-\frac{1}{2}-\gamma)\sum\limits_{j=1}^{m} \ve_{j}+ \sum\limits_{j=1}^{m}v_{j}+m+1\right)}\\&\times  ( s_{}-r)^{(H_1-\frac{1}{2}-\gamma)\sum\limits_{j=1}^{m} \ve_{j}+ \sum\limits_{j=1}^{m}v_{j}+m}
	\end{align*}
	and
	\begin{align*}
		&\int_{\nabla^{(m)}_{u,t}}\prod\limits_{j=1}^m(t_j-r)^{(H_2-\frac{1}{2}-\gamma)\ve^{\pi}_j} (t_j-t_{j+1})^{w_j}\,\mathrm{d}t_1\ldots\mathrm{d}t_m\\ \leq&C_1^m\frac{\prod\limits_{j=2}^{m}\Gamma\left(1+w_j\right) \prod\limits_{j=1}^{m}\Gamma\left((H_2-\frac{1}{2}-\gamma) \ve^{\pi}_{j}+ w_j+1\right)}{\Gamma\left((H_2-\frac{1}{2}-\gamma)\sum\limits_{j=1}^{m} \ve_{j}+ \sum\limits_{j=1}^{m}w_{j}+m+1\right)}\\&\times  ( t_{}-u)^{(H_2-\frac{1}{2}-\gamma)\sum\limits_{j=1}^{m} \ve_{j}+ \sum\limits_{j=1}^{m}w_{j}+m}.
	\end{align*}
	The proof is completed by choosing $C=C^2_1C^{(2)}_{H,T}$.
\end{proof}

By similar arguments as in the proof of the preceding lemma we derive the  next two estimates.  
\begin{lem}\label{ExistEstLem3}
	Let $\esp>0$, $(H_1,H_2)\in(0,1/2)^2$, $\esp\leq u<t\leq T$, $\esp\leq r<s\leq T$, $\sigma,\,\pi\in\mathcal{P}_m$ and $(\ve_1,\ldots,\ve_m)\in\{0,1\}^m$ be fixed. Let $v_j$ and $w_j$ be real numbers satisfying
	$v_j+(H_1-\frac{1}{2})\ve_{\sigma^{-1}(j)}>-1$ and  $w_j+(H_2-\frac{1}{2})\ve_{\pi^{-1}(j)}>-1$ for all $j=1,\ldots,m$. Then there exists a finite constant $C=C_{\ve,H,T}>0$ such that
	\begin{align}
		&\notag\int_{\nabla^{m}_{r,s}}\int_{\nabla^{m}_{u,t}}\prod\limits_{j=1}^mK_H(s_{\sigma(j)},t_{\pi(j)}, r, u)(s_j-s_{j+1})^{v_j}(t_j-t_{j+1})^{w_j}\\&\notag\times\mathrm{d}t_1\ldots\mathrm{d}t_m\mathrm{d}s_1\ldots\mathrm{d}s_m\\\notag\leq& C^m_{H,T}\Pi_{0,m}(H_1,v)\Pi_{0,m}(H_2,w)(r^{H_1-\frac{1}{2}}u^{H_2-\frac{1}{2}})^{\sum_{j=1}^m\ve_j}\\& \times(s-r)^{(H_1-\frac{1}{2})\sum_{j=1}^m\ve_j+\sum_{j=1}^mv_j+m}(t-u)^{(H_2-\frac{1}{2})\sum_{j=1}^m\ve_j+\sum_{j=1}^mw_j+m},
	\end{align}
	where
	\begin{align*}
		\Pi_{0,m}(v,H_1):=\frac{\prod\limits_{j=2}^{m}\Gamma\left(1+v_j\right) \prod\limits_{j=1}^{m}\Gamma\left((H_1-\frac{1}{2}) \ve_{\sigma^{-1}(j)}+ v_j+1\right)}{\Gamma\Big((H_1-\frac{1}{2})\sum\limits_{j=1}^{m} \ve_{j}+ \sum\limits_{j=1}^{m} v_{j}+m+1\Big)}   
	\end{align*}
	and
	\begin{align*}
		\Pi_{0,m}(w,H_2):=\frac{\prod\limits_{j=2}^{m}\Gamma\left(1+w_j\right) \prod\limits_{j=1}^{m}\Gamma\left((H_2-\frac{1}{2}) \ve_{\pi^{-1}(j)}+ w_j+1\right)}{\Gamma\Big((H_2-\frac{1}{2})\sum\limits_{j=1}^{m} \ve_{j}+ \sum\limits_{j=1}^{m} w_{j}+m+1\Big)}.   
	\end{align*}
	
\end{lem}
\begin{proof}
	The proof is similar to that of the previous lemma.
\end{proof}
\begin{lem}\label{ExistEstLem4}
	Let $\ve>0$, $(H_1,H_2)\in(0,1/2)^2$, $\ve\leq\bar u< u<t\leq T$, $\ve\leq\bar r< r<s\leq T$, $\sigma,\,\pi\in\mathcal{P}_{m+\ell}$ and $(\ve_1,\ldots,\ve_m)\in\{0,1\}^m$ be fixed. Let $v_j,\,w_j\leq0$  satisfying
	$v_j+(H_1-\frac{1}{2})\ve_{\sigma^{-1}(j)}>-1$ and $w_j+(H_2-\frac{1}{2})\ve_{\pi^{-1}(j)}>-1$ for all $j=1,\ldots,m$. Then there exists a finite constant $C=C_{\ve,H,T}>0$ such that
	\begin{align}
		&\int_{\nabla^{m+\ell}_{\bar r,s}} \int_{\Delta^{m,\ell}_{\bar u,u,t}} \prod\limits_{j=1}^{m+\ell}K_H(s_{\sigma(j)},t_{\pi(j)},\bar r,\bar u) (s_j-s_{j+1})^{v_j}(t_j-t_{j+1})^{w_j} \\&\notag\quad\times\mathrm{d}t_{m+\ell}\ldots\mathrm{d}t_{m+1}\mathrm{d}t_{m}\ldots\mathrm{d}t_1\mathrm{d}s_{m+\ell}\ldots\mathrm{d}s_{1}\\\notag\leq& C^{m+\ell}_{H,T}\Pi_{0,m+\ell}(v,H_1)\widetilde\Pi_{m,\ell}(w,H_2)(\bar r^{H_1-\frac{1}{2}}\bar u^{H_2-\frac{1}{2}})^{\sum_{j=1}^{m+\ell}\ve_j}\\&\notag\times(s-\bar r)^{(H_1-\frac{1}{2})\sum_{j=1}^{m+\ell}\ve_j+\sum_{j=1}^{m+\ell}v_j+m+\ell}(t- u)^{(H_2-\frac{1}{2})\sum_{k=1}^{m}\ve_{\pi^{-1}(k)}+\sum_{k=1}^{m}w_k+m} \\&\notag\times\vert u-\bar u\vert^{(H_2-\frac{1}{2})\sum_{j=1}^{\ell}\ve_{\pi^{-1}(m+j)}+\sum_{j=1}^{\ell}w_{m+j}+\ell},
	\end{align}
	where
	\begin{align*}
		\Pi_{0,m+\ell}(v,H_1):=\frac{\prod\limits_{j=2}^{m+\ell}\Gamma\left(1+v_j\right) \prod\limits_{j=1}^{m+\ell}\Gamma\left((H_1-\frac{1}{2}) \ve_{\sigma^{-1}(j)}+ v_j+1\right)}{\Gamma\Big((H_1-\frac{1}{2})\sum\limits_{j=1}^{m+\ell} \ve_{j}+\sum_{j=1}^{m+\ell}v_j+m+\ell +1\Big) }   
	\end{align*}
	and
	\begin{align*}
		\widetilde\Pi_{m,\ell}(w,H_2):=&\frac{\prod\limits_{j=2}^{m+\ell}\Gamma\left(1+w_j\right)}{\Gamma\Big((H_2-\frac{1}{2})\sum\limits_{k=1}^{m} \ve_{\pi^{-1}(k)}+\sum\limits_{k=1}^{m}w_k+m  +1\Big)}\\&\times\frac{ \prod\limits_{j=1}^{m+\ell}\Gamma\left((H_2-\frac{1}{2}) \ve_{\pi^{-1}(j)}+w_j+1\right)}{\Gamma\Big((H_2-\frac{1}{2})\sum\limits_{j=1}^{\ell} \ve_{\pi^{-1}(m+j)}+\sum\limits_{j=1}^{\ell}w_{m+j} +\ell+1\Big)}.
	\end{align*}
\end{lem}
\begin{proof} 
	Define
	\begin{align*}
		\mathcal{I}_{m,\ell}:=&\int_{\nabla^{m+\ell}_{\bar r,s}} \int_{\Delta^{m,\ell}_{\bar u,u,t}} \prod\limits_{j=1}^{m+\ell}K_H(s_{\sigma(j)},t_{\pi(j)},\bar r,\bar u) (s_j-s_{j+1})^{v_j}(t_j-t_{j+1})^{w_j} \\&\notag\quad\times\mathrm{d}t_{m+\ell}\ldots\mathrm{d}t_{m+1}\mathrm{d}t_{m}\ldots\mathrm{d}t_1\mathrm{d}s_{m+\ell}\ldots\mathrm{d}s_{1}.
	\end{align*}
	Using \eqref{BanApp1a},
	\begin{align*}
		&\mathcal{I}_{m,\ell}\\\leq& C^{m+\ell}(\bar r^{H_1-\frac{1}{2}}\bar u^{H_2-\frac{1}{2}})^{\sum_{j=1}^{m+\ell}\ve_j}\\&\times\Big(\int_{\nabla^{m+\ell}_{\bar r,s}} \prod\limits_{j=1}^{m+\ell}(s_{j}-\bar r)^{(H_1-\frac{1}{2}-\gamma)\ve^{\sigma}_j}\vert s_j-s_{j+1}\vert^{v_j}\,\mathrm{d}s_{m+\ell}\ldots\mathrm{d}s_{m+1}\mathrm{d}s_m\ldots\mathrm{d}s_1\Big)\\&\times\Big(\int_{\Delta^{m,\ell}_{\bar u,u,t}}\prod\limits_{j=1}^{m+\ell}(t_{j}-\bar u)^{(H_2-\frac{1}{2}-\gamma)\ve^{\pi}_j}\vert t_j-t_{j+1}\vert^{w_j}\,\mathrm{d}t_{m+\ell}\ldots\mathrm{d}t_{m+1}\mathrm{d}t_m\ldots\mathrm{d}t_1\Big),
	\end{align*}
	where $s_{m+\ell+1}=\bar r$, $t_{m+\ell+1}=\bar u$, $\ve^{\sigma}_j=\ve_{\sigma^{-1}(j)}$, $\ve^{\pi}_j=\ve_{\pi^{-1}(j)}$ for $j=1,\ldots,m+\ell$.\\
	Since $w_j\leq0$, $H_2-\frac{1}{2}-\gamma<0$, $t_m-t_{m+1}\geq t_m-u$ and $t_j-\bar u\geq t_j-u$ for all $j=1,\ldots,m$, we have
	\begin{align*}
		&	\int_{\Delta^{m,\ell}_{\bar u,u,t}}\prod\limits_{j=1}^{m+\ell}(t_{j}-\bar u)^{(H_2-\frac{1}{2}-\gamma)\ve^{\pi}_j}\vert t_j-t_{j+1}\vert^{w_j}\,\mathrm{d}t_{m+\ell}\ldots\mathrm{d}t_{m+1}\mathrm{d}t_m\ldots\mathrm{d}t_1\\
		=&\int_{\nabla^{m}_{u,t}}\int_{\nabla^{\ell}_{\bar u,u}}\prod\limits_{j=1}^{m+\ell}(t_{j}-\bar u)^{(H_2-\frac{1}{2}-\gamma)\ve^{\pi}_j}\vert t_j-t_{j+1}\vert^{w_j}\,\mathrm{d}t_{m+\ell}\ldots\mathrm{d}t_{m+1}\mathrm{d}t_m\ldots\mathrm{d}t_1\\=&\int_{\nabla^{m}_{u,t}}(t_m-\bar u)^{(H_2-\frac{1}{2}-\gamma)\ve^{\pi}_m}\prod\limits_{k=1}^{m-1}(t_{k}-\bar u)^{(H_2-\frac{1}{2}-\gamma)\ve^{\pi}_k}\vert t_k-t_{k+1}\vert^{w_k}\\&\times\Big(\int_{\nabla^{\ell}_{\bar u,u}}(t_m-t_{m+1})^{w_m}\prod\limits_{j={m+1}}^{m+\ell}(t_{j}-\bar u)^{(H_2-\frac{1}{2}-\gamma)\ve^{\pi}_j}\vert t_j-t_{j+1}\vert^{w_j}\\&\times\mathrm{d}t_{m+\ell}\ldots\mathrm{d}t_{m+1}\Big)\mathrm{d}t_m\ldots\mathrm{d}t_1\\\leq&\int_{\nabla^{m}_{u,t}}(t_m- u)^{(H_2-\frac{1}{2}-\gamma)\ve^{\pi}_m}\prod\limits_{k=1}^{m-1}(t_{k}- u)^{(H_2-\frac{1}{2}-\gamma)\ve^{\pi}_k}\vert t_k-t_{k+1}\vert^{w_k}\\&\times\Big(\int_{\nabla^{\ell}_{\bar u,u}}(t_m-u)^{w_m}\prod\limits_{j={m+1}}^{m+\ell}(t_{j}-\bar u)^{(H_2-\frac{1}{2}-\gamma)\ve^{\pi}_j}\vert t_j-t_{j+1}\vert^{w_j}\\&\times\mathrm{d}t_{m+\ell}\ldots\mathrm{d}t_{m+1}\Big)\mathrm{d}t_m\ldots\mathrm{d}t_1\\=&
		\Big(\int_{\nabla^{m}_{u,t}} \prod\limits_{k=1}^{m}(\check t_{k}- u)^{(H_2-\frac{1}{2}-\gamma) \ve_{\pi^{-1}(k)}}\vert \check t_k-\check t_{k+1}\vert^{w_k}\mathrm{d}\check t_m\ldots\mathrm{d}\check t_1\Big)\\&\quad\times\Big(\int_{\nabla^{\ell}_{\bar u,u}} \prod\limits_{j=1}^{\ell}(\hat t_{j}- \bar u)^{(H_2-\frac{1}{2}-\gamma) \ve_{\pi^{-1}(m+j)}}\vert \hat t_j-\hat t_{j+1}\vert^{w_{m+j}}\mathrm{d}\hat t_{\ell}\ldots\mathrm{d}\hat t_1\Big),
	\end{align*}
	where $\check t_k=t_k$ for $k=1,\ldots,m$ and $\check t_{m+1}=u$, $\hat t_j=t_{m+j}$ for $j=1,\ldots,\ell$ and $\hat t_{\ell+1}=\bar u$. Then the desired estimates follows from Claims 1 and 2 in the proof of Lemma 22. 
\end{proof}

\end{document}